\documentclass{amsart}

\RequirePackage{amspaperstyle}

\newif\ifsidenotes
\sidenotestrue

\ifsidenotes
\newcommand{\menote}[1]{\marginpar{\color{blue}\tiny [MLE] #1}}
\newcommand{\meadd}[1]{{\color{blue}{\tiny [MLE]} #1}}
\newcommand{\awnote}[1]{\marginpar{\color{cyan}\tiny [ALW] #1}}
\newcommand{\awadd}[1]{{\color{cyan}{\tiny [ALW]} #1}}
\newcommand{\manote}[1]{\marginpar{\color{magenta}\tiny [MHAG] #1}}
\newcommand{\maadd}[1]{{\color{magenta}{\tiny [MHAG]} #1}}
\else
\newcommand{\menote}[1]{}
\newcommand{\meadd}[1]{}
\newcommand{\awnote}[1]{}
\newcommand{\awadd}[1]{}
\newcommand{\manote}[1]{}
\newcommand{\maadd}[1]{}
\fi

\usepackage{comment}
\newcommand{\torus}{\mathbb{T}}
\newcommand{\dual}[1]{#1^{\scriptscriptstyle \vee}}
\newcommand{\fform}[2]{\langle\langle #1,#2 \rangle\rangle}
\newcommand{\glue}{\mathcal{G}}
\renewcommand{\subset}{\subseteq}

\newif\ifthesis
\thesisfalse
\ifthesis
	\newcommand{\showthesis}[1]{#1}
\else
	\newcommand{\showthesis}[1]{}
\fi

\newif\ifsquarefree
\squarefreefalse

\renewcommand{\Gr}{\operatorname{Gr}_{2,4}}
\newcommand{\Grnd}{\Gr^{\mathrm{nd}}}
\newcommand{\purewedges}{\mathbf{W}}
\newcommand{\purend}{\mathbf{W}^{\mathrm{nd}}}
\newcommand{\pureprim}{\mathbf{W}_{\mathrm{prim}}}
\newcommand{\homog}{\phi}
\newcommand{\varwedge}{W}
\newcommand{\varprim}{\pureprim}

\newcommand{\kleinwedge}{\mathbf{K}'}
\newcommand{\kleinnondeg}{\mathbf{K}'_{\neq 0}}
\newcommand{\kleinmapwedge}{\Phi'}
\newcommand{\packet}{\mathcal{Y}}

\begin{document}

\title{Planes in four space and four associated CM points}

\author{Menny Aka}
\author{Manfred Einsiedler}
\author{Andreas Wieser}
%
\address{Departement Mathematik, ETH Z{\"u}rich, R{\"a}mistrasse 101, 8092 Z{\"u}rich, Switzerland}
\thanks{The authors were supported by the SNF (grants 152819 and 178958).}
%
%
%
\date{\today}
%
%
%
\begin{abstract}
To any two-dimensional rational plane in four-dimensional space one can naturally attach a point in the Grassmannian $\operatorname{Gr}(2,4)$ and four shapes of lattices of rank two.
Here, the first two lattices originate from the plane and its orthogonal complement and the second two essentially arise from the accidental local isomorphism between $\operatorname{SO}(4)$ and $\operatorname{SO}(3)\times \SO(3)$. 
As an application of a recent result of Einsiedler and Lindenstrauss on algebraicity of joinings we prove simultaneous equidistribution of all of these objects under two splitting conditions.
\end{abstract}
\maketitle

\section{Introduction}

For a rational two-dimensional subspace $L$ of $\R^4$ we define the discriminant of~$L$ as the square of the covolume of $L \cap \Z^4$ in $L$, i.e.
\begin{align*}
\disc(L) = \vol (L/L\cap \Z^4)^2 \in \N.
\end{align*}
For any $D \in \N$ we let $\mathcal{R}_D$ be the finite set of rational planes of discriminant $D$, which is a subset of the real Grassmannian $\Gr(\R)$.
The set $\mathcal{R}_D$ is non-empty if and only~if
\begin{align*}
D\in \mathbb{D}:=\setc{D\in \N}{D\not\equiv 0,7,12,15\mod 16}
\end{align*}
(see Corollary~\ref{cor: congruence condition}).
This statement should be seen as an analogue of Legendre's theorem for sums of three squares and relates to works of Mordell \cite{mordell1,mordell2} and Ko \cite{ko} on representations of binary forms as sums of four squares.

We let $\mathbb{H}^2$ denote the hyperbolic plane and call the quotient $Y_0(1) = \lquot{\SL_2(\Z)}{\mathbb{H}^2}$ the modular surface.
Note that
\begin{align*}
\mathcal{X}_2:= \lrquot{\PGL_2(\Z)}{\PGL_2(\R)}{\PO(2)}
\end{align*}
is a two-to-one quotient of the modular surface obtained by using the orientation reversing
reflection $z\in\mathbb{H}^2\mapsto -\overline{z}$ through the imaginary axis.

To each $L\in \mathcal{R}_D$ we will naturally attach a four-tuple $(z_1^L,z_2^L,z_3^L,z_4^L)$ of CM-points on $\mathcal{X}_2$.
We conjecture that the set
\begin{align*}
\mathcal{J}_D = \setc{(L, z_1^L,z_2^L,z_3^L,z_4^L)}{L\in \mathcal{R}_D} 
\end{align*}
is equidistributing to the natural uniform measure on the product space $$\Gr(\R)\times \mathcal{X}^4_2$$ when $D\to\infty$ with $D\in \mathbb{D}$. 
In this paper we prove this conjecture under additional congruence conditions. In particular, our result implies the conjecture on average.

Let us now describe the points $z_i^L$ for $i=1,2,3,4$ in more details. 
First, consider for any rational plane $L$ the two-dimensional lattices
\begin{align*}
L(\Z):=L\cap \Z^4,\quad L^\perp(\Z):=L^{\perp}\cap \Z^4.
\end{align*}
In order to compare these lattices for different planes, we now choose a rotation $k_L\in \SO_4(\R)$ moving $L$ to $\R^2\times \set{(0,0)}\subset \R^4$ and $L^\perp$ to $\set{(0,0)} \times \R^2$.
The shape $[L(\Z)]$ (resp.~$[L^\perp(\Z)]$) is then defined to be the homothety class of the lattice $k_L.L(\Z)\subset \R^2\times \set{(0,0)}$ (resp.~$k_L.L^\perp(\Z)\subset \set{(0,0)} \times \R^2$) and is as such a well-defined element of~$\mathcal{X}_2$ (as it is independent of the choice of rotation $k_L$).
These are the points $z_1^L,z_2^L$ from above. 
Indeed, $[L(\Z)]$ (resp.~$[L^\perp(\Z)]$) may be thought of as the equivalence class of the integral positive definite binary quadratic form which is the restriction of the form $x_0^2+x_1^2+x_2^2+x_3^2$ to $L(\Z)$ (resp.~$L^\perp(\Z)$). 
As such they correspond to CM-points.
Due to the geometric construction we will refer to them as the \textbf{geometric CM points} attached to $L$.

The points $z_3^L$ and $z_4^L$ come from a natural identification of the Grassmannian $\Gr(\R)$ with the space
\begin{align*}
\kleinsp (\R) = \setc{(a_1,a_2)}{a_1,a_2\in \R^3\setminus\set{0},\ \norm{a_1}_2=\norm{a_2}_2}/\sim  
\end{align*}
where $(a_1,a_2)\sim(a_1',a_2')$ if there is $\lambda\in\R^\times$ with $(a_1,a_2)=(\lambda a_1' , \lambda a_2')$. 
Note that $\kleinsp(\R)$ is a two-to-one quotient of a product of two spheres.
 
To describe the above identification, it is useful to view $\R^4$ as the algebra of real Hamiltonian quaternions  
\begin{align*}
\quat(\R)=
\setc{x_0 + x_1 \ii + x_2 \jj + x_3 \kk}{x_0,x_1,x_2,x_3\in \R},
\end{align*}
which is equipped with a conjugation given by
$\overline x= x_0 - x_1 \mathrm{i} - x_2 \mathrm{j} - x_3 \mathrm{k}$ and a trace given by $\Tr(x)=x+\overline x=2x_0$ for $x=x_0 + x_1 \ii + x_2 \jj + x_3 \kk \in \quat(\R)$.
Furthermore, we identify $\R^3$ with the traceless quaternions, i.e.~quaternions of the form $x_1 \mathrm{i} + x_2 \mathrm{j} + x_3 \mathrm{k}$ for $x_1,x_2,x_3 \in \R$. 
Then,
\begin{align}\label{eq:introdef-kleinmap}
L \in \Gr(\R) \mapsto [(a_1(L),a_2(L))] \in \kleinsp(\R)
\end{align}
is a diffeomorphism where $a_1(L),a_2(L)\in \R^3$ for $L=\R v_1\oplus\R v_2$ are given by
 \begin{align}\label{eq:associated intpts}
\begin{array}{l}
a_1(L) := v_1\overline{v_2} - \tfrac{1}{2}\Tr(v_1\overline{v_2}),\\
a_2(L) := \overline{v_2}v_1 - \tfrac{1}{2}\Tr(\overline{v_2}v_1).
\end{array} 
\end{align}
We choose to call the map in \eqref{eq:introdef-kleinmap} the \textbf{Klein map}; the name is inspired by the name 'Klein quadric' for the projective variety obtained from $\Gr(\R)$ by applying the Pl\"ucker embedding.
The points $a_1(L),a_2(L)$ as defined here depend on the choice of basis but the equivalence class $[(a_1(L),a_2(L))]$ does not.
Furthermore, if $L$ is rational and $v_1, v_2$ are a $\Z$-basis of $L(\Z)$ the resulting vectors $a_1(L),a_2(L)$ are integer vectors.
As we will explain later, this yields that the subset $\mathcal{R}_D$ is (almost) in bijection with a set of points $[(a_1,a_2)]$ where $a_1,a_2\in \Z^3$ are vectors of length $\norm{a_1}_2=\norm{a_2}_2=\sqrt D$. 

\begin{remark}
The Klein map arises naturally through the accidental local isomorphism between the groups $\SO_4(\R)$ and $\SU_2(\R)\times \SU_2(\R)$ (and hence also $\SO_3(\R)\times\SO_3(\R)$) as follows.
If $L \in \mathcal{R}_D$ and $\BH_L(\R) < \SU_2^2(\R)$ is the connected component of the stabilizer subgroup of $L$ under the action of $\SU_2^2(\R)$ on $\R^4$ (essentially through the local isomorphism), then the projection of $\BH_L(\R)$ to either of the factors of $\SU_2^2(\R)$ must be a one-dimensional torus. Any such one-dimensional torus is the stabilizer of a point in the adjoint representation of $\SU_2(\R)$ i.e.~the isogeny $\SU_2(\R) \to \SO_3(\R)$.
The two points we obtain in this way are exactly $a_1(L)$ and $a_2(L)$ (up to multiples as described above) -- see Proposition~\ref{prop:splitting map}.
It is also worthwhile noting that the above procedure is how the explicit formulas for the Klein map were found.
\end{remark}

The identification between $\Gr(\R)$ and $\kleinsp(\R)$ in \eqref{eq:introdef-kleinmap} leads to the construction of the following two lattices.
For $i=1,2$ we define $\Lambda_{a_i(L)}=a_i(L)^\perp\cap\Z^3$. Fixing a copy of $\R^2$ in $\R^3$ we can rotate these two-dimensional lattices to it to obtain the shapes $z_3^L:=[\Lambda_{a_1(L)}]$ and $z_4^L:=[\Lambda_{a_2(L)}]$, two  well-defined points on $\mathcal{X}_2$. 
Again, these also correspond to CM-points 
when viewed as the class of the binary form obtained by restriction of the ambient form $x_1^2+x_2^2+x_3^2$. We will refer to $[\Lambda_{a_1(L)}]$ and $[\Lambda_{a_2(L)}]$ as the \textbf{accidental CM points} attached to $L$.

\begin{conjecture}\label{mainconjecture}
 The normalized counting measure on the finite set
\begin{align*}
\mathcal{J}_D = \left\lbrace \big(L,[L(\Z)], [L^\perp(\Z)],[\Lambda_{a_1(L)}], [\Lambda_{a_2(L)}]\big): L \in \mathcal{R}_D \right\rbrace
\subseteq \Gr(\R) \times \mathcal{X}_2^4
\end{align*}
equidistributes to the uniform probability measure on $\Gr(\R) \times \mathcal{X}_2^4$ as $D \to \infty$ with $D\in \mathbb{D}$. That is,
\begin{align*}
\tfrac{1}{|\mathcal{J}_D|} \sum_{x \in \mathcal{J}_D} \delta_x \to 
m_{\Gr(\R) \times \mathcal{X}_2^4}
\end{align*}
in the weak${}^\ast$-topology where $m_{\Gr(\R) \times \mathcal{X}_2^4}$ is the probability measure obtained from an $\SO(4)$-invariant measure on $\Gr(\R)$ and an $\SL_2(\R)$-invariant measure on $\mathbb{H}^2$.
\end{conjecture}

Our main theorem verifies this conjecture under extra congruence conditions:

\ifsquarefree 

\begin{theorem}[Equidistribution for a given discriminant]\label{thm:main}
Let $p,q$ be any two distinct odd primes. The normalized counting measure on the finite set $\mathcal{J}_D$
equidistributes to the uniform probability measure on $\Gr(\R) \times \mathcal{X}_2^4$ as $D \to \infty$ with $D\in \mathbb{D}$ also satisfying the following conditions: 
\begin{enumerate}[(i)]
\item $-D \in (\Fp^\times)^2,\ -D \in (\BF^\times_q)^2$ 
\item $D$ is square-free.
\end{enumerate}
\end{theorem}

\else 

\begin{theorem}[Equidistribution for a given discriminant]\label{thm:main}
Let $p,q$ be any two distinct odd primes. 
The normalized counting measure on the finite set $\mathcal{J}_D$
equidistributes to the uniform probability measure on $\Gr(\R) \times \mathcal{X}_2^4$ as $D\in\BD$ goes to infinity while $D$ satisfies the additional condition that $-D$ is a non-zero square modulo $p$ and modulo $q$.
\end{theorem}

\fi

First results in the spirit of this theorem have previously been obtained by Maass \cite{Maass56}, \cite{Maass59} and Schmidt \cite{Schmidt-shapes}, who establish the averaged equidistribution of the pairs $(L,[L(\Z)])$, where $L$ varies over the rational planes of discriminant up to $D$. 
Recently, Horesh and Karasik \cite{HoreshKarasik} have obtained equidistribution of the tuples $(L,[L(\Z)],[L^\perp(\Z)])$ in this averaged setup.
All of these results are polynomially effective in $D$. In an upcoming preprint, the authors prove together with Luethi and Michel \cite{EffDisjoint} effective equidistribution of the tuples $(L,[L(\Z)],[L^\perp(\Z)])$ for $L \in \mathcal{R}_D$ using a different method.

\ifsquarefree\else Here, we obtain as a corollary of the above theorem the following ineffective strengthening of these results:

\begin{corollary}[Averaged equidistribution]\label{cor:averaged}
The normalized counting measure on the finite set
\begin{align*}
\left\lbrace \big(L,[L(\Z)], [L^\perp(\Z)],[\Lambda_{a_1(L)}], [\Lambda_{a_2(L)}]\big): L \in \mathcal{R}_d \text{ for some } d \leq D \right\rbrace
\end{align*}
equidistributes to the uniform probability measure on $\Gr(\R) \times \mathcal{X}_2^4$ as $D \to \infty$.
\end{corollary}
\fi

First non-averaged results as in Theorem \ref{thm:main} have been established by Linnik \cite{linnik} and Skubenko \cite{skubenko} (with a congruence condition at one prime) and Duke \cite{duke88} (building on work of Iwaniec \cite{iwaniecforduke} and without any congruence condition).
Duke's theorem shows equidistribution of integer points on two-dimensional spheres and of CM-points on $\mathcal{X}_2$.

We use these results to obtain the equidistribution on the individual factors of our space.
Using a theorem \cite[Thm.~1.4]{EL-joining2} of the second named author with Lindenstrauss we then upgrade this information to joint equidistribution.
This method of proof has already been used in the work \cite{AES3D} of the first and the second named author with Shapira.
Note that the proof of Theorem~\ref{thm:main} (much like the main result in \cite{AES3D}) can currently not be adapted to provide an error rate as in \cite{Maass56,Maass59,Schmidt-count,Schmidt-shapes,HoreshKarasik,EffDisjoint}.

\begin{remark}
Our techniques also yield a version of Theorem~\ref{thm:main} for \emph{oriented} subspaces, see Theorem~\ref{thm:oriented} for the statement.
The authors find Theorem~\ref{thm:main} to be the geometrically more appealing version of these theorems which is why it is presented as the main result, though Theorems \ref{thm:main} and \ref{thm:oriented} are (certainly from the dynamical viewpoint) morally equivalent. Given an oriented rational subspace, one can associated to it shapes in the modular curve $Y_0(1)$ which are (as opposed to the above points) truly complex multiplication points.
\end{remark}

\begin{remark}
Our main motivation for the study of the sets $\mathcal{J}_D$ has been geometric.
As it turned out, the construction of the CM-points $z_1^L, z_2^L, z_3^L, z_4^L$ has a rather surprising relation to the group law of the Picard group of an order in $\Q(\sqrt{-D})$. 
In this way our equidistribution result also has an arithmetic
counterpart, see Theorem \ref{thm:arithmeticversion}. To describe a slightly simpler but analogous result, write $\mathbf{z}_{[\Fa]} \in Y_0(1)$ for the CM point associated to a proper ideal class $\Fa$ of a quadratic order $\mathcal{O}$. Then sets of tuples
\begin{align*}
\big(\mathbf{z}_{[\Fa]},\mathbf{z}_{[\Fb]},\mathbf{z}_{[\Fa\Fb]}, \mathbf{z}_{[\Fa\Fb^{-1}]}, \mathbf{z}_{[\Fa^2]}, \mathbf{z}_{[\Fb^2]}\big)
\end{align*}
where $[\Fa],[\Fb]$ vary over proper ideal classes of $\mathcal{O}$ are equidistributed in $Y_0(1)^6$ when the discriminant of $\mathcal{O}$ goes to infinity (given two splitting conditions).
There are no clear dependencies between the coordinates of these tuples (in the ambient space 
$Y_0(1)$) which is a good reason to suspect this equidistribution statement. The same cannot be said about Theorem~\ref{thm:main} a priori. 
\end{remark}

\begin{remark}[Glue groups]
We would also like to remark that the proof of our theorem can be used to strengthen our result. In fact, Theorem \ref{thm:main} can be formulated to consider only planes $L\in \mathcal{R}_D$, whose glue group  is of a fixed isomorphism type (see Theorem~\ref{thm:equi along isotypes}). 
Here, the glue group of a lattice $L$ is an invariant which captures additional information on the lattice including the discriminant (see e.g.~McMullen \cite{McMullenglue} or Section \ref{section:glue groups} for definitions).
\end{remark}

\begin{remark}[Extensions]\label{rem:extensions}
Theorem~\ref{thm:main} can be generalized to arbitrary quaternion algebras (or rather norm forms on quaternion algebras).
In \cite{2in4General}, the first and last named authors together with Horace Chaix generalize the above results to arbitrary quadratic forms using the language of Clifford algebras.

Higher-dimensional analogues of the above results have been covered in joint work of first and last author with Andrea Musso \cite{Grassmannianhigherdim}. 
The result there does not involve any accidental shapes, but is stronger in a sense comparable to \cite{AEShigherdim}, \cite{KhayutinGrids}.
\end{remark}

\subsection{Outline of the paper}

This paper is organized as follows:
\begin{itemize}
\item In Section \ref{section:Klein map} we study properties of the Klein map (including the statements made above). This yields important information about the related stabilizer groups which play a crucial role in our dynamical argument.
\item In Section \ref{section:CMpts} we discuss in more detail the four CM points attachted to each plane.
\item In Section \ref{section:formulation} we define the joint acting group for the dynamical setup and formulate a dynamical version (Theorem \ref{thm:orbit version}) of Theorem \ref{thm:main}.
\item In Section \ref{section:prooffromdynstatement} we use the orbit of the stabilizer in order to generate additional points starting from one point (Propositions \ref{prop:generating intpts} and \ref{prop:collection intpts}) and apply this to prove Theorem \ref{thm:main} assuming Theorem \ref{thm:orbit version}.
\item In Section \ref{section:proofdynresult} we use the fact that any limit measure coming from Theorem~\ref{thm:orbit version} is a joining for a higher rank torus action and the algebraicity of such joinings \cite{EL-joining2} to deduce the theorem.
\item In Sections \ref{section:furthercomments} and \ref{section:glue groups} we explain further connections to existing work and in particular the above mentioned connection to the Picard group and to glue groups.
\item
In Appendix \ref{section:proof of averaged version} we prove the averaged version (Corollary~\ref{cor:averaged}) of the main theorem.
In Appendix \ref{section:appendixA} we study the Klein map in the case of the split quaternion algebra $\Mat_2$.
\end{itemize}

\textbf{Acknowledgements:}
We would like to thank Elon Lindenstrauss, Ilya Khayutin and Philippe Michel for discussions on this paper.
We are also very grateful to Curtis McMullen for suggesting a refinement of our theorem with respect to glue groups.
Last but not least, we would like to thank the referees for many useful comments, for their very careful reading of this article, and for making us realize that we need to emphasize the naturality of the Klein map more.

\section{The Klein map}\label{section:Klein map}

In the following discussions we let $\BK$ denote a field of characteristic zero.

\subsection{Hamiltonian quaternions}\label{section:quaternions}

As in the introduction, we denote by $\quat$ the variety defined over $\Q$ representing the $\Q$-algebra of Hamiltonian quaternions.
We let $\overline{\cdot}: \quat \to \quat$ be the canonical involution (henceforth also called conjugation), $\Tr: \quat \to \G_a$ the (reduced) trace, and $\Nr: \quat \to \G_a$ the (reduced) norm form.
For any $x \in \quat(\Q)$, $\Tr(x) = x + \overline{x}$ and
 \begin{align*}
 \Nr(x) = x \overline{x} = \overline{x}x = x_0^2+x_1^2+x_2^2+x_3^2 =: Q(x_0,x_1,x_2,x_3)
 \end{align*}
writing $x = x_0 + x_1 \ii + x_2 \jj + x_3 \kk$. As already mentioned we will identify $\quat$ with the four-dimensional affine space and in particular write $\Nr = Q$.
A quaternion with zero trace is said to be pure and will often be viewed as a point in three-dimensional space. 
The variety representing pure quaternions will be denoted by $\quat_0$.

We equip $\quat$ with the integral structure arising from the choice of basis $1,\ii,\jj,\kk$: we set 
\begin{align*}
\quat(\Z) = \Z + \Z\ii + \Z \jj + \Z\kk
\end{align*}
and define $\quat(\Z_p) = \quat(\Z) \otimes \Z_p$ for any prime $p$.
Under the above identification $\quat(\Q) \simeq \Q^4$ we have $\quat(\Z) \simeq \Z^4$.
Note that $\quat(\Z)$ is a non-maximal order in $\quat(\Q)$; it is contained with index two in the order of Hurwitz quaternions which is generated by $\ii$, $\jj$, $\kk$, and $\frac{1+\ii+\jj+\kk}{2}$.
Furthermore, we let $\quat_0(\Z) = \quat_0(\Q) \cap \quat(\Z) = \Z \ii + \Z\jj + \Z\kk$ and similarly for $\quat_0(\Z_p)$.

Denote by
\begin{align*}
\SU_2 = \mathrm{ker}(\Nr: \quat^\times \to \G_m)
\end{align*}
the algebraic $\Q$-group of norm one elements of $\quat$ and observe that the action of $\SU_2^2$ on $\quat$ given by $(\alpha,\beta).x = \alpha x \beta^{-1}$ preserves the norm.
This yields a $\Q$-isogeny $P:\SU_2^2 \to \SO_4$ (and in particular a local isomorphism of the groups $\SO_4(\R)$ and $\SU_2^2(\R)$), where the kernel is given by $\set{(1,1),(-1,-1)}$.
In what follows, we shall always mean by $\SO_n$ the special orthogonal group for the sum of $n$ squares. 

The action of $\SU_2$ on pure quaternions by conjugation (i.e.~by $\alpha.x = \alpha x \alpha^{-1}$) gives an isogeny $\SU_2 \to \SO_3$ with kernel $\set{1,-1}$.
Identifiying $\quat_0$ with the Lie algebra of $\SU_2$, this is the adjoint representation.
We define $P_1,P_2:\SU_2^2 \to \SO_3$ to be the composition of this isogeny with the respective coordinate projections $\SU_2^2\to\SU_2$.

\begin{remark}
There are various integral structures on $\SU_2^2(\Q)$ which would be naturally in the context of this article. Set for the purposes of this remark
\begin{align*}
\Gamma_1 = \quat(\Z)^\times \times \quat(\Z)^\times,\
\Gamma_2 = P^{-1}(\SO_4(\Z)),\
\Gamma_3 = P_1^{-1}(\SO_3(\Z)) \cap P_2^{-1}(\SO_3(\Z)).
\end{align*}
These finite groups satisfy $\Gamma_1 \subset \Gamma_2 \subset \Gamma_3$ as well as
\begin{align*}
\Gamma_2 = \{\gamma = (\gamma_1,\gamma_2) \in \Gamma_3:  \gamma_1\gamma_2^{-1} \in \quat(\Z)^\times\}.
\end{align*}
The group $\Gamma_3$ is the product of the group of units in the (maximal) order of Hurwitz quaternions with itself. In particular, $[\Gamma_3:\Gamma_2] = [\Gamma_2: \Gamma_1] = 3$.
\end{remark}

We write $\Gr$ for the projective Grassmannian variety of two-dimensional subspaces in four-space. 
In particular, the set $\Gr(\BK)$ can be identified with the set of two-dimensional $\BK$-subspaces in $\BK^4$ (i.e.~planes containing zero).
Given a subspace $L \in \Gr(\BK)$ and a $\BK$-algebra $A$ we will write $L(A) = L\otimes_\BK A$.

Under the identification of $\quat$ with the four-dimensional affine space, the action of $\SU_2^2$ on $\quat$ induces naturally an action on $\Gr$.
The action is not transitive but there is an open and dense orbit, namely the Zariski open set of non-degenerate subspaces $\Grnd$.
Here, a subspace $L$ is non-degenerate if $L \cap L^\perp = \{0\}$.
As $\Nr$ is positive definite on $\quat(\R)$, all rational subspaces are contained in $\Grnd(\Q)$.

For any $L \in \Grnd(\BK)$ we set $\BH_L$ to be the Zariski connected component of the stabilizer subgroup of $L$. 
In particular, for any $\BK$-algebra $A$
\begin{align*}
\BH_L(A) \subset \{g \in \SU_2^2(A): g.L(A) = L(A)\}.
\end{align*}
The restriction of $\BH_L$ to $L$ resp.~$L^\perp$ yields an isogeny $\BH_L \to \SO_{\Nr|_L} \times \SO_{\Nr|_{L^\perp}}$ as $\BH_L$ is connected. 
In particular, $\BH_L$ is a two-dimensional algebraic torus defined over $\BK$.
Furthermore, for any $v\in \BK^3 = \quat_0(\BK)$ we define the one-dimensional $\BK$-torus $\BH_v< \SU_2$ to be the stabilizer subgroup of $v$.
Thus, for any $\BK$-algebra $A$
\begin{align*}
\BH_v(A) = \{g \in \SU_2(A): g.v = v \}.
\end{align*}

\subsection{Definition of the Klein map}

We define 
\begin{align*}
\kleinsp (\BK) =  \setc{(a_1,a_2)}{a_1,a_2\in \BK^3\setminus\set{0},\ Q(a_1) = Q(a_2)}/\sim
\end{align*}
where $(a_1,a_2)\sim(a_1',a_2')$ if there exists $\lambda \in\BK^\times$ with $(a_1,a_2)=(\lambda a_1',\lambda a_2')$. 
Also, observe that $\SU_2^2(\BK)$ acts on $\kleinsp (\BK)$ via $(g_1,g_2).[a_1,a_2] = [g_1.a_1,g_2.a_2]$ where $[a_1,a_2]$ denotes the equivalence class of $(a_1,a_2)$ in $\kleinsp(\BK)$.
Note that $\kleinsp(\BK)$ is the set of $\BK$-points of a quasi-projective variety $\kleinsp$ defined over $\Q$.
In the following we will use the definition of $a_1(\cdot)$ and $a_2(\cdot)$ from Equation \eqref{eq:associated intpts}.

\begin{proposition}[Klein map]\label{prop:splitting map}
The map
\begin{align*}
\Phi: L \in \Grnd(\BK) \mapsto [a_1(L),a_2(L)] \in \kleinsp (\BK)
\end{align*}
is a well-defined bijection and is equivariant for the actions of $\SU_2^2(\BK)$. The inverse of $\Phi$ is given by\footnote{
The fact that \eqref{eq:formulainverse} defines a subspace is not trivial and uses the crucial information that $a_1(L),a_2(L)$ are pure and of the same length. We note that the formula has appeared in a slightly different context in \cite{Linnikthmexpander} in order to realize the action of the class group on integer points on the $2$-dimensional sphere.
}
\begin{align}\label{eq:formulainverse}
\Phi^{-1}([a_1,a_2]) = \setc{x\in \quat(\BK)}{a_1 x=x a_2}
\end{align}
for all $[a_1,a_2]\in \kleinsp (\BK)$. Furthermore, we have the following properties:
\begin{itemize}
\item $\mathbb{H}_L = \mathbb{H}_{a_1(L)} \times \mathbb{H}_{a_2(L)}$ for any $L \in \Grnd(\BK)$.
\item $\Phi(L^\perp) = [a_1(L),-a_2(L)]$ for any $L \in \Grnd(\BK)$ where $L^\perp$ denotes the orthogonal complement with respect to $Q$.
\end{itemize}
\end{proposition}

\begin{proof}
To see that $\Phi$ is well-defined, observe first that for any $L\in \Grnd(\BK)$ with basis $v_1,v_2$ and for $a_1(L),a_2(L)$ defined as in \eqref{eq:associated intpts} using $v_1,v_2$, we have
\begin{align*}
Q(a_1(L)) 
= a_1(L) \overline{a_1(L)} 
&= (v_1\overline{v_2} - \tfrac{1}{2}\Tr(v_1\overline{v_2}))
(v_2\overline{v_1} - \tfrac{1}{2}\Tr(v_1\overline{v_2}))\\ 
&= Q(v_1)Q(v_2) + \tfrac{1}{4}\Tr(v_1\overline{v_2})^2 - \tfrac{1}{2}\Tr(v_1\overline{v_2}) (v_2\overline{v_1}+v_1 \overline{v_2})\\
&= Q(v_1)Q(v_2) - \tfrac{1}{4}\Tr(v_1\overline{v_2})^2
\end{align*}
and analogously
\begin{align}\label{eq:same length}
Q(a_2(L)) = Q(v_1)Q(v_2) - \tfrac{1}{4}\Tr(\overline{v_1}v_2)^2 = Q(a_1(L)). 
\end{align}
Furthermore, as the maps 
\begin{align}\label{eq:bilin&antisymm}
\begin{array}{l}
(u,v) \mapsto u\overline{v}-\tfrac{1}{2}\Tr(u\overline{v}), \\
(u,v) \mapsto \overline{v}u-\tfrac{1}{2}\Tr(\overline{v}u)
\end{array}
\end{align}
(that are used to define $a_1(L),a_2(L)$ in \eqref{eq:associated intpts})
are bilinear and antisymmetric, it follows that $[a_1(L),a_2(L)]$ does not depend on the choice of the basis $v_1,v_2$ of $L$.

To verify the equivariance property, let $(\alpha,\beta) \in \SU_2^2(\BK)$ and $L\in \Grnd(\BK)$ with basis $v_1,v_2$. 
Then $(\alpha,\beta).L(\BK) = \BK \alpha v_1\beta^{-1} \oplus \BK \alpha v_2\beta^{-1}$ and applying \eqref{eq:associated intpts} for this basis yields
\begin{align*}
a_1((\alpha,\beta).L) 
&= (\alpha v_1 \beta^{-1}) \overline{\alpha v_2\beta^{-1}} -\tfrac{1}{2} \Tr((\alpha v_1 \beta^{-1}) \overline{\alpha v_2\beta^{-1}}) \\
&= \alpha v_1\overline{v_2} \alpha^{-1} - \tfrac{1}{2} \Tr(v_1\overline{v_2})
= \alpha a_1(L) \alpha^{-1}.
\end{align*}
Similarly, $a_2((\alpha,\beta).L) = \beta a_2(L) \beta^{-1}$ and therefore $\Phi$ is equivariant.

To see that the map $\Psi$ defined in \eqref{eq:formulainverse} is equal to the inverse of $\Phi$ we first show that $\Psi$ is equivariant too. Indeed, using the substitution $\hat{x} = \alpha x \beta^{-1}$ we get
\begin{align*}
(\alpha,\beta).\Psi([a_1,a_2]) 
&= (\alpha,\beta).\setc{x\in \quat(\BK)}{a_1 x=x a_2} 
=\setc{\alpha x\beta^{-1}\in \quat(\BK)}{a_1 x=x a_2} \\
&= \setc{\hat{x}\in \quat(\BK)}{\alpha a_1\alpha^{-1} \hat{x}=\hat{x} \beta a_2\beta^{-1}} = \Psi([\alpha a_1\alpha^{-1} ,\beta a_2\beta^{-1}]).
\end{align*}
Let $\overline{\BK}$ denote an algebraic closure. As $\SU_2^2(\overline{\BK})$ acts transitively on $\Grnd(\overline{\BK})$ and $\kleinsp (\overline{\BK})$ and both $\Phi$ and $\Psi$ are equivariant, it suffices to verify $\Psi \circ \Phi = \id$ and $\Phi \circ \Psi = \id$ at one point.
Direct computations show that $\setc{\lambda \in \quat(\overline{\BK})}{\ii \lambda = \lambda \ii } = \langle 1,\ii\rangle_{\overline{\BK}}$ as well as $[a_1(L),a_2(L)] = [\ii,\ii]$ for $L = \langle 1,\ii \rangle_{\overline{\BK}}$
and we obtain that $\Phi$ is a bijection with inverse $\Psi$ for $\overline{\BK}$ and also for $\BK$.

The formula for the stabilizers follows from equivariance as for any $(\alpha,\beta) \in \SU_2^2(\overline{\BK})$
\begin{align*}
(\alpha,\beta). L = L &\iff (\alpha,\beta).[a_1(L),a_2(L)] = [a_1(L),a_2(L)]\\
&\iff
[\alpha a_1(L) \alpha^{-1},\beta a_2(L) \beta^{-1}] = [a_1(L),a_2(L)].
\end{align*} 
By orthogonality this shows that $(\alpha,\beta). L = L$ if and only if there is $\lambda \in \set{\pm 1}$ with $\alpha a_1(L) \alpha^{-1} = \lambda a_1(L)$ and $\beta a_2(L) \beta^{-1} = \lambda a_2(L)$. 
As $\BH_L$ is defined to be the connected component of the stabilizer, this shows the desired equality.

For the second property (which could also be verified using equivariance one more time) observe that $\BH_L = \BH_{L^\perp}$ 
and thus $\BH_{a_i(L)} = \BH_{a_i(L^\perp)}$ for $i=1,2$.
Since the stabilizer of a line within $\SO_3$ determines the line,
we must therefore have $a_1(L^\perp) = \lambda a_1(L)$ and $a_2(L^\perp) = \mu a_2(L)$ for some $\lambda,\mu\in \BK$.
Since $a_1(L^\perp)$ and $a_2(L^\perp)$ have the same value for $Q$, we have $\mu = \pm \lambda$. This shows
\begin{align*}
\Phi(L^\perp) = [a_1(L^\perp),a_2(L^\perp)]= [\lambda a_1(L),\pm \lambda a_2(L)] = [a_1(L),\pm a_2(L)]
\end{align*}
so that $\Phi(L^\perp) = [a_1(L),-a_2(L)]$ as $\Phi(L)\neq \Phi(L^\perp)$.
\end{proof}

\subsection{Associated integer points}\label{section:associated intpts}

Given any rational plane $L\in \Gr(\Q)$ the points $a_1(L),a_2(L)$ defined using a $\Z$-basis $v_1,v_2$ of the lattice $L(\Z)$ are in fact also integral by \eqref{eq:associated intpts}. 
Furthermore, the bilinearity and antisymmetry in \eqref{eq:bilin&antisymm} show that they are well-defined up to changing signs simultaneously. 
If not stated otherwise, we will construct $a_1(L),a_2(L)$ for rational planes $L$ in this fashion and will refer to these points as the \textbf{integer points associated} to $L$. 

Recall that the discriminant $\disc(L)$ of $L(\Z)$ for a rational plane $L$ was defined as the square of the covolume of $L(\Z)\subset L$. 
Alternatively, the discriminant of $L$ may be defined as the discriminant of the restriction of the quadratic form $Q$ to $L(\Z)$.
Recall that the discriminant of a quadratic form with $\Z$-coefficients is given by the determinant of any matrix representation of the form.

\begin{lemma}[Equality of discriminants]\label{lemma:formula lengths}
For any $L \in \Gr(\Q)$ we have
\begin{align*}
Q(a_1(L)) = Q(a_2(L)) = \disc(L).
\end{align*}
\end{lemma}

\begin{proof}
Representing $Q|_{L(\Z)}$ in a basis $v_1,v_2$ of $L(\Z)$ as
\begin{align*}
q(x,y) = Q(xv_1+yv_2) = Q(v_1)x^2 + \Tr(v_1\overline{v_2})xy +  Q(v_2)y^2
\end{align*}
we obtain $\disc(L) = \disc(q) = Q(v_1)Q(v_2) - \tfrac{1}{4}\Tr(v_1\overline{v_2})^2$.
The lemma thus follows from Equation~\eqref{eq:same length}.
\end{proof}

Notice that if $D \in \N$ is square-free and $L \in \mathcal{R}_D$, the associated integer points are in fact both primitive, since by Lemma \ref{lemma:formula lengths} any integer $m$ with $\tfrac{1}{m}a_1(L) \in \Z^3$ would satisfy $m^2\mid D$ and similarly for $a_2(L)$.

For non-square-free $D$ we have the following notion, which serves as a replacement of this observation.
We say that a pair of vectors $(w_1,w_2)\in \Z^3 \times \Z^3$ is \emph{pair-primitive} if $\tfrac{1}{p}w_1 \not \in \Z^3$ or $\tfrac{1}{p}w_2 \not \in \Z^3$ for all odd primes $p$ and if $\tfrac{1}{4}(w_1+w_2)\not \in \Z^3$ or $\tfrac{1}{4}(w_1-w_2)\not \in \Z^3$. 

In fact we will see examples below where the integer points
associated to a plane behave differently with respect to the prime $2$, but not too differently as
pair-primitivity implies that one of the vector $\frac14w_1$ and $\frac14w_2$ is not integral.

\begin{lemma}[Pair primitivity]\label{lemma:pairwise primitivity}
The integer points $(a_1(L),a_2(L))$ associated to a rational plane $L\in \Gr(\Q)$ are pair-primitive. Furthermore, they satisfy $a_1(L) \equiv a_2(L) \mod 2$.
\end{lemma}

\begin{proof}
The antisymmetry of the bilinear maps in \eqref{eq:bilin&antisymm} shows that $\Phi$ factors through the Pl\"ucker embedding
\begin{align*}
L(\BK) = \BK v_1\oplus\BK v_2 \in \Gr(\BK) \mapsto [v_1\wedge v_2]\in \mathbb{P}\big(\nicewedge{2} \BK^4\big).
\end{align*}
Moreover, $\bigwedge^2 \R^4$ has an integral structure given by $\bigwedge^2 \Z^4$. Furthermore, 
for any $v_1,v_2 \in \Z^4$ the wedge $v_1\wedge v_2$ is primitive if and only if\footnote{This can be seen for instance using the Smith normal form; see also \cite[Ch.~1, Lemma 2]{casselsGeoNum} for a concrete proof.} 
$v_1,v_2$ is a basis for $L(\Z)$ where $L = \R v_1 \oplus \R v_2$. If this is the case, we may retrieve the wedge $v_1\wedge v_2$ from $a_1(L)$, $a_2(L)$. In fact, identifying $\bigwedge^2 \R^4\cong \R^6$ via the standard (integral) basis $1 \wedge \ii,\ 1\wedge \jj,\ 1 \wedge \kk,\ \ii\wedge  \jj,\ \ii \wedge \kk,\ \jj \wedge \kk$
a direct calculation using bilinearity of $\wedge$ and the maps in \eqref{eq:bilin&antisymm} shows that
\begin{multline*}
v_1 \wedge v_2  
= \tfrac{1}{2} 
	\big(-(a_1+a_2)_1,-(a_1+a_2)_2,-(a_1+a_2)_3, (a_2-a_1)_3,(a_1-a_2)_2,(a_2-a_1)_1 \big)
\end{multline*}
where $a_i = a_i(L)$ for $i=1,2$.

We will now use this to prove the lemma. For the claims concerning the prime $2$  notice that $\tfrac{1}{2}(a_1+a_2) \in \Z^3$ as $v_1\wedge v_2$ is integral and therefore $a_1 \equiv -a_2 \equiv a_2 \mod 2$. (This can also be seen directly from the definition in \eqref{eq:associated intpts}.)
Furthermore, if $\tfrac{1}{4}(a_1+a_2),\tfrac{1}{4}(a_1-a_2) \in \Z^3$ then $\tfrac{1}{2}(v_1\wedge v_2) \in \Z^6$ contradicting primitivity of $v_1\wedge v_2$.

If $p$ is an odd prime with $\frac{1}{p}a_1,\frac{1}{p}a_2 \in \Z^3$ then $\frac{1}{p}(a_1+a_2),\frac{1}{p}(a_1-a_2) \in \Z^3$ and therefore $\tfrac{1}{p}(v_1\wedge v_2) \in \Z^6$. This contradicts again primitivity of $v_1\wedge v_2$
and shows that $(a_1,a_2)$ are pair-primitive.
\end{proof}



Before discussing the appropriate converse to Lemma \ref{lemma:pairwise primitivity} (see Proposition \ref{prop:integrality for splitting map} below), we would like to point out that the integer points associated to a  plane $L~\in~\Gr(\Q)$ need not be primitive in general.
The same is true for the restriction of the ambient quadratic form $Q$ to $L(\Z)$, which can in fact be non-primitive even for square-free discriminants.

\begin{example}\label{exp:nonprim intpts}
\begin{enumerate}[(a)]
\item The plane $L= \R(1+\mathrm{i}) \oplus \R(\mathrm{i}+\jj)$ has square-free discriminant $D = 3$ and primitive associated integer points
$a_1(L) = -\ii-\jj-\kk$ and $a_2(L) = -\ii-\jj+\kk$. 
However, the form $Q|_{L(\Z)}$ is represented by $2x^2+2xy+2y^2$ (in the given basis) and thus non-primitive.
\item The plane $L= \R(1+\mathrm{i}) \oplus \R(\mathrm{j}+\mathrm{k})$ of discriminant $D=4$ has associated integer points $a_1(L) = -2\mathrm{k}$ and $a_2(L) = -2\mathrm{j}$, both of which are non-primitive.
Furthermore, the quadratic form $Q|_{L(\Z)}$ is represented by the non-primitive form $\Nr(1+\mathrm{i})x^2 + \Nr(\mathrm{j}+\mathrm{k}) y^2
= 2x^2+2y^2$.
%
\item The associated integer points of the plane $L = \R (1 + 2 \ii) \oplus \R(\jj+3\kk)$ of discriminant $D=50$ are the non-primitive vector $a_1(L) = 5(\jj-\kk)$ and the vector $a_2(L) = -7 \jj - \kk$. 
This shows that Lemma \ref{lemma:pairwise primitivity}(a)
cannot be improved to include indivisibility of each vector by odd primes.
In the given integral basis of $L$ the form $Q|_{L(\Z)}$ is represented as $5x^2 + 10y^2$. 
\end{enumerate}
\end{example} 

We will return to these primitivity questions in Section \ref{section:primitivity vs glue} where we will reformulate them in terms of glue groups.
Such a reformulation is however not necessary for the proof of Theorem \ref{thm:main}.

 To help the reader  we will usually point out the stronger statements for square-free discriminants.
In this case, the associated integer points are primitive (and not only pair-primitive) and the quadratic forms on the lattices in question are  primitive  except possibly for the common divisor $2$ (cf.~Proposition \ref{prop:Heegnerpointsfor lattices}, Lemma \ref{lemma:accidental CM} and Example \ref{exp:nonprim intpts}(a) above).

\subsection{Integrality properties of the Klein map}

\begin{proposition}[Pair primitivity, converse claim]\label{prop:integrality for splitting map}
Given $(w_1,w_2) \in \Z^3 \times \Z^3$ pair-primitive with $Q(w_1) = Q(w_2)$ the rational plane $\Phi^{-1}([w_1,w_2])$ has associated integer points $w_1,w_2$ if $w_1 \equiv w_2 \mod 2$ and $2w_1,2w_2$ otherwise.

Moreover, for any $L \in \mathcal{R}_D$ the orthogonal complement $L^\perp$ has associated integer points $a_1(L),-a_2(L)$ and discriminant $D$.
\end{proposition}

\begin{proof}
Write $L = \Phi^{-1}([w_1,w_2])$ and choose by Proposition \ref{prop:splitting map} coprime integers $m,n \in \Z$ with $m a_1(L) = n w_1$ and $m a_2(L) = n w_2$.
Since $(w_1,w_2)$ is pair-primitive, we must have that $m$ is not divisible by any odd prime.
Furthermore, $m$ is also not divisible by $2$.
Otherwise, the equality
\begin{align*}
m (a_1(L) \pm a_2(L) ) = n (w_1\pm w_2)
\end{align*}
combined with $a_1(L)\equiv a_2(L) \mod 2$ (see Lemma \ref{lemma:pairwise primitivity})
shows that $ \frac{1}{4}(w_1\pm w_2) \in \Z^3$ contradicting again the pair-primitivity of $(w_1,w_2)$.
Thus, $m =\pm 1$.

Repeating the same argument for the integer $n$ (without any congruence condition on $(w_1,w_2)$ modulo $2$), we see that $n$ is not divisible by any odd prime and not divisible by $4$.
If $w_1 \equiv w_2\mod 2$ then the above argument shows that $n = \pm 1$.
If $w_1 \not\equiv w_2\mod 2$ then $2\mid n$ as $a_1(L) \equiv a_2(L) \mod 2$.

Now let $L \in \mathcal{R}_D$. 
Then one applies the first part to the pair $a_1(L),-a_2(L)$ to deduce that $L^\perp$ (which is equal to $\Phi^{-1}([a_1(L),-a_2(L)])$ by Proposition \ref{prop:splitting map}) indeed has associated integer points $a_1(L),-a_2(L)$. In fact, we have $a_1(L) \equiv a_2(L) \equiv -a_2(L) \mod 2$
and pair-primitivity by Lemma~\ref{lemma:pairwise primitivity}.
\end{proof}

The complete correspondence  (established by Lemma \ref{lemma:pairwise primitivity}
and Proposition \ref{prop:integrality for splitting map}) between 
\begin{itemize}
\item pairs of integer points up to common sign, which have the same length and are pair-primitive and congruent modulo $2$ and
\item rational planes
\end{itemize}
allows us to prove a claim from the beginning of the introduction.

\begin{corollary}\label{cor: congruence condition}
For any $D \in \mathbb{N}$ the set $\mathcal{R}_D$ is non-empty if and only if $D$ is not congruent to $0$, $7$, $12$ or $15 \mod 16$ (i.e.~$D \in \mathbb{D}$).
\end{corollary}

In particular, if $D$ is not divisible by $4$ (e.g.~ if $D$ is square-free) and $D\not\equiv 7 \mod 8$, then the corollary says that $\mathcal{R}_D$ is non-empty.
The proof will essentially consist in applying the above correspondence (i.e.\ Lemma \ref{lemma:pairwise primitivity} and Proposition \ref{prop:integrality for splitting map}) and Legendre's theorem, which states that a number $D\in \N$ can be written as $D = x^2+y^2+z^2$ for $(x,y,z) \in \Z^3$ primitive if and only if $D \not \equiv 0,4,7\mod 8$. 

\begin{proof}
Lemma \ref{lemma:pairwise primitivity} and Proposition \ref{prop:integrality for splitting map} together
imply that $\mathcal{R}_D$ is non-empty if and only if there exists a tuple $(v,v') \in \Z^3 \times \Z^3$ 
that is pair-primitive with $Q(v) = Q(v') =D$ and $v \equiv v' \mod 2$.

Assume first that $4 \nmid D$. 
We claim that in this case the pair-primitive tuple $(v,v')$ exists if and only if $D \not\equiv 7,15\mod 16$. In fact, if the pair-primitive tuple exists, the vectors represent $D$ as a sum of three squares
which implies $D\not\equiv 7,15 \mod 16$. 
Conversely, we apply Legendre's theorem to find a primitive vector $v$
and set the second integer $v'$ equal to $v$. 

 So suppose now that $4 \mid D$. Then the set $\mathcal{R}_D$ is non-empty if and only if there exists $(w,w') \in \Z^3 \times \Z^3$ pair-primitive with $Q(w) = Q(w') = \frac{D}{4}$ and $w \not\equiv w' \mod 2$.
In fact, if $L \in \mathcal{R}_D$ then $w = \frac{1}{2}a_1(L)$ and $w' = \frac{1}{2}a_2(L)$ are integer vectors as $4 \mid D$ and are not congruent mod $2$ by pair-primitivity of $(a_1(L),a_2(L))$, see Lemma \ref{lemma:pairwise primitivity}.
For the converse one can apply Proposition~\ref{prop:integrality for splitting map} to $(w,w')$.

We claim that there is such a pair $(w,w')$ if and only if $\frac{D}{4} \not \equiv 0,3 \mod 4$ i.e. $D \not \equiv 0,12 \mod 16$.
Indeed, if $(w,w')$ satisfies $Q(w) = Q(w') = \frac{D}{4}$ and $\frac{D}{4} \equiv 0 \mod 4$ then $w$ and $w'$ have only even entries.
Similarly, if $\frac{D}{4} \equiv 3 \mod 4$ the vectors $w,w'$ have only odd entries.
In either case, $w$ and $w'$ are congruent mod $2$.
Conversely, if $\frac{D}{4} \not\equiv 0,3 \mod 4$ there is a primitive vector $w$ with $Q(w) = \frac{D}{4}$ by Legendre's theorem. 
By assumption on $D$, $w$ must have an even and an odd entry so that switching two coordinates yields a vector $w'$ with $Q(w') = \frac{D}{4}$, $w \not\equiv w' \mod 2$ and $(w,w')$ pair-primitive.

This proves the corollary in both cases $4 \mid D$ and $4 \nmid D$.
\end{proof}

\begin{remark}
If one considers dimensions $3 \leq k < n$ with $n-k \geq 2$ and the analogously defined set $\mathcal{R}^{k,n}(D)$ of rational subspaces of dimension $k$ in $\Q^n$ of discriminant $D$, one can show by elementary means using an induction technique from \cite{Schmidt-count} that $\mathcal{R}^{k,n}(D)$ is never empty -- see \cite{Grassmannianhigherdim}.
The authors are however not aware of any counting results as in Corollary~\ref{cor:count} below.
\end{remark}

The above can also be used to obtain a count on the number of points in~$\mathcal{R}_D$.

\begin{corollary}\label{cor:count}
For any $D \in \BD$ we have $|\mathcal{R}_D| = D^{1+o(1)}$.
\end{corollary} 

\begin{proof}
We denote by $r_3(D)$ the number of integer vectors on the sphere of radius $\sqrt{D}$ and by $r_{3,\text{prim}}(D)$ the number of primitive integer vectors. 
Our first goal is to recall that
\begin{equation}\label{fountain-discussion}
D^{\frac12+o(1)}=r_{3,\text{prim}}(D)\leq r_3(D)= D^{\frac12+o(1)}. 
\end{equation}
It is a consequence of Siegel's lower bound \cite{siegellowerbound} 
that $r_3(D)=r_{3,\text{prim}}(D)$ is of the size $D^{\frac{1}{2}+o(1)}$ 
when $D$ is square-free.
In this case, the class group $\Cl(\mathcal{O}_D)$ of the ring of integers $\mathcal{O}_D$ in $\Q(\sqrt{-D})$ acts freely and transitively on a quotient of $\setc{v \in \Z^3}{Q(v) = D}$ by a subgroup of $\SO_3(\Z)$ (see \cite[Prop.~3.5]{Linnikthmexpander}).

Assume now that $D$ is not necessarily square-free and write $D = D'f^2$ where $D'$ is the largest square-free divisor of $D$.
Then one can express (see e.g.~\cite{cooperhirschhorn}) the number $r_{3,\text{prim}}(D)$ as
\begin{align*}
r_{3,\text{prim}}(D) 
= r_{3}(D') f \prod_{p}\left(1- p^{-1}\left(\tfrac{D'}{p}\right)\right)
\end{align*}
where $p$ runs over the odd prime divisors of $f$ and $(\tfrac{\cdot}{p})$ denotes the Legendre symbol.
From this, one deduces that $r_{3,\text{prim}}(D) = D^{\frac{1}{2}+o(1)}$.

Finally we recall that the number of divisiors of $D$ is $  D^{o(1)}$. Summing over $r_{3,\text{prim}}(D'q^2)$ for all square divisors $q^2$ of $D$ we obtain the upper bound
in \eqref{fountain-discussion}.

In particular, \eqref{fountain-discussion} implies that $|\mathcal{R}_D|\leq D^{1+o(1)}$ as the integer points associated to a plane are uniquely determined up to a simultaneous sign and are of norm $\sqrt{D}$.

If $4 \nmid D$ and $(v,v')$ is any pair of primitive integer points of norm $\sqrt{D}$, there exists $g \in \SO_3(\Z)$ such that $g.v' \equiv v \mod 2$. Thus, $|\mathcal{R}_D| \gg r_{3,\text{prim}}(D)^2$ concluding the proof in this case.

If $4\mid D$ and $(w,w')$ is any pair of primitive integer points of norm $\sqrt{D}/2$ there exists $g \in \SO_3(\Z)$ such that $g.w' \not\equiv w \mod 2$ (since $\frac{D}{4} \not\equiv 0,3\mod 4$).
Therefore, we have $|\mathcal{R}_D| \gg r_{3,\text{prim}}(D/4)^2 = D^{1+o(1)}$ also in this case.
\end{proof}

\subsection{Pointwise stabilizers}

For any plane $L \in \Grnd(\BK)$ define the pointwise stabilizer subgroup of $L$ as the connected $\BK$-group stabilizing any point in $L$. In particular, for any $\BK$-algebra $A$
\begin{align*}
\mathbb{H}_L^{\mathrm{pt}}(A)
= \setc{g\in \SU_2^2(A)}{g.x = x \text{ for all } x  \in L(A)}.
\end{align*}
The proof of the dynamical version of Theorem \ref{thm:main} will use the fact that the subgroup $\BH_L^{\mathrm{pt}}$ exhibits a ``$45^\circ$-degree'' twist with respect to the subgroups $\mathbb{H}_{a_1(L)}$, $\mathbb{H}_{a_2(L)}$. Let us illustrate this in an example first (see also Section \ref{section:appendixA}).

\begin{example}\label{exp:pointwisestab}
Consider the plane $\langle 1,\ii \rangle$ and let $h= (h_1,h_2)$ be an element of $\mathbb{H}_{\langle 1,\ii \rangle}^{\mathrm{pt}}(A)$ where $A$ is a $\BK$-algebra.
In particular, $h$ fixes $1$ or in other words $h_1 = h_2$. 
Furthermore, we have $h.\ii = h_1 \ii h_1^{-1} = \ii$. 
This shows that
\begin{align*}
\mathbb{H}_{\langle 1,\ii \rangle}^{\mathrm{pt}}(A) 
= \setc{(h_1,h_1)}{h_1 \in \mathbb{H}_\ii(A) = \Stab_{\SU_2}(\ii)(A)}.
\end{align*}
Similarly, we claim that
\begin{align*}
\mathbb{H}_{\langle \jj,\kk \rangle}^{\mathrm{pt}}(A) 
= \setc{(h_1,h_1^{-1})}{h_1 \in \mathbb{H}_\ii(A) }.
\end{align*}
For this, let $h= (h_1,h_2) \in \mathbb{H}_{\langle \jj,\kk \rangle}^{\mathrm{pt}}(A)$.
Since $\mathbb{H}_{\langle \jj,\kk \rangle}^{\mathrm{pt}}$ is contained in $\mathbb{H}_{\langle \jj,\kk \rangle} = \mathbb{H}_{\ii} \times \mathbb{H}_{\ii}$ (see Proposition \ref{prop:splitting map}), we have $h_1,h_2 \in \mathbb{H}_{\ii}(A)$.
Furthermore, by assumption $h_1 \jj = \jj h_2$ or equivalently $h_1 = \jj h_2 \jj^{-1}$.
However, notice that $h_1,h_2 \in \langle 1,\ii \rangle_A$ and therefore $h_1 = \jj h_2 \jj^{-1} = h_2^{-1}$ as $(\jj,\jj)$ acts by conjugation $x \mapsto \overline{x}$ on the plane $\langle 1,\ii \rangle$. 
\end{example}

\begin{center}
	\includegraphics[scale=0.14]{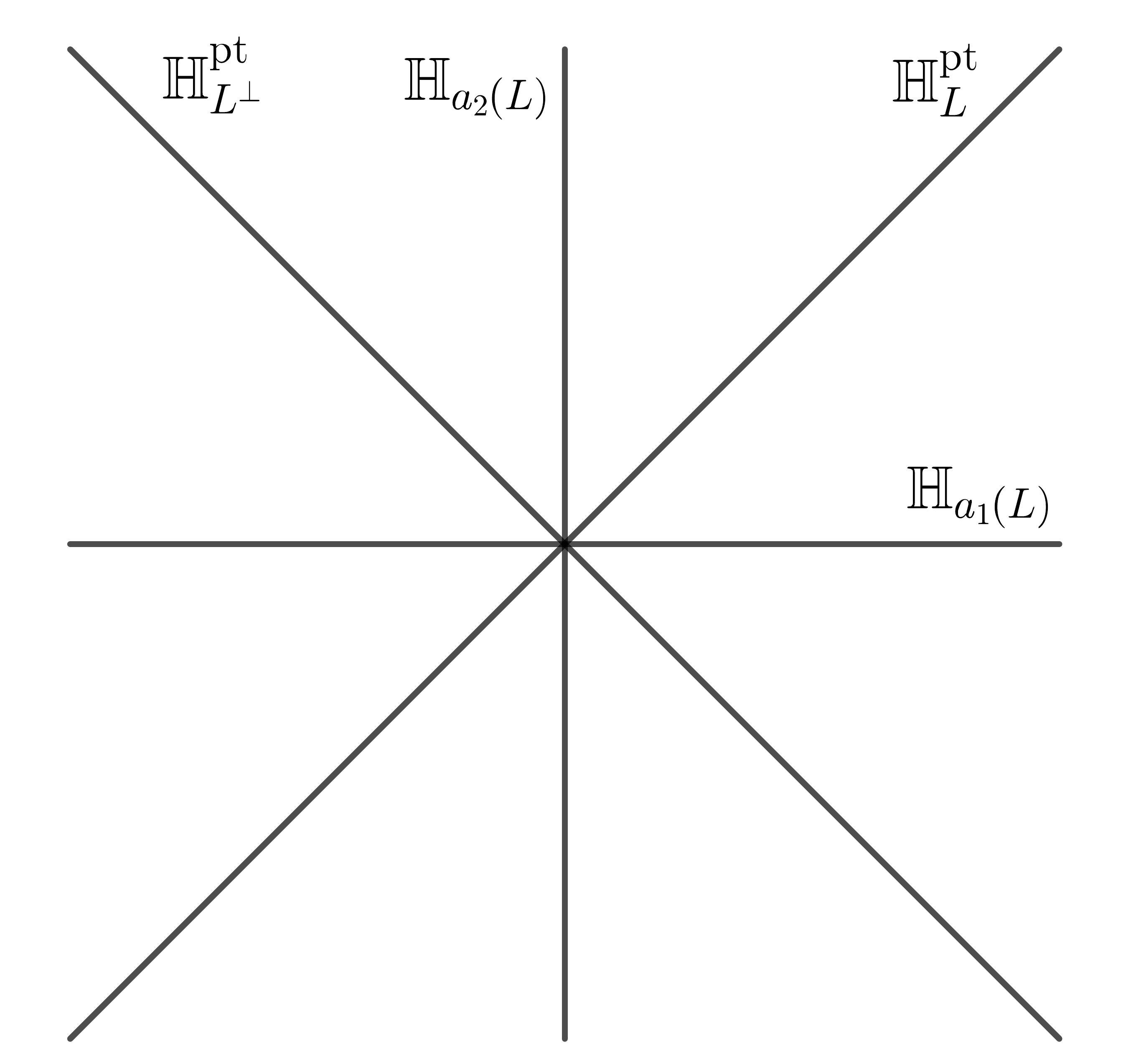}
\end{center}

In general the following holds:
\begin{lemma}[Pointwise stabilizers]\label{lemma:pointwisestab}
Let $L \in \Grnd(\BK)$ be a non-degenerate plane and let $g \in \quat(\BK)^\times$ be any element
with $a_2(L) = g a_1(L) g^{-1}$. Then
\begin{align*}
\mathbb{H}_L^{\mathrm{pt}}(A)
&= \setc{(h_1,gh_1g^{-1})}{h_1 \in \mathbb{H}_{a_1(L)}(A)}, \\
\mathbb{H}_{L^\perp}^{\mathrm{pt}}(A)
&=\setc{(h_1,gh_1^{-1}g^{-1})}{h_1 \in \mathbb{H}_{a_1(L)}(A)}
\end{align*}
for any $\BK$-algebra $A$.
\end{lemma}

We remark that for any invertible $x\in L(\BK)$ the element $g=x^{-1}$ has the property\footnote{Alternatively, recall that the projective group of invertible quaternions is isomorphic over $\Q$ to $\SO_3$ via the action by conjugation on $\quat_0$ (see \cite[Thm.~3.3]{vigneras}). Thus there exists $g \in \quat(\BK)^\times$ with $ga_1g^{-1}=a_2$ by Witt's theorem, see e.g.\ \cite[p.~21]{cassels}.} 
required in the lemma (see Proposition \ref{prop:splitting map}).
Such an element $g$ exists as $L$ is non-degenerate.

\begin{proof}
Let $(\alpha,\beta) \in \SU_2^2(\overline{\BK})$ be such that $(\alpha,\beta).\langle 1,\ii \rangle_{\overline{\BK}} = L(\overline{\BK})$.
Applying Proposition \ref{prop:splitting map} we have
\begin{align*}
\BH_{a_1(L)} \times \BH_{a_2(L)}
= \mathbb{H}_L
= (\alpha,\beta) \mathbb{H}_{\langle 1,\ii \rangle} (\alpha,\beta)^{-1}
= \alpha \mathbb{H}_\ii \alpha^{-1} \times \beta \mathbb{H}_\ii \beta^{-1}
\end{align*}
and in particular $\alpha \mathbb{H}_\ii \alpha^{-1} = \mathbb{H}_{a_1(L)}$.
By Example \ref{exp:pointwisestab}
\begin{align*}
\mathbb{H}_L^{\mathrm{pt}}(A\otimes \overline{\BK})
&= (\alpha,\beta) \mathbb{H}_{\langle 1,\ii \rangle}^{\mathrm{pt}}(A\otimes \overline{\BK}) (\alpha,\beta)^{-1}
= \setc{(\alpha h_1 \alpha^{-1}, \beta h_1 \beta^{-1})}{h_1 \in \mathbb{H}_\ii(A\otimes \overline{\BK})} \\
&= \setc{(h_1,\beta \alpha^{-1} h_1 \alpha \beta^{-1})}{h_1\in \mathbb{H}_{a_1(L)}(A\otimes \overline{\BK})}
\end{align*}
so that $\BH_{L}^{\mathrm{pt}}$ is the graph of the morphism $\BH_{a_1(L)} \to \BH_{a_2(L)}$ given by conjugation with $\alpha \beta^{-1}$. 
Furthermore, notice that
\[
 \beta \alpha^{-1} a_1(L) \alpha \beta^{-1} = \beta \ii \beta^{-1} = a_2(L)=g a_1(L) g^{-1}.
\]
Hence, $\beta \alpha^{-1} \mathbb{H}_{a_1(L)}=g\mathbb{H}_{a_1(L)}$
and one may replace $\beta \alpha^{-1}$ by $g$ in the above (recall that $\BH_{a_1(L)}$ is abelian) to obtain the first part of the lemma.

The second part follows analogously by noting that $(\alpha,\beta).\langle \jj,\kk \rangle_{\overline{\BK}} = L^\perp(\overline{\BK})$ and applying Example \ref{exp:pointwisestab} again.
\end{proof}

We finish this section by clarifying the meaning of the assumed congruence conditions in Theorem \ref{thm:main}.

\begin{lemma}\label{lemma:reason for congruencecond.}
Let $L\in \mathcal{R}_D$ be a rational plane. If $p$ is an odd prime so that $\legendre{-D}{p} = 1$, then $\mathbb{H}_L(\Q_p)$ is a split torus of rank two.
\end{lemma}

\begin{proof}
Let $v \in \quat_0(\Q)$ be any vector with $Q(v) = D$ and note that there is an isomorphism $\quat(\Qp) \cong \Mat_2(\Q_p)$ of $\Q_p$-algebras mapping $v$ to some traceless $X\in \Mat_2(\Q_p)$ of determinant $D$. Indeed, the norm form of $\quat$ is isotropic over $\Q_p$ so that $\quat(\Q_p)$ is not a division algebra and hence isomorphic to $\Mat_2(\Q_p)$ \cite[Prop.~7.6.2]{voight}.
The stabilizer subgroup $\mathbb{H}_v(\Q_p)$ is thus isomorphic to
\begin{align*}
\setc{g \in \SL_2(\Q_p)}{gX = Xg}.
\end{align*}
By assumption and by Hensel's lemma the characteristic polynomial $x^2 + D$ of $X$ has two distinct, non-zero roots in $\Q_p$ so $X$ is diagonalizable over $\Q_p$. Thus, $\mathbb{H}_v(\Q_p)$ is a one-dimensional split torus.

By Lemma \ref{lemma:formula lengths} one can apply the above to $v=a_1(L)$ and $v=a_2(L)$.
The claim then follows as $\mathbb{H}_L(\Q_p) = \mathbb{H}_{a_1(L)}(\Q_p) \times \mathbb{H}_{a_2(L)}(\Q_p)$ by Proposition~\ref{prop:splitting map}.
\end{proof}

\section{CM-points}\label{section:CMpts}

\subsection{Defining the shapes}\label{section:def shapes}

 We recall that $P,P_1,P_2$ denote the factor maps from $\SU_2\times\SU_2$
to $\SO_4$ respectively $\SO_3$ using the first or second factor, see Section \ref{section:quaternions}.
In the following we identify $\quat(\R)$ with row vectors in $\R^4$ using the basis $1,\ii,\jj,\kk$
and $\quat_0(\R)$ with row vectors in $\R^3$ using the basis $\ii,\jj,\kk$.   In particular,
we choose to let $g\in\SO_4(\R)$ act on row vectors $v\in\R^4$ via $g.v = vg^t$. 
Observe that this action simply corresponds to the usual action on column vectors.

Let us now fix a plane $L \in \mathcal{R}_D$ for some $D\in \mathbb{D}$. 
We choose an integer matrix $A_{1,L}\in \SL_4(\Z)$ whose first two rows form a basis of $L(\Z)$ and fix an isometry $k_L = (\alpha_L,\beta_L) \in\SU_2^2(\R)$ with $k_L.\langle 1, \mathrm{i}\rangle_\R=L(\R)$. In particular, this also implies $k_L.\langle \jj, \kk\rangle_\R =L^\perp(\R)$
by orthogonality, respectively
\begin{align}\label{eq:correctversion}
\alpha_L^{-1}. a_1(L)=  \beta_L^{-1}. a_2(L)= \pm  \sqrt{D}\cdot \ii.
\end{align}
by the equivariance of the isomorphism in Proposition~\ref{prop:splitting map}. Then we have
\begin{align*}
A_{1,L} P(k_L) \in \left\lbrace
\left(\begin{array}{cc|cc}
\ast&\ast & 0 & 0 \\
\ast&\ast & 0 & 0 \\ \hline
\ast &\ast & \ast &\ast \\
\ast & \ast & \ast &\ast
\end{array}\right)  \right\rbrace.
\end{align*}
Here, we used that for any $v\in L$ viewed as a row vector the vector $vP(k_L)$ corresponds to $k_L^{-1}.v$ when viewed as an element of $\quat(\R)$ and that $k_L^{-1}.v \in \langle 1,\ii \rangle$ where $\langle 1,\ii \rangle$ is identified with $\R^2 \times \set{(0,0)}$.
The \textbf{shape} of $L(\Z)$ is defined as
\begin{align*}
[L(\Z)] &:= \PGL_2(\Z) \pi_1(A_{1,L} P(k_L)) \PO(2) \in \mathcal{X}_2
\end{align*}
where~$\pi_1$ is the projection onto the upper left block and where as in the introduction
\begin{align*}
\mathcal{X}_2 &= \lquot{\PGL_2(\Z)}{\mathbb{H}^2} 
= \lrquot{\PGL_2(\Z)}{\PGL_2(\R)}{\PO(2)}
\end{align*} 
Note that the definition of the shape is independent of the choices of $A_{1,L}$ and~$k_L$. 

Similarly, one can choose a matrix $A_{2,L}\in \SL_4(\Z)$ whose last two rows form a basis of $L^\perp(\Z)$ so that
\begin{align*}
A_{2,L} P(k_L) \in \left\lbrace
\left(\begin{array}{cc|cc}
\ast&\ast & \ast & \ast \\
\ast&\ast & \ast & \ast \\ \hline
 0 & 0 & \ast &\ast \\
 0 & 0 & \ast &\ast
\end{array} \right)  \right\rbrace.
\end{align*}
Denoting by $\pi_2$ the projection onto the lower right block the shape of $L^\perp(\Z)$ is
\begin{align*}
[L^\perp(\Z)] &:= \PGL_2(\Z) \pi_2(A_{2,L} P(k_L)) \PO(2) \in \mathcal{X}_2,
\end{align*}
which is again independent of the choice of $A_{2,L}$ and $k_L$. 

As in the introduction, the shapes $[L(\Z)]$ and $[L^\perp(\Z)]$ will be called the geometric CM points attached to~$L$.
By Proposition \ref{prop:integrality for splitting map} the plane $L^\perp$ also has discriminant~$D$ and so
\begin{align}\label{eq:covolume-shape}
\det(\pi_1(A_{1,L} P(k_L))) =   \det(\pi_2(A_{2,L} P(k_L))) = \pm\sqrt{D}.
\end{align}
By adapting $A_{1,L}$ and $A_{2,L}$ we may assume that these determinants are positive. 
We note that in case $D$ is not square-free the discriminant of the geometric CM points may be a divisor of $D$
instead of $D$ itself, we will explain this in detail in Proposition \ref{prop:Heegnerpointsfor lattices} below.

\ifsquarefree
Define as in the introduction the orthogonal lattices $\Lambda_{a_1(L)}$ and $\Lambda_{a_2(L)}$, which have discriminant~$D$ (cf.~\cite[p.~379]{AES3D}\menote{AES3D has a lot of pages} and recall for this that $a_1(L),a_2(L)$ are primitive as $D$ is square-free).
\else
Define as in the introduction the orthogonal lattices $\Lambda_{a_1(L)}$ and $\Lambda_{a_2(L)}$. If $D$ is square-free, $a_1(L)$ and $a_2(L)$ are primitive and the orthogonal lattices have discriminant $D$ (c.f. \cite[p.~379]{AES3D}).
Otherwise, the discriminant of these lattices are the squared lengths of primitive vectors in $\Q\, a_1(L)$ resp. $\Q\, a_2(L)$.
\fi
Let $A_{3,L} \in \SL_3(\Z)$ be a matrix whose last two rows are a basis of $\Lambda_{a_1(L)}$ and define $A_{4,L}$ analogously for $a_2(L)$. The above choices together with \eqref{eq:correctversion} yield
\begin{align*}
A_{3,L} P_1(k_L),\ A_{4,L} P_2(k_L) \in \set{\left(\begin{array}{c|cc}
\ast & \ast & \ast \\
\hline
0 & \ast & \ast \\ 
0 & \ast & \ast
\end{array}\right) }.
\end{align*}
Denoting by $\pi$ the projection onto the lower right block we obtain the shapes
\begin{align*}
[\Lambda_{a_1(L)}] &= \PGL_2(\Z)\, \pi(A_{3,L} P_1(k_L))\, \PO(2) \in \mathcal{X}_2, \\
[\Lambda_{a_2(L)}] &= \PGL_2(\Z)\, \pi(A_{4,L} P_2(k_L))\, \PO(2) \in \mathcal{X}_2,
\end{align*}
which are independent of the choices made.
These are the accidental CM points attached to $L$.
Note that an analogous construction defines the shape $[\Lambda_v]$ of an orthogonal lattice $\Lambda_v = v^\perp \cap \Z^3$ for any primitive integer point $v \in \Z^3$. 

\subsection{CM points and relationship to shapes}

The notions \textit{shapes of two-di\-men\-sional lattices}, \textit{binary quadratic forms} and \textit{CM-points} are related as we now explain.
Recall that $\SL_2$ acts on positive definite binary forms via
\begin{align*}
q(x,y) \mapsto g.q(x,y) = q((x,y)g)
\end{align*}
preserving the discriminant.
In particular, $\SL_2(\Z)$ acts on integral binary forms; we will denote by $[\cdot]$ the orbit equivalence classes for the latter action. 
For any equivalence class $[q]$ associated to a positive definite integral form $ q(x,y) = ax^2+bxy+cy^2$ we obtain a well-defined point (called a CM point)
\begin{align*}
\cm{z}_{[q]}
=\SL_2(\Z). \frac{-b + \sqrt{|b^2-4ac|}\ii}{2a} \in \lquot{\SL_2(\Z)}{\mathbb{H}^2} = Y_0(1)
\end{align*}
which doesn't depend on the choice of representative and which satisfies $\cm{z}_{[q]} = \cm{z}_{[\alpha q]}$ for any $\alpha \in \N$.

Given a rational plane $L$ the (potential) point $\cm{z}_{[Q|_{L(\Z)}]}$ is typically not well-defined as a point of $Y_0(1)$. 
Indeed, if $q(x,y) = ax^2+ bxy + cy^2$ is a representation of $Q|_{L(\Z)}$ in a basis, then so is $q'(x,y) = ax^2 - bxy + cy^2$ which is in the $\GL_2(\Z)$-orbit (when the action is extended) but not typically in the same $\SL_2(\Z)$-orbit.
The two points $\cm{z}_{[q]}$ and $\cm{z}_{[q']}$ are related by $\cm{z}_{[q']} = -\overline{\cm{z}_{[q]}}$ and one can check that the class of $\cm{z}_{[q]}$ in the two-to-one quotient
\begin{align*}
\mathcal{X}_2= \lrquot{\PGL_2(\Z)}{\PGL_2(\R)}{\PO(2)}
\end{align*}
of $Y_0(1)$ is independent of the choice of basis.
We view the CM point $\cm{z}_{[q]}$ for a class $[q]$ as an element of $\mathcal{X}_2$ whenever there is no ambiguity.

Using the basis of $L(\Z)$ contained in $A_{1,L}$ one can represent the quadratic form $Q|_{L(\Z)}$ by
\begin{align*}
q_L(x,y) = Q((x,y,0,0)A_{1,L}).
\end{align*}
Similarly, we represent the form $Q|_{L^\perp(\Z)}$ by
\begin{align*}
q_{L^\perp}(x,y) = Q((0,0,x,y)A_{2,L}).
\end{align*}
Both binary forms $q_L$ and $q_{L^\perp}$ have discriminant $D$ as $L^\perp\in \mathcal{R}_D$ by Proposition~\ref{prop:integrality for splitting map}.

\begin{proposition}[Geometric CM points]\label{prop:Heegnerpointsfor lattices}
Let $L\in \mathcal{R}_D$ for $D \in \mathbb{D}$.
Then 
\begin{align*}
\cm{z}_{[Q|_{L(\Z)}]} = \cm{z}_{[q_L]} = [L(\Z)]
\end{align*}
and analogously $\cm{z}_{[Q|_{L^\perp(\Z)}]} = \cm{z}_{[q_{L^\perp}]}= [L^\perp(\Z)]$.

If $D$ is square-free, either $Q|_{L(\Z)}$ is a primitive integral form or $\frac{1}{2}Q|_{L(\Z)}$ is a primitive integral form.
%
%
\end{proposition}

The situation for non-square-free discriminants will be analyzed later (see Proposition \ref{prop:geometric-nonsqfree}) using local considerations.

\begin{proof}
Let us first prove the statement in the proposition concerning the CM points $\cm{z}_{[Q|_{L(\Z)}]}$, $\cm{z}_{[Q|_{L^\perp(\Z)}]}$. For this we apply a similar argument as in \cite[p.~391-392]{AES3D}.
Write $v_1,v_2$ for the first two rows of $A_{1,L}$ and note that
\begin{align*}
q_L(x,y) = ax^2 + b xy + c y^2
\end{align*}
for $a=Q(v_1)$, $b = 2(v_1,v_2)$ and $c = Q(v_2)$ where $(\cdot,\cdot)$ denotes the (standard) bilinear form induced by $Q$.
Furthermore, set
$\pi_1(A_{1,L} P(k_L))=\left(\begin{smallmatrix}
a_{11} & a_{12}\\
a_{21} & a_{22}
\end{smallmatrix}\right)$. 
By definition, we have 
$a_{11}^2 + a_{12}^2 = Q(v_1)=a$, 
$a_{21}^2 + a_{22}^2 = Q(v_2)=c$ 
and $a_{11}a_{21} + a_{12}a_{22}= (v_1,v_2)= \frac{b}{2}$. 
As $\sqrt{D}$ is the covolume of $L(\Z)$ (see \eqref{eq:covolume-shape}), we further note that
\begin{align*}
a_{11}a_{22}-a_{12}a_{21} = \det(\pi_1(A_{1,L} P(k_L))= \sqrt{D} = \sqrt{ac-\tfrac{1}{4}b^2}.
\end{align*}
To compute a representative of $\pi_1(A_{1,L} P(k_L)).\ii$ in $\mathcal{X}_2$ we may also conjugate $\left(\begin{smallmatrix}
a_{11} & a_{12}\\
a_{21} & a_{22}
\end{smallmatrix}\right)$ with $\left(\begin{smallmatrix}
0 & 1\\
-1 & 0
\end{smallmatrix}\right) \in \SO_2(\Z)$ to obtain
\begin{align*}
\begin{pmatrix}
a_{22} & -a_{21}\\
-a_{12} & a_{11}
\end{pmatrix}.\ii
&= \frac{a_{22}\ii - a_{21}}{-a_{12}\ii+a_{11}}
= \frac{(-a_{12}a_{22}-a_{11}a_{21})+ \ii (-a_{12}a_{21}+a_{11}a_{22})}{a} \\
&= \frac{-\tfrac{b}{2} + \ii \sqrt{D}}{a} = \frac{-b+\ii\sqrt{4D}}{2a}
\end{align*}
For the equality $\cm{z}_{[Q|_{L^\perp(\Z)}]} = [L^\perp(\Z)]$ we note that the discriminant of $L^\perp(\Z)$ is also~$D$ by Proposition \ref{prop:integrality for splitting map} so that the above proof applies.

For the second part of the proposition assume that $D$ is square-free.
We recall that by definition $ac-\frac{b^2}{4} = D$.
This implies that $p\nmid \gcd(a,b,c)$ for any odd prime~$p$ (as otherwise $p^2\mid D$) and $4\nmid\gcd(a,b,c)$ as claimed.
%
\end{proof}

\subsection{Local analysis at odd primes}\label{section:local for forms}

Here, we generalize the second part of Proposition~\ref{prop:Heegnerpointsfor lattices} addressing primitivity issues of $Q|_{L(\Z)}$ for $L \in \mathcal{R}_D$ to non-square-free discriminants $D$.

If $D$ is not square-free, the form $Q|_{L(\Z)}$ might be non-primitive and the discriminant\footnote{By definition of the discriminant in Section \ref{section:associated intpts}, the discriminant of a binary quadratic form $ax^2+bxy +cy^2$ is given by $ac-\frac{b^2}{4}$.} of the primitive forms $\widetilde{q}_L$ might be much smaller than $D$.
In the proposition to follow we will compute the discriminant of the form $\widetilde{q}_L$.

Given a binary integral form $q$ and a prime $p$ we denote by $\ord_p(q)$ the largest integer $k$ for which $p^{-k}q$ is integral.
Similarly, given a vector $v \in \Z^3$ we let $\ord_p(v)$ for a prime $p$ be the largest integer $k$ with $p^{-k}v \in \Z^3$. We also set $\widetilde{v}$ to be the primitive integer vector in the half-line $\Q_{+} v$.

\begin{proposition}[Geometric CM points and non-square-free discriminants]\label{prop:geometric-nonsqfree}
Let $L\in \mathcal{R}_D$ for $D \in \mathbb{D}$. Then
\begin{align*}
\ord_p(Q|_{L(\Z)}) = \ord_p(q_L) = \max\big\lbrace\ord_p(a_1(L)),\ord_p(a_2(L))\big\rbrace
\end{align*}
for any odd prime $p$ and $\ord_2(Q|_{L(\Z)})\leq 4$.
Furthermore, the discriminant of the primitive form $\widetilde{q}_L$ (or $\widetilde{q}_{L^\perp}$) satisfies
\begin{align}\label{eq:formula discriminant CM}
\disc(\widetilde{q}_L) \asymp \frac{Q(\widetilde{a}_1(L)) Q(\widetilde{a}_2(L))}{D}.
\end{align}
\end{proposition}

The proposition essentially says that the quadratic form $Q|_{L(\Z)}$ inherits the ``arithmetic complexity'' from the integer points of the plane $L$. More precisely, if one of the vectors $a_1(L)$ and $a_2(L)$ is ``very non-primitive'' then the same will hold for the form.
This fact is also reflected in Lemma \ref{lemma:pointwisestab}.
We also note that the statement at the prime $2$ could be sharpened.
This however is not needed for the proof of \eqref{eq:formula discriminant CM}.

Let $p$ be a fixed odd prime (the statement for the prime $2$ will be a simple consequence of the congruence condition $D \in \BD$).
Then the quaternion algebra $\quat$ is split over $\Q_p$ as the norm form is isotropic \cite[Prop.~7.6.2]{voight} so that we have $\quat(\Q_p) \cong \Mat_2(\Q_p)$ and in fact $\quat(\Z_p) \cong \Mat_2(\Z_p)$.
Under this fixed isomorphism, conjugation on $\quat(\Qp)$ corresponds to the adjunct on $\Mat_2(\Qp)$, the (reduced) trace to the usual trace on matrices and the (reduced) norm to the determinant.
Now note that the Klein map (see Proposition \ref{prop:splitting map}) was defined using only these operations and can thus be defined for two-dimensional subspaces of $\Mat_2$ as well\footnote{In fact, the Klein map makes sense for two-dimensional subspaces of any quaternion algebra. In principle, the arguments of this paper carry over to yield a more general statement about such planes and the induced shapes (for the norm form) -- see also Remark~\ref{rem:extensions}.}.
We will therefore freely identify the resulting Klein map for $\Mat_2(\Qp)$ with the Klein map for $\quat(\Qp)$.

Let $\mathcal{R}_{D,p}$ be the set of two-dimensional subspaces $L$ in $\Q_p^4$ with\footnote{As for subspaces $L \subset \Q^4$ we define the discriminant of any $L \in \Gr(\Q_p)$ as the discriminant of $Q|_{L(\Z_p)}$ where $L(\Z_p) = L \cap \Z_p^4$. For $\Z_p$-integral forms the discriminant is defined up to multiples in $(\Z_p^\times)^2$.}
 $\disc(L) = D (\Z_p^\times)^2$ and note that $\SU_2^2(\Z_p)$ acts on $\mathcal{R}_{D,p}$ (preserving in fact also $\Z_p$-equivalence class of the form $Q|_{L(\Z_p)}$ for any $L \in \mathcal{R}_{D,p}$).
Observe also that the Klein map using an integral basis associates to any $L \in \mathcal{R}_{D,p}$ a pair of vectors $(a_1(L),a_2(L))\in \Z_p^3 \times \Z_p^3$ with $Q(a_1(L)) = Q(a_2(L)) \in D (\Z_p^\times)^2$, which is well-defined up to simultaneous multiples in $\Z_p^\times$. 

\begin{lemma}\label{lemma:loc.conjugacy of planes}
Let $L \in \mathcal{R}_{D,p}$.
Then there is some $g \in \SU_2^2(\Z_p)$ such that
\begin{align}\label{eq: local lattices}
g.L(\Z_p) = \Z_p \begin{pmatrix}
\alpha_1 & 0 \\ 
0 & \alpha_2
\end{pmatrix} \oplus \Z_p \begin{pmatrix}
0 & 1 \\ 
-\frac{D}{\alpha_1\alpha_2} & 0
\end{pmatrix}
\end{align}
for some $\alpha_1,\alpha_2\in \Zp\setminus\set{0}$ with $\ord_p(\alpha_1) = \ord_p(a_1(L))$ and $\ord_p(\alpha_2) = \ord_p(a_2(L))$.
\end{lemma}

\begin{proof}
Choose a basis of $L$ for which $Q(a_1(L)) = Q(a_2(L)) = D$.
We may assume without loss of generality that $\ord_p(a_1(L)) = k= \max\set{\ord_p(a_1(L)),\ord_p(a_2(L))}$ which implies $\ord_p(a_2(L)) = 0$ by pair-primitivity.
(Otherwise, one can replace $L$ by $\overline{L}$ which interchanges $a_1(L)$ and $a_2(L)$.)

In order to obtain a plane of the desired form we use equivariance of the Klein map.
Note that the action of $\SU_2(\Z_p)$ on the set of primitive vectors $v \in \quat_0(\Z_p)$ with $N(v) = d$ for some fixed non-zero $d \in \Z_p$ has at most two orbits which can be represented by
\begin{align}\label{eq:intptsconjugacy over Zp}
b_\lambda(d,p) = \begin{pmatrix}
0 & \lambda\\ -d\lambda^{-1} & 0
\end{pmatrix}
\end{align}
Here, $\lambda\in \set{1,\varepsilon_p}$ where $\varepsilon_p$ a non-square in $\Z_p^\times$.
In fact, one shows that  $\GL_2(\Z_p)$ acts transitively on the above set of vectors, and the matrices $b_1(d,p)$ and $b_{\varepsilon_p}(d,p)$ are not conjugate by a matrix in $\SL_2(\Z_p)$ if and only if $p$ divides $d$ (see also the proof of \cite[Prop.~3.7]{Linnikthmexpander}).

Applying this to $p^{-k}a_1(L)$ and $a_2(L)$ to find $g\in \SU_2^2(\Z_p)$ such that 
\begin{align}\label{eq: two vectors}
g_1.a_1(L) = p^kb_{\lambda_1}(D/p^{2k},p)\text{ and } g_2.a_2(L) = b_{\lambda_2}(D,p)
\end{align}
for some $\lambda_1,\lambda_2 \in \set{1,\varepsilon_p}$.
A direct computation using either the Klein map on \eqref{eq: local lattices} or its inverse on \eqref{eq: two vectors} shows that $g.L$ is of the form desired in the lemma where we set $\alpha_1 = -\lambda_1 p^k$ and $\alpha_2 = -\lambda_2$.
\end{proof}

\begin{proof}[Proof of Proposition \ref{prop:geometric-nonsqfree}]
Let $p$ be an odd prime and let $L \in \mathcal{R}_D \subset \mathcal{R}_{D,p}$.
Furthermore, let $g,\alpha_1,\alpha_2$ be chosen as in Lemma \ref{lemma:loc.conjugacy of planes} for $L$ and set
\begin{align*}
k= \max\set{\ord_p(a_1(L)),\ord_p(a_2(L))} = \ord_p(\alpha_1\alpha_2).
\end{align*}
Clearly, $Q|_{L(\Z_p)}$ is $\Z_p$-equivalent to $Q|_{g.L(\Z_p)}$.
In the orthogonal basis for $g.L(\Z_p)$ of Lemma \ref{lemma:loc.conjugacy of planes} the form $Q|_{g.L(\Z_p)}$ is represented by $q(x,y)=\alpha_1\alpha_2 x^2+ \frac{D}{\alpha_1\alpha_2} y^2$.
Since $p^{2k}\mid D$,  we obtain $k=\ord_p(q) = \ord_p(Q|_{L(\Z_p)})$ as claimed.

For the statement at the prime $2$ note that for any $L \in \mathcal{R}_D$ we have $$\ord_2(Q|_{L(\Z)}) \leq \ord_2(D)+1 \leq 4$$ as $D \in \BD$.

To prove Equation \eqref{eq:formula discriminant CM} note that by pair-primitivity
\begin{align*}
\disc(\widetilde{q}_L) 
&\asymp \frac{D}{\prod_{p \text{ odd}} p^{2\ord_p(a_1(L))+2\ord_p(a_2(L)) }} \\
&= \frac{1}{D} \frac{Q(a_1(L))}{\prod_{p \text{ odd}} p^{2\ord_p(a_1(L))}}
\frac{Q(a_2(L))}{\prod_{p \text{ odd}} p^{2\ord_p(a_2(L))}}
\asymp \frac{Q(\widetilde{a}_1(L)) Q(\widetilde{a}_2(L))}{D}
\end{align*}
as claimed (where $\asymp$ is used in order to ignore the prime $2$).
\end{proof}

\subsection{Accidental CM points}

As for the geometric CM points, we use the basis of $\Lambda_{a_1(L)}$ contained in $A_{3,L}$ to represent $Q|_{\Lambda_{a_1(L)}}$ via the binary form $q_{a_1(L)}(x,y) = Q|_{\quat_0}((0,x,y)A_{3,L})$.
The form $q_{a_2(L)}$ is defined analogously.
As mentioned, in \cite[p.~391-392]{AES3D} and \cite[Lemma~3.3]{AEShigherdim} (in the non-square-free case) a discussion similar to Proposition \ref{prop:Heegnerpointsfor lattices} was carried out for the shapes $[\Lambda_{a_1(\cdot)}]$ and $[\Lambda_{a_2(\cdot)}]$, which we summarize in the following lemma.

\ifsquarefree
\begin{lemma}[Accidental CM points]\label{lemma:accidental CM}
Let $v \in \Z^3$ with $Q(v) = D$ square-free. If $D \equiv 1,2 \mod 4$, the quadratic form $Q|_{\Lambda_v}$ is primitive. If $D \equiv 3 \mod 4$,  the quadratic form $\tfrac{1}{2}Q|_{\Lambda_v}$ is integral and primitive. Furthermore,
\begin{align*}
\cm{z}_{[Q|_{\Lambda_v}]} = [\Lambda_v].
\end{align*}
\end{lemma}
\else
\begin{lemma}[Accidental CM points]\label{lemma:accidental CM}
Let $v \in \Z^3$ and set $d = Q(\widetilde{v})$.
Then we have $\cm{z}_{[Q|_{\Lambda_v}]} = [\Lambda_v] = [\Lambda_{\widetilde{v}}]$.

If $d \equiv 1,2 \mod 4$, the quadratic form $Q|_{\Lambda_v}$ is primitive. 
If $d \equiv 3 \mod 4$,  the quadratic form $\tfrac{1}{2}Q|_{\Lambda_v}$ is integral and primitive.
Furthermore, $\disc(\widetilde{q}_v) \asymp Q(\widetilde{v})$.
\end{lemma}
\fi
\section{The dynamical formulation of the theorem}\label{section:formulation}

\subsection{The joint acting group}\label{section:acting group}

In this section we first determine the stabilizer subgroups for the CM points associated to a given rational plane and use these to define the acting group appearing in the dynamical version of Theorem \ref{thm:main}.

Let $D\in \mathbb{D}$ and let $L \in \mathcal{R}_D$.
For any $\Q$-algebra $A$ and $h \in \SU_2^2(A)$ we (trivially) have
\begin{align*}
q_L(x,y) = Q((x,y,0,0)A_{1,L} P(h)) = Q((x,y,0,0) A_{1,L} P(h)A_{1,L}^{-1} A_{1,L}).
\end{align*}
Moreover, for any $h\in \mathbb{H}_L(A)$ the matrix $P(h)$ preserves $L$ by definition and therefore $A_{1,L} P(h) A_{1,L}^{-1}$ is of the block-form
$\big(\scriptsize\arraycolsep=0.3\arraycolsep\ensuremath{
\begin{array}{c|c}
\ast & 0  \\
\hline
\ast & \ast
\end{array}}
\big)$. Projecting this to the upper left block we obtain a homomorphism $\Psi_{1,L}: \BH_L \to \Stab_{\SL_2}(q_L) =: \BH_{q_L}$ given by 
\begin{align*}
\Psi_{1,L}(h) = \pi_1(A_{1,L} P(h) A_{1,L}^{-1})
\end{align*}
defined over $\Q$. 
Notice that the image of $\Psi_{1,L}$ lies indeed in $\SL_2$ as $\BH_L$ is (by definition) connected.
The map $\Psi_{1,L}(h)$ for $h \in \BH_L(A)$ should be thought of as the restriction of the action of $h$ to the plane $L$
represented in our basis of $L(\Z)$.
Analogously, we have a homomorphism
\begin{align*}
\Psi_{2,L}: \mathbb{H}_{L} \to \Stab_{\SL_2}(q_{L^\perp})=: \BH_{q_{L^{\perp}}},\  
\Psi_{2,L}(h) = \pi_2(A_{2,L} P(h) A_{2,L}^{-1})
\end{align*}
when projecting to the lower right block. 
Observe furthermore that $\Psi_{1,L}(h)$ is trivial if and only if $h \in \mathbb{H}_{L}^{\mathrm{pt}}(A)$ and that $\Psi_{2,L}(h)$ is trivial if and only if $h \in \mathbb{H}_{L^\perp}^{\mathrm{pt}}(A)$.
In Appendix~\ref{section:appendixA} the above isogenies (and also the isogenies for the accidental CM points defined below) are computed explicitly in a special case.

For the accidental CM points the analogous morphisms are given by
\begin{align*}
\Psi_{3,L}: \mathbb{H}_{L} & \to \Stab_{\SL_2}(q_{a_1(L)}) =: \BH_{q_{a_1(L)}},\  
&\Psi_{3,L}(h) = \pi(A_{3,L} P_1(h) A_{3,L}^{-1}) \\
\Psi_{4,L}: \mathbb{H}_{L} &\to \in \Stab_{\SL_2}(q_{a_2(L)})=: \BH_{q_{a_2(L)}},\  
&\Psi_{4,L}(h) = \pi(A_{4,L} P_2(h) A_{4,L}^{-1}) 
\end{align*}
and are also defined over $\Q$.

Overall, we define the $\Q$-group $\torus_L<\SU_2^2 \times \SL_2^4$ to be the graph of the morphism 
\begin{align*}
(\Psi_{1,L},\Psi_{2,L},\Psi_{3,L},\Psi_{4,L}): \mathbb{H}_{L} \to
\BH_{q_L} \times \BH_{q_{L^\perp}}\times \BH_{q_{a_1(L)}} \times \BH_{q_{a_2(L)}}.
\end{align*}
For convenience, we set for any $h \in \BH_L(A)$
\begin{align}\label{eq:def joint acting group}
			t_L(h):= \big(h,
			\Psi_{1,L}(h),
			\Psi_{2,L}(h),
			\Psi_{3,L}(h),
			\Psi_{4,L}(h) \big) \in \torus_L(A).
\end{align}

\subsection{$S$-arithmetic setup}

For any locally compact group $G$ we denote by $m_G$ a Haar measure on $G$. Furthermore, for a quotient $\lquot{\Gamma}{G}$ by a lattice $\Gamma$ we write $ m_{\lquot{\Gamma}{G}}$ for the unique $G$-invariant probability measure where $G$ acts via $g.(\Gamma g') = \Gamma g'g^{-1}$.

Given a set of places $S\subset V_\Q = \set{\infty, 2,3,5,\ldots,}$ of $\Q$ we set $\Q_S$ to be the restricted product of the $\Q_p$ for $p \in S \setminus \set{\infty}$ and $\Q_\infty = \R$ if $\infty\in S$. 
Furthermore, we set $\Z^S = \Z[\frac{1}{p}:p\in S \setminus \set{\infty}]$ and $\widehat{\Z} = \prod_{p \text{ prime}} \Z_p$.
We also write $\adele = \Q_{V_\Q}$ (resp.~$\adelef = \Q_{V_\Q\setminus\{\infty\}}$) for the ring of adeles (resp.~finite adeles).
 
Let $\G$ be a semisimple linear algebraic group defined over $\Q$ and let $K_{\mathrm{f}} = \prod_p K_p < \G(\adelef)$ be a compact open subgroup.
We say that $\G$ has class number one (with respect to $K_{\mathrm{f}}$) if 
\begin{align*}
\G(\adele) = \G(\Q)\, \G(\R)K_{\mathrm{f}}.
\end{align*}
If $\G$ has strong approximation (as is the case for $\SL_N$ for any $N \geq 2$), it has class number one with respect to any compact open subgroup of $\G(\adelef)$.
We also remark that the class number one property yields that $\G$ also has class number one with respect to any compact open $K_{\mathrm{f}}' \supset K_{\mathrm{f}}$.

In this paper we will consider the groups $\SU_2$, $\SL_2$ and products of these.
The integral structure on $\SL_2$ is the standard one inherited from viewing $\SL_2 \subset \Mat_2$ so that $\SL_2(\Z)$, $\SL_2(\Z_p)$ for a prime $p$, $\SL_2(\widehat{\Z})$ and so forth are made sense of. By the above, $\SL_2$ has class number one with respect to $\SL_2(\widehat{\Z})$.
The integral structure on $\SU_2$ is inherited from $\quat$: 
\begin{align*}
\SU_2(\Z) = \SU_2(\Q) \cap \quat(\Z),\ \SU_2(\Z_p) = \SU_2(\Q_p) \cap \quat(\Z_p)
\end{align*}
and so on. Note that $\SU_2(\Z_p)$ is maximal for $p \neq 2$ and that $\SU_2(\Z_2)$ is not maximal (cf.~Section~\ref{section:quaternions}).

\begin{lemma}\label{lem:SU2classnumber}
The group $\SU_2$ has class number one with respect to $\SU_2(\widehat{\Z}) = \prod_p \SU_2(\Z_p)$.
\end{lemma}

The lemma does not follow from strong approximation for $\SU_2$ as $\SU_2(\R)$ is compact.
While it is non-essential for the methods of this article (and indeed not true for other quaternion algebras), the class number one property simplifies the notations in the following which is why we prove it here.

\begin{proof}
For any prime $p$ let $K_p = \SU_2(\Z_p)$ and let $K_p'$ be the group of norm one units in the completion $\mathcal{O}_p = \mathcal{O} \otimes \Z_p$ of the Hurwitz order $\mathcal{O} \subset \quat(\Q)$. Note that $K_p =K_p'$ for all odd primes $p$. As the Hurwitz order is a principal ideal domain, $\SU_2$ has class number one with respect to $K_{\mathrm{f}}' = \prod_p K_p'$ -- see \cite[p.~29]{Linnikthmexpander}.

We deduce from this the analogous property for $\SU_2(\widehat{\Z}) = K_{\mathrm{f}}$. To this end, it suffices to show that for any $g \in K_{\mathrm{f}}'$ there exists $\gamma \in \mathcal{O}^\times$ with $\gamma g \in K_{\mathrm{f}}$. As $K_p = K_p'$ for $p >2$, it suffices to check this property at the prime $2$ which is what we do now by direct calculation. 
Let
\begin{align*}
g = \frac{a_0 + a_1 \ii + a_2 \jj + a_3 \kk}{2} \in K_2' \subset \mathcal{O}_2 \subset \quat(\Q_2)
\end{align*}
for $a_0,a_1,a_2,a_3 \in \Z_2$. By definition of the Hurwitz order either $2 \mid a_\ell$ for all $\ell \in \{0,\ldots,3\}$ or $2 \nmid a_\ell$ for all $\ell$. In the former case, $g \in \quat(\Z_2)$ and we are done. So suppose that the latter case holds and let
\begin{align*}
\gamma = \frac{\varepsilon_0+\varepsilon_1\ii + \varepsilon_2\jj + \varepsilon_3\kk}{2} \in \mathcal{O}^\times
\end{align*}
where $\varepsilon_\ell \in \{-1,1\}$ for every $\ell$. Then
\begin{align*}
\Tr(\gamma g) = \tfrac{1}{2}(a_0 \varepsilon_0 - a_1 \varepsilon_1 - a_2 \varepsilon_2 -a_3\varepsilon_3).
\end{align*}
It is not too hard to see that one can choose the signs $\varepsilon_\ell$ so that this (automatically integral) trace is divisible by $2$ which in turn implies that $\gamma g \in \quat(\Z_2)$.
\end{proof}

For any semisimple $\Q$-group $\G$ the subgroup $\G(\Q) < \G(\adele)$ is a lattice by a theorem of Borel and Harish-Chandra when $\G(\Q)$ is embedded diagonally (see for example \cite[Thm.~5.5]{platonov}). If $\G$ has an integral structure as in the above discussion (obtained for instance by embedding $\G$ into some $\SL_N$) and $S$ is a set of places containing the archimedean place, $\G(\Z^S) < \G(\Q_S)$ is a lattice when embedded diagonally.
Here and in the following we identify $\G(\Z^S)$ with its image under the diagonal embedding.

We will use the $S$-arithmetic extensions
\begin{align*}
&X_{1,S} = X_{2,S} = \lquot{\SU_2(\Z^S)}{\SU_2(\Q_S)},\\
&X_{3,S} = X_{4,S} = X_{5,S} = X_{6,S} = \lquot{\SL_2(\Z^S)}{\SL_2(\Q_S)}.
\end{align*}
of the real quotients $\lquot{\SU_2(\Z)}{\SU_2(\R)}$ and $\lquot{\SL_2(\Z)}{\SL_2(\R)}$.
For simplicity we write $X_{n,\infty} = X_{n,\set{\infty}},\ X_{n,\mathbb{A}} = X_{n,V_\Q}$ for $n=1,\ldots,6$.
The class number one property for $\SU_2$ from Lemma~\ref{lem:SU2classnumber} implies that we have for any two sets $S' \subset S$ of places containing the archimedean place a map
\begin{align*}
X_{1,S} \to X_{1,S'}
\end{align*}
equivariant under $\SU_2(\Q_{S'})$ by taking the quotient with $\prod_{p \in S \setminus S'}\SU_2(\Z_p)$ from the right. Similarly, one obtains maps $X_{n,S} \to X_{n,S'}$ for $n >1$.

We also define the $\Q$-group
\begin{align*}
\mathbb{G} = \SU_2 \times \SU_2 \times \SL_2 \times \SL_2 \times \SL_2 \times \SL_2
\end{align*}
which has class number one with respect to $\G(\widehat{\Z}) = \SU_2(\widehat{\Z})^2 \times \SL_2(\widehat{\Z})^4$ by the above.
We equip any subgroup $\mathbb{M} < \G$ with the integral structure inherited from $\G$ so that for instance $\mathbb{M}(\Z) = \G(\Z) \cap \mathbb{M}(\Q)$.
We set
\begin{align*}
X_S = \lquot{\G(\Z^S)}{\G(\Q_S)} = \prod_{n=1}^6 X_{n,S}
\end{align*}
for any set of places $S$ containing the archimedean place.

\subsection{Toral packets and the dynamical result}\label{section:formulation dyn.result}

Let $L \in \Gr(\Q)$ be a rational plane. 
Following the discussion in Section \ref{section:acting group} we will study orbits in $X_S$ under the $\Q_S$-points of the $\Q$-group $\torus_L$.
First, let us consider the compact adelic orbit $\G(\Q)\torus_L(\adele) \subset X_\adele$ (it is compact as $\torus_L$ is anisotropic over $\Q$). The projection $\packet_L$ of this orbit to $X_S$, where $S= \set{\infty,p,q}$ and $p$, $q$ are two distinct odd primes is an example of a \textit{homogeneous toral packet}, see \cite[Sec.~4,5]{Dukeforcubic} for this terminology.


In order to normalize the behavior on the real quotient $X_\infty$ we 
choose for every plane $L\in \Gr(\Q)$ an element $\ell_L \in \G(\R)$ such that
\begin{align*}
\ell_L^{-1} \torus_L(\R) \ell_L < K:=  \Stab_{\SU_2^2(\R)}(\langle 1, \ii \rangle) \times \SO_2(\R)^4
\end{align*}
and also consider the pushed packet $\ell_L^{-1}.\packet_L$.
Furthermore, let $\mu_L$ be the pushforward of the normalized Haar measure on the shifted orbit $\G(\Q)\torus_L(\adele)\ell_L$ to $\ell_L^{-1}.\packet_L$ (under the natural projection to $X_S$).

\ifsquarefree
In the following we shall call a sequence $(L_n)_n$ of rational planes in $\R^4$ admissible if for every $n$ the discriminant $\disc(L_n(\Z)) = D_n$ satisfies the assumptions in Theorem~\ref{thm:main} for the fixed primes $p,q$ and if $D_n \to \infty$ as $n \to \infty$.
\else
In the following we shall call a sequence $(L_n)_n$ of rational planes in $\R^4$ \textit{admissible} (with respect to $p,q$) if the following conditions are satisfied:
\begin{itemize}
\item For every $n$ the discriminant $\disc(L_n) = D_n$ satisfies the assumptions in Theorem~\ref{thm:main} for the fixed primes $p,q$.
\item As $n$ goes to infinity
\begin{align*}
Q(\widetilde{a}_1(L_n)),\ Q(\widetilde{a}_2(L_n)),\ \frac{Q(\widetilde{a}_1(L_n))Q(\widetilde{a}_2(L_n))}{\disc(L_n)}
\end{align*}
go to infinity. Here, $\widetilde{a}_1(L_n)$, $\widetilde{a}_2(L_n)$ denote the primitive vectors in the half-lines $\Q_+ a_1(L_n)$, $\Q_+ a_2(L_n)$ as in Proposition \ref{prop:geometric-nonsqfree}.
\end{itemize}
We remark that the second condition is automatically satisfied if the discriminants $D_n$ are square-free or more generally if the square-free part of $D_n$ goes to infinity.

The following implies our main theorem (Theorem \ref{thm:main}).

\begin{theorem}[Equidistribution of packets]\label{thm:orbit version}
Let $(L_n)_n$ be an admissible sequence of rational planes. Then 
$\mu_{L_{n}} \to m_{X_S}$ as $n \to \infty$.
\end{theorem}

%

\section{Proof of the main theorem from the dynamical version}\label{section:prooffromdynstatement}

In this section we show that Theorem \ref{thm:orbit version} does indeed imply Theorem \ref{thm:main}. 
For this we will use adelic orbits of the form $\G(\Q)\torus_L(\adele)\subset X_\adele$ for some $L \in \mathcal{R}_D$ in order to generate additional points in $\mathcal{J}_D$ as in Theorem \ref{thm:main} from one such point (see for instance \cite[Thm. 8.2]{platonov}).

Let $L\in \Gr(\Q)$ be a rational plane. 
Recall that the class number
\begin{align}\label{eq:classin4}
\left|\lrquot{\torus_L(\Q)}{\torus_L(\adele)}{\torus_L(\R \times \widehat{\Z})}\right|
\end{align}
of the group $\torus_L$ is finite (c.f. \cite[Thm.~5.1]{platonov}) and that $\G = \SU_2^2 \times \SL_2^4$ has class number one (i.e. $\G(\Q)\G(\R\times\widehat{\Z}) = \G(\BA)$). We may thus write
\begin{align}\label{eq:decomporbit}
\G(\Q)\torus_L(\mathbb{A}) = \bigsqcup_{\rho \in \mathcal{M}_L} \G(\Q) \rho \torus_L(\R \times \widehat{\Z})
\end{align}
for a finite set $\mathcal{M}_L \subset \G(\R \times \widehat{\Z})$ of representatives.
Note that by construction the cardinality of $\mathcal{M}_L$ is the class number of $\torus_L$
in \eqref{eq:classin4}. 

We now construct points in $\lquot{\SU_2^2(\Z)}{\mathcal{J}_D}$ using the above stabilizer orbit, where $\SU_2^2(\Z)$ acts naturally on $\mathcal{R}_D$ not affecting the other components.
We will implicitly identify $\Gr(\Q)$ with the image in $\Gr(\adele)$ under the injective map
\begin{align*}
L \in \Gr(\Q) \mapsto L \otimes_\Q \adele \in \Gr(\adele)
\end{align*}
and analogously for binary quadratic forms with rational coefficients.
In the following, points $g\in \G(\adele)$ will sometimes be written as $g= (g_1,\ldots,g_6)$ for the corresponding elements $g_1,g_2\in \SU_2(\adele)$ and $g_3,g_4,g_5,g_6 \in \SL_2(\adele)$.

\begin{proposition}[Generating integer points]\label{prop:generating intpts}
Let $L\in \mathcal{R}_D$ be a rational plane and let $g\in \G(\R \times \widehat{\Z})$ with $\G(\Q)g \in \G(\Q)\torus_L(\mathbb{A})$.
\begin{enumerate}[(i)]
\item The plane $L_g = (g_1,g_2).L$ is rational and has the same discriminant as $L$.
Furthermore,
\begin{align*}
Q(\widetilde{a}_1(L_g))= Q(\widetilde{a}_1(L)), \quad Q(\widetilde{a}_2(L_g)) = Q(\widetilde{a}_2(L)).
\end{align*}
\item The quadratic form $g_3.q_L$ is an integral binary form and is equivalent to~$q_{L_g}$. The analogous statement also holds for $q_{L^\perp}$, $q_{a_1(L)}$ and $q_{a_2(L)}$.
\end{enumerate}
It follows that for any $\rho\in \mathcal{M}_L$ we obtain a corresponding point in the set $\mathcal{J}_D$ and in particular a rational plane $L_\rho$.
\end{proposition}

\begin{proof}
Choose $t \in \torus_L(\adele)$ and $\gamma \in \G(\Q)$ such that $\gamma t = g$ and write $t = t_L(h)$ as in~\eqref{eq:def joint acting group} for some $h\in \mathbb{H}_L(\adele)$.

To see (i) notice first that $L_g$ is rational as
\begin{equation}\label{eq:Lggammaplane}
L_g = (g_1,g_2).L = (\gamma_1,\gamma_2)h.L = (\gamma_1,\gamma_2).L .
\end{equation}
To compute the discriminant let $v_1$, $v_2$ be a $\Z$-basis of $L(\Z)$. 
Recall that for any group $G$ acting on a module $V$ there is a natural corresponding action of $G$ on $\bigwedge^2 V$ given by $g.(w_1\wedge w_2) = (g.w_1)\wedge (g.w_2)$ for $w_1,w_2 \in V$. We observe
\begin{align*}
\nicewedge{2} (\R \times \widehat{\Z})^4 \ni (g_1,g_2).(v_1 \wedge v_2) 
&= (\gamma_1h_1,\gamma_2h_2).(v_1 \wedge v_2) \\
&= (\gamma_1,\gamma_2).(v_1 \wedge v_2) \in \nicewedge{2} \Q^4
\end{align*}
since $h=(h_1,h_2)$ preserves $L$. As $\Q \cap (\R \times \widehat{\Z}) = \Z$ it follows that $(g_1,g_2).(v_1 \wedge v_2)$ is integral.

To show primitivity we write $g_{1,p},g_{2,p} \in \SU_2(\Zp)$ for the $p$-adic coordinate of $g_1,g_2$. Then
\begin{align*}
\norm{(g_1,g_2).(v_1 \wedge v_2)}_p = \norm{(g_{1,p},g_{2,p}).(v_1 \wedge v_2)}_p = \norm{v_1 \wedge v_2}_p = 1
\end{align*}
for all $p$ (for the maximum norm\footnote{To define the norm $\norm{\cdot}_p$ one chooses a $\Z_p$-basis of $\bigwedge^2\Z_p^4$ and takes the maximum of $p$-adic absolute values of the coordinates in this basis.} $\norm{\cdot}_p$
on the wedge product).
Therefore, $(g_1,g_2).(v_1 \wedge v_2)$ is primitive. 
Hence, the Euclidean norm of $(g_1,g_2).(v_1 \wedge v_2)$ is exactly the discriminant of $L_g$. As $g_1,g_2 \in \SU_2(\adele)$ the former is the Euclidean norm of $v_1 \wedge v_2$ which in turn is the discriminant of $L$.

It remains to show the equality for the lengths of the primitive vectors.
By the equivariance in Proposition \ref{prop:splitting map}, we have $a_1(L_g) = g_1.a_1(L)$. 
As above, it follows from considering every prime that $g_1.\widetilde{a}_1(L)$ as a multiple of $a_1(L_g) = g_1.a(L)$ is primitive.
Thus, $Q(\widetilde{a}_1(L_g)) = Q(g_1.\widetilde{a}_1(L)) = Q(\widetilde{a}_1(L))$ as desired. The argument for $\widetilde{a}_2(L_g)$ is analogous.

For (ii) we begin by showing that $g_3.q_L$ is an integral form. To this end, we just note that
\begin{align*}
g_3.q_L = \gamma_3 t_3.q_L = \gamma_3.q_L
\end{align*}
where the form on the left has coefficients in $\R\times \widehat{\Z}$ and the form on the right has coefficients in $\Q$.
The analogous argument shows that $g_4.q_{L^\perp}$ is integral. 

We now wish to show that $\gamma_3.q_L$ and $q_{L_g}$ are equivalent.
Recall that (as is implicit in the definition of $q_L$ resp.~$q_{L_g}$) we have chosen a matrix $A_{1,L}\in \SL_4(\Z)$ (resp.~$A_{1,L_g}\in \SL_4(\Z)$) so that the first two rows form a basis of $L(\Z)$ (resp. $L_g(\Z)$). 

Now notice that
\begin{align*}
\gamma_3.q_L(u)
&= Q((u\gamma_3,0)A_{1,L})
= Q((u\gamma_3,0)A_{1,L} P(\gamma_1,\gamma_2)^{-1} ) \\
&= Q\big((u\gamma_3,0)A_{1,L} P(\gamma_1,\gamma_2)^{-1} A_{1,L_g}^{-1}A_{1,L_g} \big)
= Q((u,0)C A_{1,L_g})
\end{align*}
where
\begin{align*}
C = \begin{pmatrix}
\gamma_3 & 0 \\ 
0 & \id
\end{pmatrix} A_{1,L} P(\gamma_1,\gamma_2)^{-1} A_{1,L_g}^{-1} \in \Mat_4(\Q).
\end{align*}
Observe that $C$ (acting on row vectors) maps $\R^2 \times \set{(0,0)}$ to itself. Indeed, $A_{1,L}$ maps $\R^2\times \set{(0,0)}$ to $L$, by \eqref{eq:Lggammaplane} $P(\gamma_1,\gamma_2)^{-1}$ maps $L$ to $L_g$, 
and finally $A_{1,L_g}^{-1}$ maps $L_g$ back to $\R^2\times \set{(0,0)}$.
In other words, $C$ is of the block form $\big(\scriptsize\arraycolsep=0.3\arraycolsep\ensuremath{
\begin{array}{c|c}
\ast & 0  \\
\hline
\ast & \ast
\end{array}}
\big)$. Denoting by $\pi_1(C)$ the projection onto the upper left block (as in Section \ref{section:def shapes}) the above calculation can be summarized as $\gamma_3.q_L(u) = \pi_1(C).q_{L_g}(u)$.

We want to show that $\pi_1(C) \in \GL_2(\Z)$.
To see that $\pi_1(C)$ is integral, we use the third component of $\gamma t =g$ 
together with~\eqref{eq:def joint acting group} to obtain
\begin{align*}
\pi_1(C) 
&= \gamma_3\pi_1\left( A_{1,L} P(\gamma_1,\gamma_2)^{-1} A_{1,L_g}^{-1} \right) \\
&= g_3 \underbrace{\pi_1\left(A_{1,L} P(h^{-1}) A_{1,L}^{-1}\right)}_{=t_3^{-1}=\Psi_{1,L}(h^{-1})}
\pi_1\left( A_{1,L} P(\gamma_1,\gamma_2)^{-1} A_{1,L_g}^{-1} \right) \\
&= g_3 \pi_1\left(A_{1,L} P(h^{-1}) P(\gamma_1,\gamma_2)^{-1} A_{1,L_g}^{-1} \right) = g_3 \pi_1\left(A_{1,L} P(g_1,g_2)^{-1} A_{1,L_g}^{-1} \right)\\
&\quad \in\Mat_2(\Q)  \cap \Mat_{2}(\R\times \widehat{\Z}) = \Mat_2(\Z).
\end{align*}
 We note that $A_{1,L} P(g_1,g_2)^{-1} A_{1,L_g}^{-1}\in\GL_4(\R\times\widehat{\Z})$ 
 is a block matrix with $0$ in the right top $2$-by-$2$ block, which implies that the upper left block
 \[
 \pi_1\left(A_{1,L} P(g_1,g_2)^{-1} A_{1,L_g}^{-1} \right) 
 \]
 and hence also $\pi_1(C)=g_3\pi_1\left(A_{1,L} P(g_1,g_2)^{-1} A_{1,L_g}^{-1} \right)\in\GL_2(\R\times\widehat{\Z})$ 
 are invertible. 
This implies that $\pi_1(C) \in \GL_2(\Z)$ as claimed.

This concludes (ii) for $q_L$.
The argument to verify that $g_4.q_{L^\perp}$ is equivalent to $q_{L_g^\perp}$ is completely analogous. 

For the remaining copies one can use the equivariance property in Proposition~\ref{prop:splitting map} to reduce the statement to \cite[Prop.~3.2]{AES3D}.
\end{proof}

%

For $L \in \mathcal{R}_D$ we have by definition of the generated planes in Proposition \ref{prop:generating intpts} that $L_{ \gamma\rho t} = \gamma.L_\rho$ for $\rho\in \mathcal{M}_L$, $t \in \torus_L(\R \times \widehat\Z)$ and $\gamma\in \SU_2^2(\Z)$ and the analogous statement holds for the binary forms.
The function
\begin{align}\label{eq:proj to J_D}
\G(\Q)\torus_L(\mathbb{A})
\to\lquot{\SU_2^2(\Z)}{\mathcal{J}_D}
\end{align}
which maps $\G(\Q)\rho t$ for $\rho\in \mathcal{M}_L$ and $t \in \torus_L(\R \times \widehat{\Z})$ to $\SU_2^2(\Z).L_\rho$ and the attached CM points is thus well-defined (i.e.~independent of the choice of $\mathcal{M}_L$).

Call a plane $L \in \mathcal{R}_D$ \emph{exceptional} if the finite group $\BH_{L}'(\Z)$ strictly contains $\{(\pm1,\pm1)\}$ where $\BH_{L}' = \{g \in \SU_2^2: g.L \subset L\}$ denotes the full stabilizer subgroup.

\begin{proposition}[On the collection of generated points]\label{prop:collection intpts}
Let $L\in \mathcal{R}_D$ for $D \in \mathbb{D}$.
\begin{enumerate}[(a)]
\item Every fiber of the map in \eqref{eq:proj to J_D} is a union of at most two $\torus_L(\R \times \widehat{\Z})$-orbits.
\item
Amongst the fibers in \eqref{eq:proj to J_D} of non-exceptional planes the number of $\torus_L(\R \times \widehat{\Z})$-orbits is equal.
\item 
The set of exceptional planes in $\mathcal{R}_D$ has size as most $ D^{\frac{1}{2}+o(1)}$.
\item The volume of the $\torus_L(\R \times \widehat{\Z})$-orbit through $\G(\Q)\rho$ is independent of $\rho \in~\mathcal{M}_L$.
\end{enumerate}
\end{proposition}

We will call the image of \eqref{eq:proj to J_D} the \emph{packet} attached to $L$ and denote it by $\mathcal{P}(L)$.
Clearly, if two points lie on the same $\torus_L(\R \times \widehat{\Z})$-orbit, they give rise to the same point in $\lquot{\SU_2^2(\Z)}{\mathcal{J}_D}$.
By Proposition~\ref{prop:generating intpts}(ii) it suffices to prove the analogous statement to (a) and (b) for the fibers of the map
\begin{align}\label{eq:projection to planes}
\SU_2^2(\Q)\BH_L(\mathbb{A})
\to\lquot{\SU_2^2(\Z)}{\mathcal{R}_D}.
\end{align}
To this end, we will use the following lemma.

\begin{lemma}\label{lem:fibertechnical}
Define the group
\begin{align*}
F_L = \setc{h' \in \BH_L'(\R \times \widehat{\Z})}{\SU_2^2(\Q)\BH_L(\adele)h' = \SU_2^2(\Q)\BH_L(\adele)}.
\end{align*}
Then the fiber of the map \eqref{eq:projection to planes} through any point $x \in \SU_2^2(\Q)\BH_L(\mathbb{A})$ is equal to $xF_L$.
If we denote by $n(x)$ the number of $\BH_L(\R \times \widehat{\Z})$-orbits in the fiber $xF_L$, we have
\begin{align}\label{eq:no of fibers}
n(x) \leq [F_{L}:\BH_{L}(\R \times \widehat{\Z})].
\end{align}
\end{lemma}

\begin{proof}[Proof of Lemma~\ref{lem:fibertechnical}]
By definition of $F_L$ and the generated planes in Proposition~\ref{prop:generating intpts}, the fiber through a point $x \in \SU_2^2(\Q)\BH_L(\adele)$ contains $xF_L$.
Conversely, if $y \in \SU_2^2(\Q)\BH_L(\adele)$ is another point in the fiber through $x$, we write $x = \SU_2^2(\Q)g_x$ and $y = \SU_2^2(\Q)g_y$ for elements $g_x,g_y \in \SU_2^2(\R \times \widehat{\Z})$ and replace $g_y$ if necessary so that $g_x.L = g_y.L$.
Then $h' = g_x^{-1}g_y$ is an element of $\BH_L'(\adele) \cap \SU_2^2(\R \times \widehat{\Z})= \BH_L'(\R \times \widehat{\Z})$ and satisfies $xh' = y$. 
Moreover, since $\BH_L < \BH_L'$ is normal
\begin{align*}
\SU_2^2(\Q)\BH_L(\adele)h' =  x \BH_L(\adele) h' =  xh' \BH_L(\adele) = y \BH_L(\adele) = \SU_2^2(\Q)\BH_L(\adele)
\end{align*}
and so $h' \in F_L$ as claimed.

For the second part, write $x = \SU_2^2(\Q)g$ for $g \in \SU_2^2(\R \times \widehat{\Z})$. Then 
\begin{align*}
n(x) 
&= \left|\lrquot{g^{-1}\SU_2^2(\Q)g \cap F_L}{F_L}{\BH_L(\R\times \widehat{\Z})}\right| \\
&= \left|\lrquot{\SU_2^2(\Q) \cap gF_Lg^{-1}}{gF_Lg^{-1}}{g\BH_L(\R\times \widehat{\Z})g^{-1}}\right|.
\end{align*}
Now note that $g\BH_L(\R\times \widehat{\Z})g^{-1} = \BH_{g.L}(\R \times \widehat{\Z})$ by definition of $\BH_L$ and integrality of $g$.
Also, we have that $gF_Lg^{-1} = F_{g.L}$. 
Indeed, for any $h' \in F_L$ the element $gh'g^{-1}$ is in $g\BH_L'(\R \times \widehat{\Z})g^{-1} = \BH_{g.L}'(\R \times \widehat{\Z})$ and
\begin{align*}
\SU_2^2(\Q) \BH_{g.L}(\adele)gh'g^{-1}
&= \SU_2^2(\Q) g\BH_L(\adele)g^{-1}gh'g^{-1} = x \BH_L(\adele)h'g^{-1} \\
&= \SU_2^2(\Q) \BH_L(\adele) h'g^{-1} = \SU_2^2(\Q) \BH_L(\adele)g^{-1}\\
&= \SU_2^2(\Q)g\BH_L(\adele)g^{-1} = \SU_2^2(\Q) \BH_{g.L}(\adele).
\end{align*}
In particular, $gF_Lg^{-1} \cap \SU_2^2(\Q) \subset \BH_{g.L}'(\Z)$ and using normality of $\BH_L$ in $\BH_L'$ equality holds.
We conclude that
\begin{align}\label{eq:indexestimate}
\begin{array}{l}
n(x) =  \left|\lrquot{\BH_{g.L}'(\Z)}{F_{g.L}}{\BH_{g.L}(\R\times \widehat{\Z})}\right| \\
\qquad\ \leq [F_{g.L}:\BH_{g.L}(\R \times \widehat{\Z})]
= [F_{L}:\BH_{L}(\R \times \widehat{\Z})].
\end{array}
\end{align}
which proves the second part of the lemma.
\end{proof}

\begin{proof}[Proof of Proposition~\ref{prop:collection intpts}]
To prove the claim in (a) notice first that by \eqref{eq:no of fibers} it suffices to estimate the index $[F_{L}:\BH_{L}(\R \times \widehat{\Z})]$.
To do the latter, observe that if $h' \in F_{L}$ satisfies $h'_{\infty} \in \BH_{L}(\R)$ then $h' \in \BH_{L}(\R \times \widehat{\Z})$.
Indeed, we can write the point $\SU_2^2(\Q)h'$ as $\SU_2^2(\Q)h$ for $h \in \BH_{L}(\adele)$ and so there is $\gamma \in \SU_2^2(\Q)$ with $\gamma h = h'$. Thus, $\gamma = h'h^{-1} \in \BH_{L}'(\Q)$ and since $h'_\infty \in \BH_{L}(\R)$ we further have $\gamma = h'_\infty h_\infty^{-1} \in \BH_{L}(\R)$ so that $\gamma \in \BH_{L}(\Q)$.
This proves that $h' = \gamma h \in \BH_{L}(\adele)\cap \BH_{L}'(\R \times \widehat{\Z}) = \BH_{L}(\R \times \widehat{\Z})$.
As the index of $\BH_{L}(\R)$ in $\BH_{L}'(\R)$ is $2$, $[F_{L}:\BH_{L}(\R \times \widehat{\Z})]\leq 2$ and \eqref{eq:no of fibers} implies the claim in (a).

To see (b), note that if the plane $g.L$ generated by $x$ is non-exceptional, we have $\BH_{g.L}'(\Z) = \{(\pm 1,\pm 1)\} = \BH_{g.L}(\Z)$ which is contained in the center.
In this case, \eqref{eq:indexestimate} implies that equality holds in \eqref{eq:no of fibers}.
Since the right hand side in \eqref{eq:no of fibers} is independent of the point $x= \SU_2^2(\Q)g$, this proves (b).

%
%

For (c) it suffices to show that for any $g \in \SU_2^2(\Z) \setminus \{(\pm1,\pm1)\}$ the number of planes $L' \in \mathcal{R}_D$ with $g.L' = L'$ is of size $D^{\frac{1}{2}+o(1)}$.
We distinguish two cases.
\begin{itemize}
\item There exists a two-dimensional irreducible subspace $L_0 \subset \Q^4$ for the action of $g$. In this case, $g$ must act as a rotation on $L_0$ and hence $g_1.b_1 = b_1$ and $g_2.b_2 = b_2$ where $b_i = a_i(L_0)$.
Suppose without loss of generality that $g_1 \neq \pm1$. In this case, its action on pure quaternions by conjugation is non-trivial and hences fixes a unique rational line.
Therefore, we have that $a_1(L') \in \Q b_1$ for any $L'\in \mathcal{R}_D$ with $g.L' \subset L'$ and in particular $a_1(L') = \pm a_1(L'')$ for any two such $L',L''$ as the norms agree.
Using the Klein map the only 'free' choice is thus in the vector $a_2(\cdot)$. The set of all such $L'$ must be thus bounded by the number of vectors $v \in \quat^0(\Z)$ with $\Nr(v) = D$ which is $D^{\frac{1}{2}+o(1)}$.
\item 
Any irreducible subspace for the action of $g$ has dimension $1$.
Thus, the action of $g$ on $\Q^4$ is diagonalizable and has determinant one. As $g \not\in \{(\pm1,\pm1)\}$, there must be exactly one two-dimensional subspace $L_0$ such that $g|_{L_0} = \id_{L_0}$ and $g|_{L_0}^\perp = -\id_{L_0^\perp}$.
If $L'$ is a plane with $g.L' \subset L'$ then either $L'\in \{ L_0,L_0^\perp\}$ or $\dim(L_0 \cap L') = \dim(L_0^\perp \cap L') =1$.
In the latter case, $g$ is a reflection when restricted to $L'$ and therefore (see~\eqref{eq:bilin&antisymm})
$g_1.a_1(L')= -a_1(L')$ and $g_2.a_2(L')= -a_2(L')$. 
This shows that $a_1(L') \perp a_1(L_0)$ and $a_2(L') \perp a_2(L_0)$ as $g_1,g_2$ act by orthogonal transformations and preserve $a_1(L_0)$ resp.~$a_2(L_0)$.
So this case corresponds to counting the number of representations of $D$ by an integral binary form, which is of order $D^{o(1)}$ (cf.~\cite[p.~372]{cassels}).
\end{itemize}

Thus, the number of such pairs $(a_1(L'),a_2(L'))$ is $D^{\frac{1}{2}+o(1)}$.
This proves (c).

For (d) observe that for any $\rho \in \mathcal{M}_L$
\begin{align*}
\Stab_{\torus_L(\R \times \widehat{\Z})}(\G(\Q)\rho)
&= \Stab_{\torus_L(\R \times \widehat{\Z})}(\G(\Q)t)
= \Stab_{\torus_L(\R \times \widehat{\Z})}(\G(\Q))\\
&= \torus_L(\R \times \widehat{\Z}) \cap \G(\Q) =  \torus_L(\Z)
\end{align*}
where $t\in \torus_L(\mathbb{A})$ was chosen with $\G(\Q) \rho = \G(\Q)t$ and where we used that $\torus_L$ is abelian in the second equality.
%
\end{proof}

\begin{remark}[Decomposing into packets]\label{rem:decomp into packets}
Note that given two planes $L,L'\in \mathcal{R}_D$ the question whether or not $L'$ can be generated from $L$ as in Proposition \ref{prop:generating intpts} is equivalent to asking whether or not $\gamma \in \SU_2^2(\Q)$ and $g \in \SU_2^2(\R\times \widehat{\Z})$ exist with $\gamma.L = g.L = L'$.
This defines an equivalence relation and hence we can decompose the set $\lquot{\SU_2^2(\Z)}{\mathcal{J}_D}$ into disjoint packets.
\end{remark}

Let $L \in \mathcal{R}_D$ for $D \in \mathbb{D}$\ifsquarefree\ square-free\fi.
We now project the set in \eqref{eq:decomporbit} onto the $S=\set{\infty,p,q}$-arithmetic quotient $X_S$ to obtain the packet of orbits
\begin{align*}
\packet_L = \bigsqcup_{\rho \in \mathcal{M}_L} \G(\Z^S) (\rho_\infty,\rho_p,\rho_q) \torus_L(\R \times \Z_p \times \Z_q)
\subset X_S
\end{align*}
invariant under $\torus_L(\Q_S)$.
The union is still disjoint: If $x = \G(\Q)t$, $x' = \G(\Q)t'$ have the same image in $\packet_L$ there exists $g\in \G(\prod_{\sigma\not\in S}\Z_\sigma)$ and $\gamma\in \G(\Q)$ with $\gamma t = t' g$.
In particular, this equation at the infinite place yields $\gamma \in \torus_L(\Q)$ so that $g \in \G(\prod_{\sigma\not\in S}\Z_\sigma) \cap \torus(\adele) = \torus_L(\prod_{\sigma\not\in S}\Z_\sigma)$ as desired.
Note that the projection~\eqref{eq:proj to J_D} factors through the projection to $\packet_L$.

%


\begin{proof}[Proof of Theorem \ref{thm:main} assuming Theorem \ref{thm:orbit version}]
For any plane $L\in \Gr(\Q)$ choose an element $\ell_L \in \BG(\R)$ with $\ell_L^{-1} \torus_L(\R) \ell_L < K$.
By Theorem \ref{thm:orbit version} we know that $\mu_{L_n} \to m_{X_S}$ along any admissible sequence $(L_n)_n$.
In particular, the convergence holds after pushforward to the real quotient $X_\infty$ and to
\begin{align}\label{eq:packet}
\rquot{X_\infty}{K} = \big(\lquot{\SU_2^2(\Z)}{\Gr(\R)}\big) \times \mathcal{X}_2^4 =: Y.
\end{align}
For any $D \in \mathbb{D}$ and $L \in \mathcal{R}_D$ we let $\nu_{L}$ be the normalized sum of Dirac measures over the packet $\mathcal{P}(L) \subset \lquot{\SU_2^2(\Z)}{\mathcal{J}_D}$ attached to $L$ -- see Proposition~\ref{prop:collection intpts}.
By part (b)-(d) of that proposition and by Corollary~\ref{cor:count}, the pushforward measure on $\mu_L$ to~$Y$ and $\nu_{L}$ differ for $f\in C_b(Y)$ by at most $\norm{f}_\infty D^{-\frac{1}{2}+o(1)}$ 
as the weights of these measures on $\mathcal{P}(L)$ need to be changed by at most $1$ on exceptional planes only.
So the measures $\nu_{L_n}$ are also equidistributing as $n \to \infty$.

Since $\mathcal{J}_D$ is $\SU_2^2(\Z)$-invariant, by a similar argument the sets $\mathcal{J}_D$ equidistribute in $\Gr(\R) \times \mathcal{X}_2^4$ if and only if the sets $\lquot{\SU_2^2(\Z)}{\mathcal{J}_D}$ equidistribute in $Y$. 
We claim that the latter is true.

To see this, write the set $\lquot{\SU_2^2(\Z)}{\mathcal{J}_D}$ for $D \in \BD$ with 
\begin{align*}
\legendre{-D}{p} = \legendre{-D}{q} = 1
\end{align*}
as a disjoint union of packets -- see Remark~\ref{rem:decomp into packets}.
Let $\mathscr{G}_D$ be the union of the packets attached to planes $L_D\in \mathcal{R}_D$ with $Q(\widetilde{a}_1(L_D))\geq D^{\frac{2}{3}}$ and $Q(\widetilde{a}_2(L_D)) \geq D^{\frac{2}{3}}$.
Recall from Corollary~\ref{cor:count} that $\lquot{\SU_2^2(\Z)}{\mathcal{J}_D}$ is of size $D^{1+o(1)}$.
Also, observe that the number of pairs of primitive integer points where one of the points has quadratic value at most $D^{\frac{2}{3}}$ is of size $D^{\frac{1}{2}+o(1)}D^{\frac{1}{3}+o(1)} = D^{\frac{5}{6}+o(1)}$ (see the proof of Corollary~\ref{cor:count}).
Thus, the sets $\lquot{\SU_2^2(\Z)}{\mathcal{J}_D}$ equidistribute if and only if the subsets $\mathscr{G}_D$ equidistribute.

However, by Theorem \ref{thm:orbit version} and the above discussion any sequence of packets $\mathcal{P}(L_D) \subset \mathscr{G}_D$ equidistributes since 
\begin{align*}
\frac{Q(\widetilde{a}_1(L_D))Q(\widetilde{a}_2(L_D))}{D} \geq D^{\frac{1}{3}}
\end{align*}
(which implies admissibility of the underlying planes $L_D$).
This implies that the sets $\mathscr{G}_D$ equidistribute (as finite unions of equidistributing subsets) and hence concludes the proof of Theorem \ref{thm:main}.
\end{proof}

%
\section{Proof of Theorem \ref{thm:orbit version}}\label{section:proofdynresult}

Let $(L_j)_j$ be a sequence of admissible planes so that the sequence of measures $\mu_j = \mu_{L_j}$ converges to a measure $\mu$. 
We want to show that $\mu$ is the normalized Haar measure $m_{X_S}$ on $X_S = \lquot{\G(\Z^S)}{\G(\Q_S)}$. 
Since the limit $\mu$ is then independent of the arbitrarily chosen sequence $(L_j)_j$, this implies Theorem \ref{thm:orbit version}.
To prove that $\mu = m_{X_S}$ we will use the fact that the pushforward of $\mu$ under all projections to the factors in $X_S$ is the Haar measure on the respective factor and then apply a joinings classification of Lindenstrauss and the second named author \cite{EL-joining2}.

\begin{proposition}[Limit measures are joinings]\label{prop:individualequidistr}
The push-forward of $\mu$ under any projection $X_S \to X_{i,S}$ for $i=1,\ldots,6$ is the Haar measure on $X_{i,S}$. In particular, $\mu$ is a probability measure.
\end{proposition}

This proposition should be seen as a version of Duke's Theorem \cite{duke88} or its strengthenings to subcollections (see specifically \cite{harcosmichelII}).
Indeed, as we project on a factor, we obtain subsets of toral packets for a form of $\SL_2$.
Since we assume splitting conditions, individual equidistibution may be proven by means of Linnik's ergodic method \cite{linnik}, since the total volume of the packets we consider is large enough in each factor.
In the following we give more details to this reasoning.

In the case $i=1,2$, the proposition amounts to showing that $\SU_2(\Q)\BH_{a_i(L_j)}(\adele)\ell_{L_j,i}$ projected to $X_{i,S}$ equidistributes to the Haar measure.
This is a version of Duke's theorem for the form $\SU_2$ of $\SL_2$ -- see for example \cite[Thm.~4.6]{Dukeforcubic}, \cite{harcosmichelII} or \cite[Thm.~7.1]{W18} (as we are assuming splitting conditions).
We remark that as we consider the simply connected cover $\SU_2$ (instead of the adjoint group $\SO_3$), we are considering not the whole Picard group attached to the quadratic order defined by $a_i(L_j)$, but rather the set of squares in it \cite[Lemma~7.2]{W18}.
Since the $2$-torsion of the Picard group has size $D^{o(1)}$ (see e.g.~\cite[p.~342]{cassels}), the squares form a subgroup of size $D^{\frac{1}{2}+o(1)}$ which is why one can still apply Linnik's ergodic method (as is done in \cite[Thm.~7.1]{W18}).

For $i=3$ (and similarly $i=4$) the image under projection to $X_{i,\adele}$ could potentially be 'small' in comparison to $\SL_2(\Q)\BH_{q_{L_j}}(\adele)\ell_{L_j,i}$. The following lemma rules this out and hence one may apply the same theorems cited above.

\begin{lemma}[About the image]\label{lemma:image of isogenies}
Let $\BK$ be a field with $\operatorname{char}(\BK) = 0$ and let $L \in \mathcal{R}_D$.
The maps $\Psi_{1,L}:\BH_L(\BK) \to \BH_{q_L}(\BK)$ and $\Psi_{2,L}:\BH_L(\BK) \to \BH_{q_{L^\perp}}(\BK)$ are surjective.
In particular, the natural map induced by $\Psi_{1,L}$
\begin{align*}
\lrquot{\BH_L(\Q)}{\BH_L(\adele)}{\BH_L(\R \times \widehat{\Z})}
\to \lrquot{\BH_{q_L}(\Q)}{\BH_{q_L}(\adele)}{\BH_{q_L}(\R \times \widehat{\Z})}
\end{align*}
is surjective (and similarly for $\Psi_{2,L}$).
\end{lemma}

We remark that the isogenies $\Psi_{3,L}$ and $\Psi_{4,L}$ are not surjective on $\BK$-points.
In fact, the image of the $\BK$-points is the set of squares and a similar statement holds for the induced map on class groups.
We refer to \cite[Sec.~7.1.1]{W18} for a thorough discussion of this, see also \cite[Sec.~4.2]{Linnikthmexpander}.
This lack of surjectivity is inconsequential in the subsequent arguments as the $2$-torsion of the class group is small.

\begin{proof}[Proof of Lemma~\ref{lemma:image of isogenies}]
It suffices to show that the restriction $\psi$ of $\Psi_{1,L}$ to $\BH_{a_1(L)}$ is a $\Q$-isomorphism. 
Now note that the kernel of $\psi$ is $\BH_{a_1(L)}\cap \BH_{L}^{\mathrm{pt}}$ and thus trivial by Lemma~\ref{lemma:pointwisestab}.
As an injective $\Q$-homomorphism between $\Q$-tori of rank $1$ is a $\Q$-isomorphism. (We remark that the inverse can be explicitly constructed in the case at hand.)
%
\end{proof}

In the case $i=5$ or $i=6$ we correspondingly see the set of fourth powers. This is due to the orthogonal complement construction -- see for instance \cite[Sec.~4.2]{Linnikthmexpander} and the way it is used in \cite{AES3D}.
The $4$-torsion of the Picard group is still of size $D^{o(1)}$ so that \cite{harcosmichelII} or \cite[Thm.~7.4]{W18} cover Proposition~\ref{prop:individualequidistr}.
%
%

Essential to the characterization of the joining $\mu$ is the fact that $\mu$ exhibits invariance under a higher rank diagonalizable action. This is the reason why the additional congruence conditions in Theorem \ref{thm:main} are needed (see also Lemma \ref{lemma:reason for congruencecond.}).

\begin{lemma}\label{lemma:inv.of the limit}
There exist planes $\Lambda_p\subset \Q_p^4$ and $\Lambda_q \subset \Q_q^4$
so that $\mu$ is invariant under the two commuting, diagonalizable subgroups
\begin{align*}
&T_p := \left\lbrace
\big(h,\Psi_{1,\Lambda_p}(h),\Psi_{2,\Lambda_p}(h),\Psi_{3,\Lambda_p}(h),\Psi_{4,\Lambda_p}(h)\big):
h \in \mathbb{H}_{\Lambda_p}(\Q_p)\right\rbrace \subset \G(\Q_p), \\
&T_q := \left\lbrace
\big(h,\Psi_{1,\Lambda_q}(h),\Psi_{2,\Lambda_q}(h),\Psi_{3,\Lambda_q}(h),\Psi_{4,\Lambda_q}(h)\big)):
h \in \mathbb{H}_{\Lambda_q}(\Q_q)\right\rbrace \subset \G(\Q_q).
\end{align*}
In other words, $\mu$ is invariant under $\mathbf{T} = T_p \times T_q$. Furthermore, $\mathbf{T}$ contains a subgroup $\mathbf{A} \cong \Z^4$ of class-$\mathcal{A}'$, which acts ergodically with respect to the Haar measure on each factor $X_{i,S}$ where $i= 1,\ldots,6$.
\end{lemma}

Here, the homomorphisms $\Psi_{1,\Lambda_p}, \Psi_{2,\Lambda_p}, \ldots$ can be defined as in Section \ref{section:acting group}.

For the general definition of the term \textit{class-$\mathcal{A}'$} we refer to \cite[Def.1.3]{EL-joining2}.
In our case it suffices to show that the group $\mathbb{H}_{\Lambda_p}(\Q_p)$ contains a subgroup $\mathbf{A}_p$ generated by some $h_1 \in \mathbb{H}_{a_1(\Lambda_p)}$ and $h_2 \in \mathbb{H}_{a_2(\Lambda_p)}$ each with eigenvalues $p^2,1,p^{-2}$ and that the same holds for $\mathbb{H}_{\Lambda_q}(\Q_q)$.
We remark that $\mathbf{A}_p$ (isomorphic to $\Z^2$) is mapped under each of the maps $\Psi_{1,\Lambda_p},\ldots,\Psi_{4,\Lambda_p}$ to a subgroup of rank one (and not two) by the discussion in Section \ref{section:acting group}.
The group $\mathbf{A}$ in the lemma is then simply the graph of $(\Psi_{1,\Lambda_p},\ldots,\Psi_{4,\Lambda_p})$ on the product $\mathbf{A}_p \times \mathbf{A}_q$.

\begin{proof}
By compactness we may assume that $A_{i,L_j} \to A_i \in \SL_4(\Zp)$ as $j \to \infty$ for all $i \in \set{1,2,3,4}$. 
Denote by $\Lambda_p$ (resp. $\tilde{\Lambda}_p$) the $\Q_p$-plane spanned by the first two rows of $A_1$ (resp. the last two rows of $A_2$). By continuity $a_1(\Lambda_p)\in \Z_p^3$ (with respect to the basis in $A_1$) is the limit of the sequence $a_1(L_j)$ and the same is true for $a_2(\Lambda_p)$.
By Proposition \ref{prop:splitting map} and by continuity of the Klein map we also have $\Phi(\tilde{\Lambda}_p) = [a_1(\Lambda_p),-a_2(\Lambda_p)]$ and hence $\tilde{\Lambda}_p = \Lambda_p^\perp$.

The admissability assumption on the planes $L_j$ yields that 
\begin{align*}
-Q(a_1(\Lambda_p)) = -Q(a_2(\Lambda_p))
\end{align*}
modulo $p$ is a non-zero square. In particular, the proof of Lemma \ref{lemma:reason for congruencecond.} shows that the stabilizer group $\mathbb{H}_{\Lambda_p}(\Q_p)<\SU_2^2(\Q_p) \cong \SL_2^2(\Q_p)$ is a maximal split torus (maximal as the rank is two).

Furthermore, as in Section \ref{section:def shapes} we obtain four binary forms (defined over~$\Z_p$) using $A_1,A_2,A_3,A_4$ each of which represents a restriction of $Q$ to a $\Z_p$-submodule of rank two (e.g.~the restriction of $Q$ to $\Lambda_p(\Z_p)$ uses the basis contained in $A_1$).
By the above these forms have discriminant in $-(\Z_p^\times)^2$ and are hence isotropic by Hensel's lemma.
This shows that $T_p$ is diagonalizable.

The same discussion applies to define $\Lambda_q$ (along a further subsequence) and to see that the obtained group $T_q$ is diagonalizable.
Since the measure $\mu_j$ is $\torus_{L_j}(\Q_p \times \Q_q)$-invariant, it follows directly that $\mu$ is $\mathbf{T}$-invariant.
The existence of the subgroup $\mathbf{A}_p$ (and hence $\mathbf{A}$ as in the lemma) follows from the fact that any maximal torus in $\SU_2^2(\Q_p)\cong \SL_2^2(\Q_p)$ is conjugate to the diagonal one where one can consider the subgroup generated by $(\diag(p,p^{-1}),\id)$ and $(\id,\diag(p,p^{-1}))$.
Indeed, this subgroup $\mathbf{A}_p$ acts ergodically on $X_{1,S}\times X_{2,S}$ by Mautner's phenomenon and by the fact that $\SU_2^2(\Q_p)$ acts ergodically by strong approximation. A similar argument for ergodicity applies in the other coordinates.
\end{proof}

\begin{proof}[Proof of Theorem \ref{thm:orbit version}]
As mentioned, we now want to apply the joinings classification in \cite[Thm.1.4]{EL-joining2}. 
To this end, recall that $\SL_2$ and $\SU_2$ are simply-connected (\cite[p.64]{platonov}) so that $\BG$ is also simply-connected. In particular, it follows that $X_S$ is saturated by unipotents in the sense of \cite{EL-joining2} i.e.~the subgroup generated by the unipotent elements acts ergodically. 
Let $\mathbf{A}$ be a subgroup as in Lemma \ref{lemma:inv.of the limit}. Then almost every ergodic component of $\mu$ is again a joining for the $\Z^4$-action on $X_S$ given by $\mathbf{A}$ (as the Haar measures are ergodic for the $\mathbf{A}$-action). 
Let $\nu$ be one such ergodic component. It is sufficient to show that $\nu =m_{X_S} $ to prove the theorem.

Moreover, by Corollary~1.5 in \cite{EL-joining2} we may as well show that the projection~$\nu_{k,\ell}$ of $\nu$ to any product of two factors $X_{k,S} \times X_{\ell,S}$ for $k<\ell$ is the trivial joining. Let $\mathbb{G}_k,\mathbb{G}_\ell \in \set{\SL_2,\SU_2}$ be the corresponding $\Q$-groups. 
By Theorem~1.4 in \cite{EL-joining2} there exists a linear algebraic group $\mathbb{M}<\G_k\times\G_\ell$ defined over $\Q$, a finite-index subgroup $M'<\mathbb{M}(\Q_S)$ and some $h \in (\G_k\times\G_{\ell})(\Q_S)$ so that $\nu_{k,\ell}$ is the Haar measure on $\Gamma M' h$ where $\Gamma = (\G_k\times\G_{\ell})(\Z^S)$.
The measure $\nu_{k,\ell}$ is exactly invariant under the subgroup $h^{-1}M' h$, which has finite index in $M = M_\infty \times M_p \times M_q := h^{-1}\BM(\Q_S)h$.
Since $h^{-1}M'h$ contains the projection of $\mathbf{A}$ to the $(k,\ell)$-th coordinate pair, the subgroup $M$ contains the projection of the Zariski-closure $\mathbf{T}$ of $\mathbf{A}$.

Assuming for a moment that $\BM = \G_k \times \G_{\ell}$ then for instance \cite[Cor.~6.7, p.~534]{boreltits1} proceeds to show that $\BM(\Q_S)$ does not have any proper, finite-index subgroup (as $\BM$ is simply-connected). In particular, $h^{-1}M'h = \G_k(\Q_S)\times\G_{\ell}(\Q_S)$ and $\nu_{k,\ell}$ is the trivial joining.

So suppose that $\BM \neq \G_k \times \G_{\ell}$.
As $\nu_{k,\ell}$ is a joining and as $\G_k$ or $\G_{\ell}$ are both simply-connected groups, the projections of $\mathbb{M}$ to $\G_k$ and $\G_{\ell}$ are isomorphisms.
In other words, $\mathbb{M}$ is the graph of some isomorphism between $\G_k$ and $\G_{\ell}$ defined over $\Q$ and in particular, $M_p$ is the graph of an isomorphism between $\G_k(\Q_p)$ and $\G_{\ell}(\Q_p)$.
To obtain a contradiction, we distinguish three cases.

\textsc{Case 1: $k,\ell\leq 2$.}
By assumption, we have $M_p \cong \SU_2(\Q_p)$ so that any maximal torus in $M_p$ has rank $1$.
On the other hand, by Lemma~\ref{lemma:inv.of the limit} the subgroup $M_p<\SU_2^2(\Qp)$ contains the torus $\mathbb{H}_{\Lambda_p}(\Q_p)$, which is of rank two (see Proposition~\ref{prop:splitting map} or Lemma~\ref{lemma:inv.of the limit}).

\textsc{Case 2: $k\leq 2, \ell \geq 3$.}
In this case there is no isomorphism between $\G_k = \SU_2$ and $\G_{\ell} = \SL_2$ as $\SU_2(\R)$ is compact and $\SL_2(\R)$ is not.

\textsc{Case 3: $k,\ell\geq 3$.}
We exhibit elements of $M_p<\SL_2(\Q_p)^2$ of the form $(\mathrm{Id},a)$ for some non-trivial $a \in \SL_2(\Q_p)$ contradicting the assumption.
\begin{itemize}
\item If $(k,\ell) = (3,4)$ we can consider $(\Psi_{1,\Lambda_p}(h),\Psi_{2,\Lambda_p}(h)) \in M_p$ for some non-trivial $h \in \mathbb{H}^{\mathrm{pt}}_{\Lambda_p}(\Qp)$.
This element is of the desired form $(\mathrm{Id},a)$ as $\Psi_{1,\Lambda_p}(h)= \text{Id}$ for any $h\in\mathbb{H}^{\mathrm{pt}}_{\Lambda_p}(\Qp)$ and as $\Psi_{2,\Lambda_p}(h)$ is non-trivial.

\item If $(k,\ell) = (3,5)$ we can consider $(\Psi_{1,\Lambda_p}(h),\Psi_{3,\Lambda_p}(h)) \in M_p$ for some non-trivial $h \in \mathbb{H}^{\mathrm{pt}}_{\Lambda_p}(\Qp)$ so that again $\Psi_{1,\Lambda_p}(h) = \id$.
Recall that by Lemma~\ref{lemma:pointwisestab} (the ``$45^\circ$-degree'' twist) the pointwise stabilizer of $\Lambda_p$ acts non-trivially on the orthogonal complement of $a_1(\Lambda_p)$ so that $a=\Psi_{3,\Lambda_p}(h)$ is in fact non-trivial.
The cases $(k,\ell) = (3,6),(4,5),(4,6)$ are analogous (in particular using Lemma~\ref{lemma:pointwisestab} again).
\item If $(k,\ell) = (5,6)$ we can consider $(\Psi_{3,\Lambda_p}(h),\Psi_{4,\Lambda_p}(h)) \in M_p$ for some non-trivial $h= (\id,h_2)$ where $h_2 \in \mathbb{H}_{a_2(\Lambda_p)}(\Qp)$.
\end{itemize}
This concludes the proof.
\end{proof}

\section{Equidistribution of oriented planes}

In this section, we discuss an extension of Theorem~\ref{thm:main} which takes into account the orientation of the subspaces.
In particular, we will associate to any oriented rational subspace four actual CM points avoiding the additional identification in the introduction.

\subsection{Homogenization of $\Gr$}\label{sec:wedgespaces}

Let $\purewedges$ be the homogeneous variety of pure $2$-wedges for affine four space; we see $\purewedges$ as a subvariety of the $6$-dimensional affine space via the choice of basis
\begin{align*}
e_1 \wedge e_2, e_1 \wedge e_3, e_1 \wedge e_4, e_2 \wedge e_3, e_2 \wedge e_4, e_3 \wedge e_4.
\end{align*}
One can view $\purewedges$ as the homogenization of the projective variety $\Gr$ and in particular we have a morphism of varieties
\begin{align*}
\homog: \purewedges \setminus \{0\} \to \Gr.
\end{align*}
As a map on $\BK$-points for a field $\BK$, the morphism $\phi$ sends a pure wedge $v_1 \wedge v_2$ for $v_1,v_2 \in \BK^4$ linearly independent to the subspace spanned by $v_1$ and $v_2$.

Furthermore, setting
\begin{align*}
\disc(v_1 \wedge v_2) = \det \begin{pmatrix}
(v_1,v_1) & (v_1,v_2) \\ (v_2,v_1) & (v_2,v_2)
\end{pmatrix}
\end{align*}
for any pure wedge $v_1 \wedge v_2$ we obtain a morphism of varieties defined over $\Q$ (also called the discriminant)
\begin{align*}
\disc: \purewedges \to \G_a.
\end{align*}
We denote by $\purend \subset \purewedges$ the Zariski open subset where $\disc \neq 0$ or equivalently $\purend = \homog^{-1}(\Grnd)$.

The level set $\disc = 1$ in $\purewedges$ is denoted by $\purewedges^1$.
The compact manifold $\purewedges^1(\R)$ is naturally the manifold of \emph{oriented} two-dimensional subspaces of $\R^4$.
Indeed, the fiber of any subspace $L \in \Gr(\R)$ under the surjective map $\homog: \purewedges^1(\R) \to \Gr(\R)$ consists of two points, one for each choice of orientation on $L$.
The action of $\SU_2^2(\R)$ on $\purewedges^1(\R)$ induced from the action on $\quat(\R) \simeq \R^4$ is transitive with connected stabilizers (as opposed to the action on the Grassmannian).
We equip $\purewedges^1(\R)$ with the unique invariant probability measure.

The integral structure on $\purewedges$ is inherited from the integral structure on the wedge: We let $\purewedges(\Z)$ be the set of pure wedges in $\bigwedge^2 \quat(\Z)$ and define $\purewedges(\Z_p)$ similarly for any prime $p$.
For any integer $D\geq 1$ we set
\begin{align}
\mathcal{R}'_D = \{w \in \purewedges(\Z): \disc(w) = D,\ w \text{ primitive}\}
\end{align}
and view it as a subset of $\purewedges(\R)$. As such, it may be projected onto $\purewedges^1(\R)$ by division with $\sqrt{D}$.

Write $\pureprim(\Z)$ for the set of primitive wedges in $\purewedges(\Z)$; any such wedge is of the form $w = v_1\wedge v_2$ where $v_1,v_2$ forms a basis of the two-dimensional subspace
\begin{align*}
L_w(\Q) = \Q v_1 + \Q v_2 = \{v \in \Q^4: v \wedge w = 0 \}.
\end{align*}
satisfying $L_w(\Z) = \Z v_1 + \Z v_2$.
By definition, $\disc(w) = \disc(L_w)$. The map 
\begin{align}\label{eq:orthcompl}
w \in \pureprim(\Z) \mapsto L_w \in \Gr(\Q)
\end{align}
is surjective and the fiber of each subspace consists of two points (each corresponding to a choice of orientation). Explicitly, if $L \in \Gr(\Q)$ and $v_1,v_2$ is a $\Z$-basis of $L(\Z)$, then $\pm v_1 \wedge v_2$ are the two points in the fiber of $L$.

\subsection{Theorem for oriented subspaces}

The orthogonal complement construction on the set of primitive wedges $\pureprim(\Z)\subset\purewedges(\Z)$ may be defined as follows, and descends to the usual orthogonal complement construction on $\Gr(\Q)$ under the map in \eqref{eq:orthcompl}.
For $w \in \pureprim(\Z)$ the vector $w_\perp \in \pureprim(\Z)$ is the unique vector corresponding to $L_w^\perp(\Z)$ with $w \wedge w_\perp$ being a positive multiple of $e_1 \wedge
e_2 \wedge e_3 \wedge e_4$.
In words, we choose $w_\perp$ so that the orientations of $w$ and $w_\perp$ are compatible.

The shape $[w]$ for $w = v_1 \wedge v_2 \in \pureprim(\Z)$ with $v_1,v_2 \in \Z^4$ is the \emph{proper} equivalence class of the quadratic form $(x,y) \mapsto Q(xv_1+ yv_2)$. In particular, we can view $[w]$ as a point in
\begin{align*}
Y_0(1) =  \lrquot{\SL_2(\Z)}{\SL_2(\R)}{\SO_2(\R)}.
\end{align*}
As before, the Klein map yields two points $a_1(w),a_2(w) \in \quat_0(\Z)$ (now with no ambuigity) for any $w \in \pureprim(\Z)$ -- we will briefly discuss this below.
The shape $[\Lambda_{a_1(w)}],[\Lambda_{a_2(w)}]$ of the lattices in the respective orthogonal complements can also be defined as points in $Y_0(1)$ (see also the definition of the shape in \cite{AES3D}).

\begin{theorem}\label{thm:oriented}
Let $p,q$ be two distinct odd primes.
The set of tuples
\begin{align*}
\big\{(\tfrac{1}{\sqrt{D}}w, [w], [w_\perp],[\Lambda_{a_1(w)}],[\Lambda_{a_2(w)}]): w \in \mathcal{R}_D'\big\}
\end{align*}
is equidistributed in $\purewedges^1(\R) \times Y_0(1)^4$ as $D \in \mathbb{D}$ goes to infinity while satisfying the additional condition that $-D$ is a non-zero square modulo $p$ and modulo $q$.
\end{theorem}

Except for the necessary discussion of the Klein map the theorem follows from Theorem~\ref{thm:orbit version} in a fashion similar to Theorem~\ref{thm:main}.

\subsection{The Klein map and orientation}\label{sec:kleinoriented}

We give the analogue of the Klein map (cf.~Proposition~\ref{prop:splitting map}) in this setup.
Roughly speaking, we need to explain in which sense the Klein map could capture the orientation of a subspace.

So suppose first that $L \in \Gr(\R)$ and that $v_1,v_2$ is an orthonormal basis of $L$.
Defining $a_1(L),a_2(L)$ as in \eqref{eq:associated intpts} we already obtain two points in $\mathbb{S}^2$. 
Note that reversing the order of $v_1,v_2$ changes the sign on $a_1(L)$ and $a_2(L)$.
We now explain how the choice of sign in $\pm (a_1(L),a_2(L))$ defines the orientation on $L$.
Proposition~\ref{prop:splitting map} easily implies that $L$ is invariant under left-multiplication by $a_1(L)$ (which is the same as right-multiplication by $a_2(L)$). As $a_1(L)^2 = -1$, this left-multiplication is a rotation by 90 degrees and the rotation by 90 degrees in the opposite direction is given by multiplication with $-a_1(L)$.
Therefore, fixing a non-zero vector $v\in L$ two bases with distinct orientation are given by $v,a_1(L)v$ resp.~$v,-a_1(L)v$.
In the following, we make this discussion more precise and extend it to any base field.

The action of $\SU_2^2$ on $\quat$ induces an action of $\SU_2^2$ on $\purewedges$. 
Under the morphism $\homog$, this action descends to the aforementioned action on $\Gr$.
We remark that for any $w \in \purend(\BK)$ the stabilizer $\BK$-subgroup $\BH_w< \SU_2^2$ is equal to $\BH_L$ where $L = \homog(w)$.

We define $\kleinwedge$ to be the subvariety of $\quat_0 \times \quat_0$ given by the equation $\Nr(x) = \Nr(y)$ and by $\kleinnondeg$ the Zariski open subset where $\Nr(x) \neq 0$.
We also denote by $a_1(w)$ resp.~$a_2(w)$ the points defined in \eqref{eq:associated intpts} for any $w = v_1 \wedge v_2 \in \purewedges(\BK)$.
These do not depend on any choices (as opposed to the definition for subspaces) and satisfy
\begin{align*}
\disc(w) = \Nr(a_1(w)) = \Nr(a_2(w))
\end{align*}
(see Lemma~\ref{lemma:formula lengths}).

\begin{corollary}[Affine version of the Klein map]
The map
\begin{align*}
\kleinmapwedge: \purend(\BK) \to \kleinnondeg(\BK),\ w \mapsto (a_1(w),a_2(w))
\end{align*}
is a well-defined bijection and equivariant for the $\SU_2^2(\BK)$-actions. Furthermore, it has the property that $\kleinmapwedge(\lambda w) = \lambda \kleinmapwedge(w)$ for any $w \in \purend(\BK)$ and $\lambda \in \BK^\times$.
\end{corollary}

\begin{proof}
These claims are straightforward to deduce from the calculations in the proof of Proposition~\ref{prop:splitting map} except for the bijectivity which we verify here by constructing an inverse.
Let $(a_1,a_2) \in \kleinnondeg(\BK)$, let $L = \{g \in \quat(\BK): a_1 g = g a_2\}$ be the plane $a_1,a_2$ define and let $g \in L(\BK)$ be invertible (as $L$ is non-degenerate, such a $g$ exists). In particular, $g$ satisfies $a_1 g = g a_2$ and $a_1g \in L(\BK)$. We define
\begin{align*}
\Psi(a_1,a_2) = \tfrac{1}{\Nr(g)}(a_1 g \wedge g)= \tfrac{1}{\Nr(g)}(g a_2 \wedge g).
\end{align*}
We leave it to the reader to verify that $\Psi(a_1,a_2)$ does not depend on the above choice of $g$, that $\kleinmapwedge \circ \Psi = \id_{\kleinnondeg(\BK)}$, and that $\Psi \circ \kleinmapwedge = \id_{\purend(\BK)}$.
%
\end{proof}

In particular, the Klein map induces an equivariant bijection
\begin{align*}
\kleinmapwedge: \purewedges(\R) \setminus \{0\} \to \kleinwedge(\R) \setminus \{(0,0)\}
\end{align*}
which induces an equivariant bijection $\purewedges^1(\R) \to \mathbb{S}^2 \times \mathbb{S}^2$.
In words, the manifold of oriented two-dimensional subspaces of $\R^4$ can be identified with $\mathbb{S}^2 \times \mathbb{S}^2$.
\section{Further comments and relations to class groups}\label{section:furthercomments}

In this section we formulate an arithmetic interpretation of Theorem~\ref{thm:main} which permits (possible) generalizations thereof. 
Let us first describe how the class group $\Cl(\mathcal{O}_D)$ acts on the projections of the collections appearing in Theorem \ref{thm:main} to each factor.
Here, $D\in \N$ is square-free and $\mathcal{O}_D$ is the maximal order in $\Q(\sqrt{-D})$.

\begin{itemize}
\item First, recall that for $D \in \BD$ square-free the class group $\Cl(\mathcal{O}_D)$ acts on the set of CM-points $\mathcal{CM}_D$ (as they are simply ideal classes in $\Cl(\mathcal{O}_D)$).
\item Similarly, the quotient
$\widetilde{\mathcal{R}_3}(D)$ of the set of integer points of norm $\sqrt{D}$ by $\SO_3(\Z)$ carries a transitive action\footnote{To be more precise, denote by $\mathcal{O}_\quat$ the (maximal) order of Hurwitz quaternions in $\quat(\Q)$. For any $x,y \in \quat_0$ with $\Nr(x) = \Nr(y) = D$ consider
\begin{align*}
\Lambda_{x \to y} 
= \setc{\lambda \in \mathcal{O}_\quat}{x\lambda = \lambda y}
= \mathcal{O}_\quat \cap \Phi^{-1}([x,y])
\end{align*}
which is a left-$\mathcal{O}_x$-ideal where $\mathcal{O}_x = \Q[x]\cap \mathcal{O}_\quat\ \widehat{=}\ \mathcal{O}_D$. 
The element of the class group mapping $x$ to $y$ is then given by the class of $\Lambda_{x \to y}$ (up to finite index issues) -- see \cite[Prop.~3.5]{Linnikthmexpander}.} 
of the class group (see \cite[Sec.~3]{Linnikthmexpander}).
We will identify the set $\widetilde{\mathcal{R}_3}(D)$ with the image in $\lquot{\SO_3(\Z)}{\mathbb{S}^2}$ after projection.
The Klein map then yields an action of $\Cl(\mathcal{O}_D)^2$ on the set $\mathcal{R}_D$ (or more precisely a finite-to-one quotient thereof).
\end{itemize}
All of the above actions of the class group $\Cl(\mathcal{O}_D)$ will be denoted by $[\mathfrak{a}].(\cdot)$ for any element $[\Fa] \in \Cl(\mathcal{O}_D)$.


\subsection{Monomials in ideals and simultaneous equidistribution}

Using these actions we may now give an analogous formulation of Theorem \ref{thm:main} which may be thought of as an arithmetic (rather than geometric) interpretation of it.

\begin{theorem}[Arithmetic version]\label{thm:arithmeticversion}
\ifsquarefree
Let $p$ and $q$ be two distinct odd primes. For any square-free $D\in \BD$ fix basepoints $v^{(D)},w^{(D)} \in \widetilde{\mathcal{R}_3}(D)$ as well as $\cm{z}_1^{(D)}$, $\cm{z}_2^{(D)}$, $\cm{z}_3^{(D)}$, $\cm{z}_4^{(D)}$ in  $\mathcal{CM}_D$. Then the subsets
\begin{align*}
\left\lbrace
\big([\Fa].v^{(D)},[\Fb].w^{(D)},[\Fa][\Fb].\cm{z}_1^{(D)},[\Fa][\Fb]^{-1}.\cm{z}_2^{(D)},[\Fa]^2.\cm{z}_3^{(D)},[\Fb]^2.\cm{z}_4^{(D)}\big): 
[\Fa],[\Fb]\in \Cl(\mathcal{O}_D)
\right\rbrace
\end{align*}
of $(\lquot{\SO_3(\Z)}{\mathbb{S}^2})^2 \times \mathcal{X}_2^4$ equidistribute as $D \to \infty$ while $D$ satisfies the assumptions of Theorem \ref{thm:main}.
\else
Let $p$ and $q$ be two distinct odd primes. For any square-free $D\in \N$ which is not of the form\footnote{This guarantees that $\widetilde{\mathcal{R}_3}(D)$ is non-empty.} $4^a(8b+7)$ for integers $a,b$ fix basepoints $v^{(D)},w^{(D)} \in \widetilde{\mathcal{R}_3}(D)$ as well as $\cm{z}_1^{(D)}$, $\cm{z}_2^{(D)}$, $\cm{z}_3^{(D)}$, $\cm{z}_4^{(D)}$ in  $\mathcal{CM}_D$. Then the subsets
\begin{align*}
\left\lbrace
\big([\Fa].v^{(D)},[\Fb].w^{(D)},[\Fa][\Fb].\cm{z}_1^{(D)},[\Fa][\Fb]^{-1}.\cm{z}_2^{(D)},[\Fa]^2.\cm{z}_3^{(D)},[\Fb]^2.\cm{z}_4^{(D)}\big): 
[\Fa],[\Fb]\in \Cl(\mathcal{O}_D)
\right\rbrace
\end{align*}
of $(\lquot{\SO_3(\Z)}{\mathbb{S}^2})^2 \times Y_0(1)^4$ equidistribute when $D$ goes to infinity while $D$ is square-free and satisfies $\legendre{-D}{p} = \legendre{-D}{q} = 1$.
\fi
\end{theorem}

Note the symmetry in the actions here.
Clearly, the acting element in the first (resp.~the second) coordinate in $Y_0(1)$ is the quotient (resp.~the product) of the two acting elements in the coordinates in $\widetilde{\mathcal{R}_3}(D)$.
Also, the acting element $[\Fa]^2$ in the third coordinate in $Y_0(1)$ is given by the product of the acting elements $[\Fa][\Fb]^{-1}$ and $[\Fa][\Fb]$ in the first resp.~second coordinate in $Y_0(1)$. Similarly, the acting element $[\Fb]^2$ in the fourth coordinate in $Y_0(1)$ is the quotient of $[\Fa][\Fb]$ and $[\Fa][\Fb]^{-1}$.

The authors find it to be a pleasing coincidence that the objects in Theorem~\ref{thm:main} obtained by geometric constructions (the Klein map and the orthogonal complement) admit a description as in Theorem \ref{thm:arithmeticversion}.
Observe that the relation between the basepoints is determined by these geometric constructions.
We will hint at this relation in Appendix~\ref{section:appendixA}.

\subsection{Extensions}\label{section:extensions}

The above arithmetic game may be extended. For instance, one can change the $45^\circ$-degree picture (see Lemma \ref{lemma:pointwisestab}) and show that the triples
\begin{align*}
\big([\Fa].\cm{z}_1^{(D)},[\Fb].\cm{z}_2^{(D)},[\Fa][\Fb]^3.\cm{z}_3^{(D)})
\end{align*}
for $[\Fa],[\Fb] \in \Cl(\mathcal{O}_D)$ equidistribute where $\cm{z}_1^{(D)},\cm{z}_2^{(D)},\cm{z}_3^{(D)}$ denotes any choice of basepoints.
Another interesting case (as first studied by M.~Bhargava) concerns the triples
\begin{align*}
\big([\Fa].\cm{z}_1^{(D)},[\Fb].\cm{z}_2^{(D)},([\Fa][\Fb])^{-1}.\cm{z}_3^{(D)}).
\end{align*}
Here one can prove equidistribution under weakened congruence conditions. This will appear in an upcoming preprint \cite{ELM} of the second-named author with E.~Lindenstrauss and Ph.~Michel. 

Given $k \in \N$ one could also ask about the distribution of the set of tuples
\begin{align}\label{eq:equidistribution of powers}
\big([\Fa].\cm{z}_1^{(D)},[\Fa]^2.\cm{z}_2^{(D)},\ldots,[\Fa]^k.\cm{z}_k^{(D)}\big)\in Y_0(1)^k
\end{align}
for $[\Fa] \in \Cl(\mathcal{O}_D)$ and any fixed choice of basepoints $\cm{z}_1^{(D)},\ldots,\cm{z}_k^{(D)} \in \mathcal{CM}_D$.
The difficulty here lies in the individual equidistribution as there is no sufficient quantitative control on the $k$-torsion of the class group.
Assuming individual equidistribution, the joinings classification in \cite{EL-joining2} can be used to show equidistribution of these tuples under sufficient congruence conditions.
Such a theorem will appear in \cite{ELM} in the case of $k=2$ (as there is control on the $2$-torsion).

One may combine this theorem with the problems from the beginning of this subsection (and the like) to obtain an equidistribution statement of the following kind.
If $n = (n_1,n_2) \in \Z^2$ is primitive with $n_1 \neq 0 \neq n_2$ then the tuples
\begin{align*}
\big(	[\Fa].\cm{z}_1^{(D)},[\Fb].\cm{z}_2^{(D)},
[\Fa]^{n_1}[\Fb]^{n_2}.\cm{z}_3^{(D)},
[\Fa]^{2n_1}[\Fb]^{2n_2}.\cm{z}_4^{(D)}	 \big)
\in Y_0(1)^4
\end{align*}
for $[\Fa],[\Fb] \in \Cl(\mathcal{O}_D)$ and any fixed choice of basepoints $\cm{z}_1^{(D)},\cm{z}_2^{(D)},\cm{z}_3^{(D)},\cm{z}_4^{(D)} \in \mathcal{CM}_D$ equidistribute as $D \to \infty$ under sufficient congruence conditions.
Note that one can apply \cite[Cor.~1.4]{EL-joining2} in order to generalize this to any finite number of primitive vectors (weights) $n \in \Z^2$ as long as no two vectors are equal or opposite to each other\footnote{An elementary argument shows that the sets $\left\lbrace\big(
[\Fa].\cm{z}^{(D)},[\Fa]^{-1}.\cm{z}^{(D)}\big): [\Fa] \in \Cl(\mathcal{O}_D)\right\rbrace$ for a basepoint $\cm{z}^{(D)}\in \mathcal{CM}_D$
are not equidistributed in $Y_0(1)^2$ as $D \to \infty$.}.

\subsection{Subcollections}

Given any subset $\mathcal{S}_D \subset \Cl(\mathcal{O}_D)^2$ for every $D$ one could enquire about equidistribution of the tuples considered in Theorem \ref{thm:main} or Section~\ref{section:extensions} when restricted to $\mathcal{S}_D$.
Let us discuss this question only in the context of Theorem~\ref{thm:arithmeticversion} and in a concrete example here.

Motivated by the mixing conjecture of Michel and Venkatesh fix an ideal class $[\mathfrak{b}] \in \Cl(\mathcal{O}_D)$ and consider the subset
\begin{align*}
\mathcal{S}_D = \{ ([\Fa],[\Fa][\Fb]): \Fa \in \Cl(\mathcal{O}_D) \} \subset \Cl(\mathcal{O}_D)^2.
\end{align*}
Given the collections from Theorem \ref{thm:arithmeticversion} for equal basepoints $v^{(D)} = w^{(D)}$ and $\cm{z}_1^{(D)} = \ldots = \cm{z}_4^{(D)} = \cm{z}^{(D)}$ one obtains along the subset $\mathcal{S}_D$ the finite set of tuples of the form
\begin{align*}
\big([\Fa].v^{(D)},[\Fa][\Fb].v^{(D)},
[\Fa]^2[\Fb].\cm{z}^{(D)},[\Fb]^{-1}.\cm{z}^{(D)},
[\Fa]^2.\cm{z}^{(D)},[\Fa]^2[\Fb]^2.\cm{z}^{(D)}\big)
\end{align*}
for $[\Fa]\in \Cl(\mathcal{O}_D)$.
Clearly, equidistribution in the fourth component is impossible as only one point in $Y_0(1)$ is considered.
Recent work of Khayutin \cite{ilyaJointCM} yields equidistribution under sufficient assumptions on $[\Fb]$ and on the quadratic field $\Q(\sqrt{-D})$.

\section{Glue groups}\label{section:glue groups}

In this section we formulate a further strengthening of Theorem \ref{thm:main} in terms of glue groups, whose definition we now recall.

The \emph{dual lattice} of a lattice $\Lambda \subset \Q^n$ is defined as
\begin{align*}
\dual{\Lambda} = \setc{v \in \Lambda \otimes \Q}{\scalara{v}{w} \in \Z \text{ for all } w \in \Lambda } \cong \Hom(\Lambda,\Z)
\end{align*}
where $\scalara{\cdot}{\cdot}$ is standard Euclidean inner product.
When $\Lambda$ is integral (i.e.~$\scalara{\cdot}{\cdot}$ restricted to $\Lambda$ takes values in $\Z$), then $\dual{\Lambda}$ contains $\Lambda$ and the \emph{glue group} (or \emph{discriminant group}) of $\Lambda$
\begin{align*}
\glue(\Lambda) = \rquot{\dual{\Lambda}}{\Lambda}
\end{align*}
is a finite abelian group of order $\disc(\Lambda)$. 
A glue group comes with a naturally attached binary form
\begin{align*}
\fform{\cdot}{\cdot}: \glue(\Lambda)  \times \glue(\Lambda)  &\to \rquot{\Q}{\Z},\\
\fform{v + \Lambda}{w + \Lambda} &= \scalara{v}{w} + \Z
\end{align*}
called the \emph{fractional form} of the glue group.

We remark that the name ``glue'' originates from the question whether or not two lattices  $\Lambda_1,\Lambda_2$ can be glued together in the sense that there is a unimodular lattice $\Lambda$ and an embedding of the orthogonal sum $\Lambda_1\oplus \Lambda_2 \hookrightarrow \Lambda$ such that $\Lambda \cap (\Lambda_1\otimes\R) = \Lambda_1$ and $\Lambda \cap (\Lambda_2\otimes\R) = \Lambda_2$.
In fact, $\Lambda_1,\Lambda_2$ can be glued in this way if and only if there is an isomorphism $\phi:\glue(\Lambda_1) \to \glue(\Lambda_2)$ that maps the fractional form on $\glue(\Lambda_1)$ to the negative of the fractional form on $\glue(\Lambda_2)$.
For this and for further background on glue groups we refer to \cite{conwaysloane} and \cite{McMullenglue}.

Following a question raised by C. McMullen we consider the distribution of rational planes $L$ in $\Gr(\Q)$ whose glue group $\glue(L(\Z))$ is of a given isomorphism type.
Here, two glue groups  $(\glue(\Lambda_1),\fform{\cdot}{\cdot}_1)$ and $(\glue(\Lambda_2),\fform{\cdot}{\cdot}_2)$ are called isomorphic if there is an isomorphism $\varphi:\glue(\Lambda_1)\to\glue(\Lambda_2)$ between the abstract groups with $\fform{\varphi(x)}{\varphi(y)}_2=\fform{x}{y}_1$ for all $x,y\in \glue(\Lambda_1)$. 
For simplicity, we just write this as $\glue(\Lambda_1)\cong \glue(\Lambda_2)$.

\begin{theorem}[Equidistribution along prescribed isomorphism types]\label{thm:equi along isotypes}
Let $p,q$ be any two distinct odd primes. 
For any $D \in \BD$ fix a plane $L_D \in \mathcal{R}_D$ and set
\begin{align*}
\mathcal{R}_D(L_D) = \setc{L \in \mathcal{R}_D}{\glue(L(\Z))\cong \glue(L_D(\Z))}.
\end{align*}
Let $\mathcal{J}_D(L_D) \subset \mathcal{J}_D$ be the subset of points, whose underlying planes are in $\mathcal{R}_D(L_D)$. 
Then the normalized counting measure on the finite sets $\mathcal{J}_{D}(L_D)$
equidistributes to the uniform probability measure on $\Gr(\R) \times \mathcal{X}_2^4$ when $D$ goes to infinity along any sequence for which the planes $L_D$ are admissible with respect to $p,q$.
\end{theorem}

Here, admissible sequences of planes were defined in Section \ref{section:formulation dyn.result}.
Given the local interpretation of glue groups from the next subsection and Theorem \ref{thm:orbit version} this result is quite directly deduced.

\subsection{Local glue groups}

The glue group of a lattice $\Lambda \subset \Q^n$ can be computed from local quantities.
For any prime $p$ denote by $\Lambda_p = \Lambda \otimes \Z_p$ the completion at $p$ and define the \emph{$p$-glue group} of $\Lambda$ as the abelian $p$-group
\begin{align*}
\glue(\Lambda)_p = \rquot{\dual{\Lambda_p}}{\Lambda_p}.
\end{align*}
Notice that $\glue(\Lambda)_p$ also comes equipped with a fractional form
\begin{align*}
\fform{\cdot}{\cdot}_p:\glue(\Lambda)_p \times \glue(\Lambda)_p \to \rquot{\Q_p}{\Z_p}.
\end{align*}
Then the glue group of $\Lambda$ can be computed as
\begin{align*}
\glue(\Lambda) \cong \prod_p \glue(\Lambda)_p
\end{align*}
and for any $v,w\in\dual{\Lambda}$ the value $\fform{v+\Lambda}{w+\Lambda}$ is uniquely determined by the values $\fform{v+\Lambda_p}{w+\Lambda_p}_p$ for all primes $p$.
We remark that the $p$-glue group is clearly also defined for general $\Z_p$-lattices in $\Q_p^n$.

\begin{proof}[Proof of Theorem \ref{thm:equi along isotypes}]
By Theorem \ref{thm:orbit version} it suffices to show that for any $L \in \mathcal{R}_D$ the planes underlying the points generated from the adelic orbit $\G(\Q)\torus_L(\adele)$ all have the same isomorphism type as $\glue(L(\Z))$.
If $L' \in \mathcal{R}_D$ is any such plane and $p$ is a prime, then the construction shows that there is $g \in \SU_2^2(\Z_p)$ such that $g.L(\Z_p) = L'(\Z_p)$.
In particular, $\glue(L(\Z))_p \cong \glue(L'(\Z))_p$ for all primes $p$ and thus $\glue(L(\Z)) \cong \glue(L'(\Z))$.
\end{proof}

\subsection{Primitivity and glue}\label{section:primitivity vs glue}

In this section we compute the glue group for lattices of the form $L(\Z)$ where $L \in \mathcal{R}_D$.
Notice that this is not needed for the proof of Theorem \ref{thm:equi along isotypes}, but relates our discussion of glue groups to Section~\ref{section:associated intpts} and Section~\ref{section:local for forms}.
We refrain from discussing the fractional forms here and focus on the abstract groups.

\begin{proposition}[Primitivity at odd primes]\label{prop:isotype at odd primes}
For an odd prime $p$ let $L\in \mathcal{R}_{D,p}$ and define the numbers $k = \max\set{\ord_p(a_1(L)),\ord_p(a_2(L))}$ and $n = \ord_p(D)$.
Then the $p$-glue group of $L(\Z_p)$ satisfies
\begin{align*}
\glue(L(\Z)))_p \cong \rquot{\Z}{p^k\Z}\times \rquot{\Z}{p^{n-k}\Z}.
\end{align*}
\end{proposition}

We remark that the subspaces $L$ and $L^\perp$ can be shown to have isomorphic (local or non-local) glue groups by elementary means \cite[p.~100]{conwaysloane} which agrees with the above proposition where $k$ is invariant under taking orthogonal complements.
In fact, one can prove using $\mathcal{G}(L(\Z)) \simeq \mathcal{G}(L^\perp(\Z))$ that $\ord_p(q_L) = \ord_p(q_{L^\perp})$ which can be seen to suffice for the proof of Theorem~\ref{thm:main} (replacing Proposition~\ref{prop:geometric-nonsqfree}).

\begin{proof}
Notice that the action of $\SU_2^2(\Z_p)$ on the set $\mathcal{R}_{D,p}$ preserves the isomorphism class of $p$-glue groups (that is, including the fractional form).
Recall that by Lemma~\ref{lemma:loc.conjugacy of planes} and under the identification $\quat(\Q_p) \cong \Mat_2(\Q_p)$ there exists some $g \in \SU_2^2(\Z_p)$ and non-zero $\alpha_1,\alpha_2 \in \Z_p$ such that
\begin{align*}
g.L(\Z_p) 
= \Z_p \begin{pmatrix}
\alpha_1 & 0 \\ 
0 & \alpha_2
\end{pmatrix} \oplus 
\Z_p \begin{pmatrix}
0 & 1 \\ 
-\frac{D}{\alpha_1\alpha_2} & 0
\end{pmatrix}
\end{align*}
with $\ord_p(\alpha_1) = \ord_p(a_1(L))$ and $\ord_p(\alpha_2) = \ord_p(a_2(L))$.
By pair-primitivity we have $k = \ord_p(\alpha_1) + \ord_p(\alpha_2) = \ord_p(\alpha_1\alpha_2)$.

Since this direct sum is orthogonal, the $p$-glue group of $g.L(\Z_p)$ is simply the product of the $p$-glue groups of the summands.
Now for any primitive vector $v \in \Z_p^3$ the $p$-glue group of $\Z_p v$ is $\rquot{\Z_p}{Q(v)\Z_p} \cong \rquot{\Z}{p^{\ord_p(Q(v))}}\Z$. 
For $v = \Big(\SmallMatrix{\alpha_1 & 0 \\ 
0 & \alpha_2}\Big)$ this gives
\begin{align*}
\mathcal{G}\left(\Zp \Big(\SmallMatrix{\alpha_1 & 0 \\ 
0 & \alpha_2}\Big)\right)
\cong
\rquot{\Zp}{\alpha_1\alpha_2\Zp} \cong \rquot{\Z}{p^{\ord_p(\alpha_1\alpha_2)}\Z}
\cong \rquot{\Z}{p^k\Z}.
\end{align*}
Similarly, $\mathcal{G}\left(\Zp \Big(\SmallMatrix{0 & 1 \\ 
-\frac{D}{\alpha_1\alpha_2} & 0}\Big)\right) \cong \rquot{\Z}{p^{n-k}\Z}$ and the proposition follows.
\end{proof}

Contrary to the behaviour at odd primes, the local glue groups at $2$ are determined by the discriminant only.

\begin{proposition}[Primitivity at $2$]\label{prop:isotype at 2}
Let $L \in \mathcal{R}_{D,2}$ for $D \in \BD$ and assume that $n = \ord_2(D)$ is positive.
Then
\begin{align*}
\glue(L(\Z))_2 \cong \rquot{\Z}{2\Z} \times \rquot{\Z}{2^{n-1}\Z}.
\end{align*}
\end{proposition}

\begin{proof}
By Corollary \ref{cor: congruence condition} (describing the assumption $D \in\BD$) we only need to handle the cases $n=1,2,3$.
Note that for $n=1$ the statement is clear, as $\glue(L(\Z))_2$ is a finite group of order $2$.

In the cases $n=2$ or $n=3$ we will heavily use the fact that $\quat(\Q_2)$ is a division algebra (compare the discussion to follow with \cite[Prop.~3.7]{Linnikthmexpander}).
As $D$ is divisible by $4$, the vectors $w_1 = \frac{1}{2}a_1(L)$ and $w_2 = \frac{1}{2}a_2(L)$ are integral. 
Choose a non-zero element $g \in L(\Z_2)$ of maximal norm.
Looking at the inverse Klein map (cf.~Proposition~\ref{prop:splitting map}), this element satisfies $w_1 g = g w_2$. 
Moreover,
\begin{align*}
L(\Q_2) 
&= \setc{x\in \quat(\Q_2)}{w_1x = x w_2}
= \setc{x\in \quat(\Q_2)}{w_1x = x g^{-1}w_1g}\\
&= \setc{y \in \quat(\Q_2)}{w_1y = y w_1}g = \Q_2 g \oplus \Q_2 w_1 g,
\end{align*}
which is in fact an orthogonal sum as $\Tr(w_1g\overline{g}) = \Tr(w_1) g \overline{g} = 0$.

If $x \in L(\Z_2)$ is any non-zero vector and $x_1,x_2 \in \Q_2$ are such that $x = x_1 g + x_2 w_1g$ then $\Nr(x) = \Nr(g) (x_1^2 + \Nr(w_1)x_2^2)$.
We introduce the short-hand $d = \frac{D}{4} = \Nr(w_1)$ and note that by our choice of $g$ we have $x_1^2 + d x_2^2 \in \Z_2$.
Since $d \equiv 1,2\mod 4$ (as $D \in \BD$), the congruence equation $z_1^2+dz_2^2 \equiv 0 \text{ mod }4$ has no non-trivial solutions.
This implies in particular for any $z_1,z_2 \in \Q_2$ with $z_1^2+dz_2^2\in \Z_2$ that $z_1,z_2 \in \Z_2$.
Hence, we have $x_1,x_2 \in \Z_2$ and
\begin{align*}
 L(\Z_2) = \Z_2 g \oplus \Z_2 w_1 g.
\end{align*}
In particular, $\disc(L(\Z_2)) = D = 4 \Nr(w_1)$ equals $\Nr(g)^2\Nr(w_1)$ up to a square in $\Z_2^\times$ which implies that $\ord_2(\Nr(g)) = 1$.
As $\ord_2(\Nr(w_1)) = \ord(d) = n-2$, the statement for the glue group can be obtained as in the conclusion of the proof of Proposition \ref{prop:isotype at odd primes}.
\end{proof}

As mentioned after Proposition \ref{prop:splitting map}, one can refine the statement at the prime~$2$ therein.
For instance, one can show for $L \in \mathcal{R}_D$ that $\ord_{2}(Q|_{L(\Z)}) = 0$ whenever $D\equiv 1,2\mod 4$ and that $\ord_{2}(Q|_{L(\Z)}) = 1$ whenever $D \equiv 0 \mod 4$.
For this, one can apply the technique in the proof of Proposition \ref{prop:isotype at 2} above.
We omit this here.

\begin{appendix}
\section{Proof of Corollary \ref{cor:averaged}}\label{section:proof of averaged version}

In this section, we will prove the averaged version (Corollary ~\ref{cor:averaged}) of our main theorem (Theorem \ref{thm:main}) using the homogeneous counting results of Duke-Rudnick-Sarnak \cite{dukerudnicksarnak} and Eskin-McMullen \cite{eskinmcmullen}.

Consider first the following tentative argument.
Let $\mathbb{P}_m$ be the set of the first $m$ odd primes.
The 'probability' that $-D \in \Z$ is zero or a non-square modulo $ p$ is bounded from above by $\frac{2}{3}$.
Thus, the probability that there are no two distinct primes $p,q \in \mathbb{P}_m$ with $-D \mod p \in (\Fp^\times)^2$ and with $-D \mod q \in (\BF_q^\times)^2$ is at most $(m+1)(\frac{2}{3})^m$ (essentially by the Chinese remainder theorem).
Somewhat similarly, we would now like to know the proportion of the set of planes $L$ in $\mathcal{R}_{< D}= \bigcup_{d < D}\mathcal{R}_d$ for which the discriminant $\disc(L)$ satisfies $-\disc(L) \mod p \in (\Fp^\times)^2$ for at most one $p \in \mathbb{P}_m$. 
In fact, we claim that this proportion is also $\ll (m+1)(\frac{2}{3})^m$.
To prove this, we will use the counting results mentioned above to estimate the number of points in $\mathcal{R}_{< D}$ that satisfy certain congruence conditions. 
Notice that from the claim Corollary \ref{cor:averaged} follows quite immediately from Theorem~\ref{thm:main} (cf.~the proof below).

\subsection{Definition of the homogeneous space}

As in Section~\ref{sec:wedgespaces} we denote by $\purewedges$ the variety of pure $2$-wedges in affine four space. As this section uses mostly the real points and subsets thereof, we write $\varwedge = \purewedges(\R)$ for simplicity.
Recall that $\rquot{\varprim(\Z)}{\set{\pm 1}}$ can be identified with the set of rational planes.

We view
\begin{align*}
\varwedge \subset \nicewedge{2} \R^4 \simeq \R^6
\end{align*}
by choosing the standard basis $e_m \wedge e_n$ with $m <n$ for the identification $\bigwedge^2 \R^4 \simeq \R^6$.
Let us denote by $x_{mn}$ the coordinates in this basis.
Notice that $\varwedge$ is the zero locus of the quadratic form $w \in \bigwedge^2 \R^4 \mapsto w \wedge w$ which is represented in the standard basis by the form $x_{12}x_{34}-x_{13}x_{24}+x_{14}x_{23}$.

Let $\norm{\cdot}$ be the Euclidean norm on $\R^6$. We may retrieve the discriminant of a rational plane $L\in \Gr(\Q)$ with integral basis $v_1,v_2\in \Z^4$ by the formula
\begin{align*}
\norm{v_1\wedge v_2}^2 = \disc(L).
\end{align*}
In particular,
\begin{align*}
2\big|\mathcal{R}_{< D}\big| = \big|\varprim(\Z)\cap B_{\sqrt{D}}(0)\big|.
\end{align*}

\subsubsection{$\varwedge$ as a homogeneous variety}
Note that $\SL_4(\R)$ acts transitively on $\varwedge$ via $g. (v_1\wedge v_2) = (g.v_1) \wedge (g.v_2)$ (this is simply the natural action on planes).
The induced action of $\SL_4(\Z)$ on $\varprim(\Z)$ is transitive.
Furthermore, the stabilizer of the wedge $e_1 \wedge e_2$ under the action of $\SL_4(\R)$ is the group
\begin{align*}
H = \left\lbrace \begin{pmatrix}
A & \ast \\
0 & B
\end{pmatrix}: A,B \in \SL_2(\R)\right\rbrace.
\end{align*}
We may thus identify $\varwedge$ with the quotient $\rquot{\SL_4(\R)}{H}$. 
More generally, we denote by $H_w$ the stabilizer of $w \in \varwedge$.

\subsubsection{Reducing points on the variety}
Let $N \in \N$ be odd and square-free.
We consider the similarly defined finite set 
\begin{align*}
\purewedges\big(\rquot{\Z}{N\Z}\big) 
\subset \nicewedge{2} \big(\rquot{\Z}{N\Z}\big)^4 \simeq \big(\rquot{\Z}{N\Z}\big)^6
\end{align*}
and let $\varprim(\rquot{\Z}{N\Z})$ to be the set of primitive vectors $\mathbf{a} \in \purewedges\big(\rquot{\Z}{N\Z}\big)$.
Denote by $\purewedges_{\mathbf{a}}(\Z) \subset \varprim(\Z)$ the set of wedges~$w$ with $w \equiv \mathbf{a} \mod N$. 
By the Chinese remainder theorem we have
\begin{align}\label{eq:chinese remainder}
\varprim\big(\rquot{\Z}{N\Z}\big) 
\simeq \prod_{p \mid N} \varprim\big(\Fp\big)
\end{align}
where taking the discriminant on the left-hand side corresponds to taking the discriminant componentwise on the right-hand side.
Clearly, $\varprim(\Fp) = \purewedges(\Fp)\setminus~\{0\}$.

\subsection{Proof of Corollary \ref{cor:averaged}}

The proof of Corollary \ref{cor:averaged} uses (apart from Theorem \ref{thm:main}) the following two ingredients.

\begin{proposition}[Counting under congruence conditions]\label{prop:counting in fibers of reduction}
Let $N\in \N$ be odd and square-free and let $\mathbf{a} \in \varprim(\rquot{\Z}{N\Z})$ so that $\purewedges_{\mathbf{a}}(\Z) \neq \emptyset$. Then we have
\begin{align*}
|\purewedges_{\mathbf{a}}(\Z) \cap B_{\sqrt{D}}(0)| 
\asymp \tfrac{2}{|\varprim(\rquot{\Z}{N\Z})|} |\mathcal{R}_{<D}|.
\end{align*}
\end{proposition}

\begin{proposition}\label{prop:elm counting variety mod N}
Let $N \in \N$ be odd and square-free.
Then the relative number of vectors $\mathbf{a} \in \varprim(\rquot{\Z}{N\Z})$ with the property that $-\disc(\mathbf{a})$ is a non-zero square mod~$p$ for at most one prime $p|N$ is $\ll m (\frac{2}{3})^m$ where $m$ is the number of prime divisors of~$N$.
\end{proposition}

We note that the homogeneous counting results \cite{dukerudnicksarnak} and \cite{eskinmcmullen} are used to prove Proposition~\ref{prop:counting in fibers of reduction}.
Also, we remark that explicit asymptotics for $|\mathcal{R}_{<D}|$ as $D \to \infty$ are known (see Schmidt \cite{Schmidt-count}, \cite{Schmidt-shapes}), but will not be needed here.
We will prove Propositions \ref{prop:counting in fibers of reduction} and \ref{prop:elm counting variety mod N} below, but let us first explain how they can be combined to obtain the corollary of Theorem \ref{thm:main}.

\begin{proof}[Proof of Corollary \ref{cor:averaged}]
Let $m \geq 1$ and let $N$ be the product of the first $m$ odd primes.
Furthermore, we set $\mathcal{B}_m$ to be the subset of points $\mathbf{a}$ in $\varprim(\rquot{\Z}{N\Z})$ for which $-\disc(\mathbf{a})$ is a non-zero square modulo at most one prime $p\mid N$.
Denote by $\mathcal{B}_m(D)$ the subset of points in the set $\mathcal{J}_{<D} =\bigcup_{d < D} \mathcal{J}_d$ whose underlying planes $L\in \varprim(\Z)$ satisfy $L \mod N \in \mathcal{B}_m$.
Then by Proposition \ref{prop:counting in fibers of reduction} and Proposition~\ref{prop:elm counting variety mod N}
\begin{align*}
|\mathcal{B}_m(D)| 
\ll \sum_{\mathbf{a} \in \mathcal{B}_m} \tfrac{1}{|\varprim(\rquot{\Z}{N\Z}|)} |\mathcal{R}_{<D}| 
\ll m (\tfrac{2}{3})^m |\mathcal{R}_{<D}|
\end{align*}
Therefore, the average of a continuous function $f:\Gr(\R)\times \mathcal{X}_2^4 \to \C$ over $\mathcal{J}_{<D}$ differs from the average over $\mathcal{J}_{<D} \setminus \mathcal{B}_m(D)$ by $\ll m(\frac{2}{3})^m \norm{f}_\infty$.
Notice that each discriminant appearing in $\mathcal{J}_{<D} \setminus \mathcal{B}_m(D)$ satisfies the splitting conditions of Theorem~\ref{thm:main} at least at two prime divisiors of $N$.
Thus, Theorem \ref{thm:main} implies equidistribution of these finite subsets and so
\begin{align*}
\limsup_{D \to \infty} \Big| \tfrac{1}{|\mathcal{J}_{<D}|} \sum_{x \in \mathcal{J}_{<D}}f(x) - \int f \Big| \ll m (\tfrac{2}{3})^m \norm{f}.
\end{align*}
Since $m$ was arbitrary, Corollary \ref{cor:averaged} follows.
\end{proof}

\subsection{Counting under congruence conditions}

Let $\mathbf{a} \in \varprim(\rquot{\Z}{N\Z})$ be fixed.
Furthermore, let $\Gamma_{\mathbf{a}}$ be the subgroup of $\SL_4(\Z)$ consisting of the elements which preserve the subset $\purewedges_{\mathbf{a}}(\Z) \subset \varprim(\Z)$.

\begin{lemma}[Index of $\Gamma_{\mathbf{a}}$]
Whenever $\purewedges_{\mathbf{a}}(\Z)$ is non-empty, $\Gamma_{\mathbf{a}}$ acts transitively on $\purewedges_{\mathbf{a}}(\Z)$ and the index of $\Gamma_{\mathbf{a}}$ in $\SL_4(\Z)$ is equal to $|\varprim\big(\rquot{\Z}{N\Z}\big)|$.
Furthermore, $\Gamma_{\mathbf{a}}$ is a congruence subgroup.
\end{lemma}

\begin{proof}
Denote by $w \in \Z^6 \mapsto \overline{w}$ the reduction mod $N$ and let $w \in \varprim(\Z)$ with $\overline{w}=\mathbf{a}$.
Then for any $\gamma\in \SL_4(\Z)$ we have $\overline{\gamma.w} = \overline{\gamma}.\mathbf{a}$ where we used the analogous notation for the (surjective) reduction map
$\SL_4(\Z) \to \SL_4(\rquot{\Z}{N\Z})$.
In particular, if $H_{\mathbf{a}}<\SL_4(\rquot{\Z}{N\Z})$ denotes the stabilizer of $\mathbf{a}$ then $\Gamma_{\mathbf{a}}$ is the preimage of $H_{\mathbf{a}}$ under the reduction map.
Therefore, the index of $\Gamma_{\mathbf{a}}$ in $\SL_4(\Z)$ is the index of $H_{\mathbf{a}}$ in $\SL_4(\rquot{\Z}{N\Z})$.
Notice that $\SL_4(\rquot{\Z}{N\Z})$ acts transitively on $\varprim(\rquot{\Z}{N\Z})$.
Indeed, this follows from the Chinese remainder theorem in \eqref{eq:chinese remainder} and its analogue for $\SL_4$ as well as the fact that $\SL_4(\Fp)$ acts transitively on $\purewedges(\Fp)\setminus\{0\}$ for any odd prime~$p$.
Therefore, $\varprim(\rquot{\Z}{N\Z}) = \rquot{\SL_4(\rquot{\Z}{N\Z})}{H_{\mathbf{a}}}$ which implies the latter claim in the lemma.

To prove transitivity of the action of $\Gamma_{\mathbf{a}}$, let $w_1,w_2 \in \purewedges_{\mathbf{a}}(\Z)$ and choose some $\gamma\in\SL_4(\Z)$ with $\gamma.w_1 = w_2$. 
But then $\overline{\gamma}. \mathbf{a} = \overline{\gamma.w_1} = \overline{w_2} = \mathbf{a}$ and so~$\gamma\in \Gamma_{\mathbf{a}}$.
\end{proof}

\begin{proof}[Proof of Proposition \ref{prop:counting in fibers of reduction}]
As mentioned, we use the technique in \cite{eskinmcmullen} (see also \cite{dukerudnicksarnak}).
We begin by recalling the necessary dynamical statement.
Fix some $w\in \purewedges_{\mathbf{a}}(\Z)$.
Note that $\Gamma_\mathbf{a}\cap H_w = H_w(\Z)$ is a lattice in $H_w$ (since $H_w$ has no non-trivial $\Q$-characters).
Let $m_\varwedge$ be a non-trivial measure on $\varwedge$ invariant under $G = \SL_4(\R)$ and assume (after rescaling of the Haar measure $m_{H_w}$ on $H_w$) that for any $f \in C_c(G)$ we have 
\begin{align*}
\int_{G}f \de m_G = \int_{\varwedge} \int_{H_w} f(gh) \de m_{H_w}(h) \de m_\varwedge(gH_w).
\end{align*}
Since $H_w(\Z) = \Gamma_{\mathbf{a}}\cap H_w$ we may set
\begin{align*}
C := \vol\Big(\rquot{H_w}{H_w(\Z)}\Big) = \vol\Big(\rquot{H_w}{\Gamma_{\mathbf{a}} \cap H_w}\Big).
\end{align*}

To simplify notation we substitute $r = \sqrt{D}$ and write $B_r = B_r(0)$.
The balls $B_r$ are well-rounded in the sense of \cite{eskinmcmullen}.
We have the following mixing statement on average for $f \in C_c(G/\Gamma_{\mathbf{a}})$
\begin{align*}
\frac{1}{C}\frac{1}{m_\varwedge(B_r)} \int_{B_r}\int_{H_w\Gamma_{\mathbf{a}}}
f(gh\Gamma_{\mathbf{a}}) 
&\de m_{\rquot{H_w}{\Gamma_{\mathbf{a}}\cap H_w}}(h(\Gamma_{\mathbf{a}}\cap H_w)) \de m_\varwedge(gH_w) \\
&\to \frac{1}{\vol(\rquot{G}{\Gamma_{\mathbf{a}}})}\int_{G/\Gamma_{\mathbf{a}}}f.
\end{align*}
Then \cite[Thm.~1.4]{eskinmcmullen} implies
\begin{align*}
|\purewedges_{\mathbf{a}}(\Z) \cap B_r|
=|\Gamma_{\mathbf{a}}.w \cap B_r| 
&\asymp \frac{C}{\vol(\rquot{G}{\Gamma_{\mathbf{a}}})} m_\varwedge(B_r) \\
&= \frac{C}{[\SL_4(\Z):\Gamma_{\mathbf{a}}] \vol(\rquot{G}{\SL_4(\Z)})}m_\varwedge(B_r).
\end{align*}
By the analogous argument using the whole lattice $\SL_4(\Z)$ instead of the congruence subgroup $\Gamma_{\mathbf{a}}$ we have
\begin{align*}
|\varprim(\Z) \cap B_r|
=|\SL_4(\Z).w \cap B_r| 
&\asymp \frac{C}{\vol(\rquot{G}{\SL_4(\Z)})} m_\varwedge(B_r).
\end{align*}
Since $|\varprim(\Z) \cap B_r| = 2|\mathcal{R}_{<D}|$ 
and $[\SL_4(\Z):\Gamma_{\mathbf{a}}] = |\varprim(\rquot{\Z}{N\Z})|$ this proves that
\begin{align*}
|\purewedges_{\mathbf{a}}(\Z) \cap B_r| 
&\asymp \frac{1}{|\varprim(\rquot{\Z}{N\Z})|} \frac{C}{\vol(\rquot{G}{\SL_4(\Z)})} m_\varwedge(B_r)\\
&\asymp \frac{2}{|\varprim(\rquot{\Z}{N\Z})|} |\mathcal{R}_{<D}|
\end{align*}
as desired.
\end{proof}

\subsection{Counting representations by the discriminant}

To prove Proposition \ref{prop:elm counting variety mod N} we will use the following auxiliary lemma which can be found in greater generality in \cite[Lemma 1.3.1 and Thm.~1.3.2]{kitaoka}.

\begin{lemma}[Counting solutions to quadratic equations]\label{lemma:sol quad equation}
Let $p$ be an odd prime and let $\alpha \in \Fp$.
The number $r_p(\alpha)$ of solutions to $x_1^2+x_2^2+x_3^2 = \alpha$ over~$\Fp$ satisfies $|r_p(\alpha)-p^2| \ll p$ where the implicit constant is independent of $\alpha$ and $p$.
Also, $r_p(u^2\alpha) = r_p(\alpha)$ for all $u \in \Fp^\times$.
\end{lemma}

\begin{proof}
Since $x_1^2+x_2^2+x_3^2$ is isotropic, it is equivalent to $x_1x_2-x_3^2$ (by discriminant comparison) and so $r_p(\alpha)$ is equal to the number of solutions to $x_1x_2 = \alpha+ x_3^2$.
We let $S = \{\alpha+x_3^2: x_3 \in \Fp\}$ and note that any $s \in S$ is represented by exactly two values $x_3$ except for $\alpha$.
Also, $|S| = \tfrac{p+1}{2}$. 
For any non-zero $s\in S$ the number of solutions to $x_1x_2 = s$ is equal to the number of solutions to $x_1x_2 = 1$ which is $|\Fp^\times|= p-1$.
Furthermore, the number of solutions to $x_1x_2 =0$ is $2p-1$.
We now distinguish two cases.

Case 1: $0 \not\in S$ (i.e.~$-\alpha$ is not a square).
Then $\alpha \neq 0$ and
\begin{align*}
r_p(\alpha) = (p-1) 2 (|S|-1) + (p-1)
= p^2-p.
\end{align*}

Case 2: $0 \in S$. If $\alpha = 0$ then
\begin{align*}
r_p(0) = (p-1) 2 (|S|-1) + (2p-1) = p^2.
\end{align*}
Otherwise,
\begin{align*}
r_p(\alpha) = (p-1) 2 (|S|-2) + (p-1) + (2p-1)2
= p^2+p
\end{align*}
which concludes the proof.
\end{proof}

\begin{proof}[Proof of Proposition \ref{prop:elm counting variety mod N}]
Let us first assume that $N = p$ is an odd prime and let us begin by counting non-zero (i.e.~primitive) points in $\purewedges(\Fp)$ of discriminant $\alpha\in \Fp$.
Choosing the standard basis of $\bigwedge^2\Fp^4$ the quadratic form $\disc$ is represented by $x_{12}^2+ \ldots +x_{34}^2$.
Furthermore, $\purewedges(\Fp)$ is the set of solutions to $x_{12}x_{34}-x_{13}x_{24}+x_{14}x_{23} = 0$.
By adding and subtracting one sees that the system of equations
\begin{align*}
\begin{cases}
x_{12}^2+ \ldots +x_{34}^2 = \alpha \\
x_{12}x_{34}-x_{13}x_{24}+x_{14}x_{23} = 0
\end{cases}
\end{align*}
is equivalent to the decoupled system
\begin{align*}
\begin{cases}
y_1^2+y_2^2+y_3^2 = \alpha \\
y_4^2+y_5^2+y_6^2 = \alpha
\end{cases}
\end{align*}
where $y_1 = x_{12}-x_{34}$, $y_2 = x_{13}+x_{24}$, $y_3 = x_{14}-x_{23}$, $y_4 = x_{12}+x_{34}$, $y_5 = x_{13}-x_{24}$ and $y_6 = x_{14}+x_{23}$.
If $\alpha\neq 0$, the number of (non-zero) solutions to the latter system is equal to $r_p(\alpha)^2$ and so 
\begin{align}\label{eq:no of rep at every prime}
|\{\mathbf{a} \in \varprim(\Fp): \disc(\mathbf{a}) = \alpha\}| 
= r_p(\alpha)^2.
\end{align}
If $\alpha = 0$, the number of non-zero solutions is $r_p(0)^2-1$.

Fix $\alpha_1,\alpha_2 \in \Fp^\times$ with $-\alpha_1 \in (\Fp^\times)^2$ and $-\alpha_2 \not\in (\Fp^\times)^2$.
We now apply \eqref{eq:no of rep at every prime} to estimate
\begin{align*}
\frac{1}{|\varprim(\Fp)|}&
\big|\{\mathbf{a} \in \varprim(\Fp): -\disc(\mathbf{a}) \not\in (\Fp^\times)^2 \}\big|\\
&=\frac{(r_p(0)^2-1) + |(\Fp^\times)^2|r_p(\alpha_2)^2}{(r_p(0)^2-1) + r_p(\alpha_1)^2|(\Fp^\times)^2| + r_p(\alpha_2)^2 |(\Fp^\times)^2|} \\
&\leq \frac{r_p(0)^2}{r_p(\alpha_1)^2|(\Fp^\times)^2|}
+ \frac{r_p(\alpha_2)^2}{r_p(\alpha_1)^2+r_p(\alpha_2)^2}= \frac{\big(\frac{r_p(0)}{r_p(\alpha_1)}\big)^2}{|(\Fp^\times)^2|} +  \frac{1}{1+\big(\frac{r_p(\alpha_1)}{r_p(\alpha_2)}\big)^2}
\end{align*}
By Lemma \ref{lemma:sol quad equation}, $\frac{r_p(x_1)}{r_p(x_2)}$ converges to $1$ as $p$ goes to infinity uniformly in $x_1,x_2 \in \Fp$ and so we have
\begin{align}\label{eq:prop non-sq estimate 2}
\frac{|\{\mathbf{a} \in \varprim(\Fp): -\disc(\mathbf{a}) \in (\Fp^\times)^2 \}|}{|\varprim(\Fp)|} 
\leq \frac{1}{15} + \frac{1}{1+\frac{2}{3}} = \frac{2}{3}
\end{align}
for all but finitely many odd primes $p$.

Now let $N$ be an arbitrary odd and square-free number and let $N = p_1 \ldots p_m$ be its prime decomposition.
Let $M_k$ be the number of vectors $\mathbf{a} \in \varprim(\rquot{\Z}{N\Z})$ for which $-\disc(\mathbf{a})$ is a non-zero square only modulo $p_k$.
Then by the application of the Chinese remainder theorem in \eqref{eq:chinese remainder} and the estimate \eqref{eq:prop non-sq estimate 2}
\begin{align*}
\frac{M_k}{|\varprim(\rquot{\Z}{N\Z})|} \leq \Big(\frac{2}{3}\Big)^{m-1-m_0} \ll \Big(\frac{2}{3}\Big)^{m}
\end{align*}
where $m_0$ is the number of exceptions to \eqref{eq:prop non-sq estimate 2}.
Similarly, if $M_0$ is the number of vectors $\mathbf{a} \in \varprim(\rquot{\Z}{N\Z})$ for which $-\disc(\mathbf{a})$ is not a non-zero square modulo any $p_k$, we have $\frac{M_0}{|\varprim(\rquot{\Z}{N\Z})|} \ll (\frac{2}{3})^m$. This proves the proposition.
\end{proof}

\section{Class groups and Theorem \ref{thm:arithmeticversion}}\label{section:appendixA}

In this section we would like to explain the relationship between Theorem \ref{thm:main} and Theorem~\ref{thm:arithmeticversion} by illustrating it in a special case.
%
This will also give more intuition on the $45^\circ$-twist discussed in Lemma~\ref{lemma:pointwisestab}.

\subsection{Planes in the split quaternion algebra}
We consider the quaternion algebra $\quat = \Mat_{2}$ and the group $\SL_2$ of norm one units (see also Section \ref{section:local for forms}).
Here, recall that the conjugation on $\Mat_{2}$ is given by the adjunct
\begin{align*}
\begin{pmatrix}
a & b \\ 
c & d
\end{pmatrix}^{\ad}
= \begin{pmatrix}
d & -b \\ 
-c & a
\end{pmatrix}.
\end{align*}
The norm is the determinant
\begin{align*}
Q(a,b,c,d) = 
\begin{pmatrix}
a & b \\ 
c & d
\end{pmatrix}\begin{pmatrix}
a & b \\ 
c & d
\end{pmatrix}^{\ad} = (ad-bc) \id
\end{align*}
and the trace is the usual trace
\begin{align*}
\begin{pmatrix}
a & b \\ 
c & d
\end{pmatrix}+\begin{pmatrix}
a & b \\ 
c & d
\end{pmatrix}^{\ad} = \begin{pmatrix}
a+d & 0 \\ 
0 & d+a
\end{pmatrix}
= (a+d) \id.
\end{align*}
As before, we let $\SL_2^2$ act on $\Mat_{2,2}$ and $\SL_2$ act on the traceless matrices $\quat_0~=~\mathfrak{sl}_2$.
The formula \eqref{eq:associated intpts} as well as Proposition \ref{prop:splitting map} on the Klein map can directly be generalized to this setup and so one can identify two-dimensional subspaces $L \subset \Mat_{2}$ with equivalence class of pairs $[(a_1(L),a_2(L))]$ where $a_1(L),a_2(L) \in \quat_0$ satisfy $\det(a_1(L)) = \det(a_2(L))$.

Note that $\Mat_{2,2}$ carries an integral structure given by $\Mat_{2,2}(\Z)$.
The analogue of formula \eqref{eq:associated intpts} does not directly yield integral matrices (the trace is not automatically divisible by $2$) which is why we multiply the defining expression by~$2$.

\subsection{The acting tori in a special case}
For the purposes of this subsection we would like to consider the plane
\begin{align*}
L = \langle E_{11}, E_{22}  \rangle
\end{align*}
where $E_{ij}$ denotes the matrix which is one at the $(i,j)$-th entry and zero otherwise. An integral basis of $L(\Z) = L \cap \Mat_{2,2}(\Z)$ is then given by $E_{11},E_{22}$.
So the integer points associated $L$ (see also Section \ref{section:associated intpts}) are given by
\begin{align*}
a_1(L) 
&= 2E_{11}E_{22}^{\ad} - \Tr(E_{11}E_{22}^{\ad})
= 2E_{11} - \Tr(E_{11}) = 2E_{11}-\id\\ 
&= E_{11}-E_{22} = \begin{pmatrix}
1 & 0 \\ 
0 & -1
\end{pmatrix} = a_2(L).
\end{align*}
The analogue of Proposition \ref{prop:splitting map} thus yields
\begin{align*}
\Stab_{\SL_2^2}(L) = \Stab_{\SL_2}(a_1(L)) \times \Stab_{\SL_2}(a_2(L))
= A \times A
\end{align*}
where $A = \set{a_s = \diag(s,s^{-1})}$
denotes the subgroup of diagonal matrices.
A direct computation provides the pointwise stabilizers
\begin{align*}
\Stab^{\mathrm{pt}}_{\SL_2^2}(L) &= \setc{(a,a)}{a \in A} \\
\Stab^{\mathrm{pt}}_{\SL_2^2}(L^\perp) &= \setc{(a,a^{-1})}{a \in A}
\end{align*}
in analogy to Lemma \ref{lemma:pointwisestab}.
We now fix ourselves an element $(a_s,a_t) \in A^2$ and examine the way it acts on all relevant subspaces (see Section \ref{section:acting group}).
\begin{itemize}
\item The action of $(a_s,a_t)$ on the subspace $L$ is represented by $a_{st^{-1}} \in A$ in the integral basis $E_{11}$, $E_{22}$ as
\begin{align*}
a_s E_{11} a_t^{-1} = st^{-1}E_{11}, \quad 
a_s E_{22} a_t^{-1} = s^{-1}t E_{22}.
\end{align*}
Note that the restriction of $Q$ to $L$ represented in the basis $E_{11}$, $E_{22}$ is exactly the binary form $q_L(x,y) = xy$ so that $A = \Stab_{\SL_2}(q_L)$.
In other words, the homomorphism $\Psi_{1,L}$ defined in analogy to Section \ref{section:acting group} is given by
\begin{align*}
\Psi_{1,L}: (a_s,a_t) \in A \times A \mapsto a_{st^{-1}} \in A.
\end{align*}
\item  We proceed similarly for $L^\perp$ for which we consider the integral basis given by $E_{12},E_{21}$. Then
\begin{align*}
a_s E_{12} a_t^{-1} = s E_{12} a_{t}^{-1} =  stE_{11}, \quad 
a_s E_{21} a_t^{-1} = (st)^{-1} E_{22}.
\end{align*}
Therefore,
\begin{align*}
\Psi_{2,L}: (a_s,a_t) \in A \times A \mapsto a_{st} \in A.
\end{align*}
\item The orthogonal complement $a_1(L)$ inside $\quat_0$ is given by $L^\perp$, for which we again choose the integral basis $E_{12},E_{21}$. Then
\begin{align*}
\Psi_{3,L}: (a_s,a_t) \in A \times A \mapsto a_{s^2} \in A.
\end{align*}
This is because the action of $(a_s,a_t)$ on $a_1(L)^\perp \subset \quat_0$ is the conjugation with $a_s$ (by the analogon of Proposition \ref{prop:integrality for splitting map}) so that one can apply the previous calculation with $s =t$.
\item Since $a_1(L) = a_2(L)$ one analogously obtains
\begin{align*}
\Psi_{4,L}: (a_s,a_t) \in A \times A \mapsto a_{t^2} \in A.
\end{align*}
\end{itemize}

\end{appendix}

\bibliographystyle{amsalpha}
\bibliography{Bibliography}

\end{document}
